\documentclass[reqno,11pt]{amsart}
\usepackage{amsfonts}
\usepackage{graphicx}
\usepackage{amsfonts,amsmath, amssymb}
\usepackage{hyperref}
\usepackage{marginnote}
\usepackage{color}
\usepackage{dsfont}
\usepackage{xcolor}
\usepackage{mathtools}
\usepackage{enumitem}
\usepackage{leftindex}

\usepackage[
backend=biber,
style=trad-abbrv,doi=false,isbn=false,url=false,eprint=false
]{biblatex}

\numberwithin{equation}{section}

\newcommand{\R}{\mathbb{R}}

\newcommand{\midbar}[1]{
  \mkern7mu
  \overline{\mkern-7mu #1\mkern-1.5mu}
  \mkern1.5mu
}

\newtheorem{theorem}{Theorem}[section]
\newtheorem{corollary}[theorem]{Corollary}
\newtheorem{lemma}[theorem]{Lemma}

\newtheorem{remark}[theorem]{Remark}
\theoremstyle{definition}
\newtheorem{definition}[theorem]{Definition}
\newtheorem{step}{Step}
\newtheorem{case}{Case}

\def\v{\varepsilon}

\def\d{{\rm d}}

\def\dd{{\, \rm d}}
\def\M{{\mathcal{M}}}

\def\supp{{\rm supp}}
\newcommand{\longrightharpoonup}{\rm - \!\!\! \rightharpoonup}
\usepackage{tikz}

\begin{document}

\title[Nonlinear Stability and Instability of rotating Riesz star solutions]{Nonlinear stability and instability of rotating Riesz star solutions of the compressible Euler-Riesz equations}

\author[S.R. Charles]{Samuel R. Charles}
\address{S.R. Charles:\, Department of Mathematics, National University of Singapore, Singapore 119076, Singapore}
\email{samuel.charles@nus.edu.sg}

\begin{abstract}
    The compressible Euler–Riesz equations arise in the modelling of a wide range of physical phenomena, including stellar dynamics, plasma physics, and mathematical biology. We study rotating steady states of the attractive compressible Euler–Riesz equations, which we call rotating Riesz stars, and establish existence and nonlinear stability or instability results in both the mass-subcritical and mass-supercritical regimes. In the mass-subcritical regime, we prove the existence of rotating Riesz stars under suitable subhomogeneity assumptions on the angular momentum profile. We then establish their nonlinear stability through a concentration compactness argument adapted to the axisymmetric setting. Rotation creates new compactness difficulties, most notably the possibility that minimising sequences are tight along rings whose radii diverge to infinity. In the mass-supercritical regime, we prove existence in the polytropic setting under suitable superhomogeneity assumptions on the angular momentum profile, thereby extending the theory beyond the small angular velocity regime. The proof requires a careful analysis of mass-preserving scalings, which are more delicate than in the non-rotating case. Finally, by analysing the concavity of the free-energy along these scalings, we establish the instability of the resulting mass-supercritical rotating Riesz stars. Our results show that rotation can have either a stabilising or a destabilising effect, depending on the singularity of the Riesz interaction.
\end{abstract}

\subjclass[2020]{35Q35, 35Q31, 35B35, 35A15, 35R09, 35L65, 76N10, 35B38, 35B33}
\keywords{Compressible Euler-Riesz equations; rotating Riesz stars; nonlinear stability; nonlinear instability; Riesz potentials; non-local interaction; concentration compactness; variational methods.}
\date{\today}

\maketitle

\tableofcontents

\thispagestyle{empty}

\section{Introduction}

For a compressible fluid, the attractive three-dimensional (3-D) compressible Euler–Riesz equations form a system of four partial differential equations given by
\begin{align}\label{0.0}
	\begin{cases}
            \displaystyle
		\partial_t \rho+\nabla \cdot \boldsymbol{\M}=0, \\[1mm]
            \displaystyle
		\partial_t \boldsymbol{\M}+ \nabla \cdot \Big(\frac{\boldsymbol{\M}\otimes\boldsymbol{\M}}{\rho}\Big)
        +\nabla p + \rho \nabla \mathcal{W}_\alpha = \boldsymbol{0},
	\end{cases}
 \end{align}
 where, for each point in space-time $(t,\boldsymbol{x}) \in [0,\infty) \times \mathbb{R}^3$, the function $\rho: [0,\infty) \times \mathbb{R}^3 \to [0,\infty)$ denotes the fluid density, $\boldsymbol{\mathcal{M}}: [0,\infty) \times \mathbb{R}^3 \to \mathbb{R}^3$ the momentum, $p: [0,\infty) \times \mathbb{R}^3 \to \mathbb{R}$ the pressure, and $\mathcal{W}_\alpha : [0,\infty) \times \mathbb{R}^3 \to \mathbb{R}$ the Riesz potential for $\alpha \in (0,3)$. The Riesz potential describes a non-local interaction between fluid particles and is defined through the convolution $\mathcal{W}_\alpha = \Phi_\alpha * \rho$, where $\Phi_\alpha : \mathbb{R}^3 \to \mathbb{R}$ denotes the Riesz kernel associated with $\alpha \in (0,3)$, given by
 \[
 \Phi_\alpha(\boldsymbol{x}) = -\frac{|\boldsymbol{x}|^{-\alpha}}{\alpha}.
 \]
 When $\alpha \in (0,3)$, the kernel $\Phi_\alpha$ is integrable near the origin and can therefore be analysed using mean-field and potential-theoretic methods. In contrast, for $\alpha \geq 3$, the kernel $\Phi_\alpha$ fails to be integrable at the origin, preventing the direct application of such techniques. In this regime, the kernel is referred to as hypersingular and behaves more like a short-range interaction; see \cite{Hardin_2018}. Accordingly, we restrict our attention to the range $\alpha \in (0,3)$, which includes the important Newtonian case $\alpha = 1$, for which the Riesz kernel coincides, up to a multiplicative constant, with the fundamental solution of the Poisson equation. Moreover, in the Manev or super-Manev regime $\alpha \geq 2$, the gradient $\nabla \mathcal{W}_\alpha$ cannot be transferred onto the Riesz kernel in a distributional sense. To properly define $\nabla \mathcal{W}_\alpha$ in this setting, we introduce
\[
 \nabla \mathcal{W}_\alpha(\boldsymbol{x}) :=
 \begin{cases}
    \displaystyle
     (\nabla \Phi_\alpha * \rho) (\boldsymbol{x}) & \text{for } \alpha \in (0,2), \\[1mm]
     \displaystyle
     \int_{\mathbb{R}^3} \nabla \Phi_\alpha(\boldsymbol{x} - \boldsymbol{y}) (\rho(\boldsymbol{y}) - \rho(\boldsymbol{x})) \dd \boldsymbol{y} \quad & \text{for } \alpha \in [2,3).
 \end{cases}
 \]
Thus, for $\alpha \in [2,3)$, if $\rho \in C^{0,\beta}(\mathbb{R}^3)$ for some $\beta \in (\alpha-2,1)$, the difference $\rho(\boldsymbol{y})-\rho(\boldsymbol{x})$ compensates for the singularity of $\nabla \Phi_\alpha$, ensuring that $\nabla \mathcal{W}_\alpha$ is well-defined.

\smallskip

 In the super-Newtonian regime, namely for $\alpha \in (1,3)$, the potential $\mathcal{W}_\alpha$ is known to be proportional to the inverse fractional Laplacian operator:
\[
    (-\Delta)^\frac{3-\alpha}{2} \mathcal{W}_\alpha = - c_{\alpha} \rho,
\]
in the sense of distributions, where
\[
c_{\alpha} = \frac{2^{3-\alpha} \pi^\frac{3}{2} \Gamma\left (\frac{3-\alpha}{2} \right )}{\alpha \Gamma\left (\frac{\alpha}{2} \right )},
\]
see \cite{Gelfand_1958} for the details of the calculation. The fractional Laplacian is a non-local operator that may be defined either through Fourier multipliers or via a singular integral representation in physical space. Serfaty \cite{Serfaty_2020} first showed that the pressure-less Euler equations of the form \eqref{0.0}, in the repulsive regime with Newtonian or super-Newtonian Riesz potentials, arise as mean-field limits of interacting particle systems. This result was later extended by Nguyen–Rosenzweig–Serfaty \cite{Nguyen_2022} to the full range $\alpha \in (0,3)$ through the introduction of a modulated energy method; see also Serfaty \cite{Serfaty:2024ojp}.

\smallskip

Riesz gases arise in a broad variety of physical contexts, including plasma physics, stellar and galactic dynamics, solid-state physics, ferrofluids, elasticity theory, and random matrix theory; see \cite{Campa_2014}. Riesz interactions have also been studied in the contexts of jellium and the uniform electron gas; see \cite{Lewin_2017,Cotar_2019}. For a recent survey, including numerous open problems, we refer to \cite{Lewin_2022}. Further applications occur in mathematical biology, particularly in models of collective behaviour and swarming; see \cite{Carrillo_2021,Carrillo_2017}.

\smallskip

When $\alpha = 1$, the interaction potential reduces to the Newtonian potential. The associated Newtonian gas serves as a fundamental toy model for classical matter in the absence of quantum effects. For example, Gamow’s liquid drop model provides a simplified description of atomic nuclei, electrons, and atoms; see \cite{Serfaty:2024ojp}. Purely attractive Newtonian interactions arise naturally in gravitational systems, astrophysical Vlasov–Poisson models, Keller–Segel chemotaxis, aggregation equations, and bosonic mean-field models such as boson-star and Hartree equations; see, for example, \cite{Rein2007,BinneyTremaine2008,KellerSegel1970,BlanchetDolbeaultPerthame2006,BertozziLaurent2007,BertozziLaurentLeger2012,LiebYau1987,ElgartSchlein2007,FrohlichJonssonLenzmann2007}. When $\alpha = 2$, the interaction is known as the Manev potential. This potential arises as a relativistic correction to the Newtonian potential and has been used to explain phenomena such as the precession of Mercury; see \cite{Illner}.

\smallskip

In this paper, we assume that the fluid is barotropic, meaning that the pressure depends only on the density, namely $p = p(\rho)$. A commonly used equation of state in the study of physical systems is that of a power-law for a polytropic fluid, for which the pressure is given by
\[
 p = p(\rho) = a_0 \rho^\gamma,
 \]
where $\gamma \geq 1$ denotes the adiabatic exponent. The case $\gamma = 1$ is referred to as the isothermal regime. Moreover, for $\gamma > 1$, a scaling argument allows us to normalise the constant to $a_0 = \frac{(\gamma-1)^2}{4 \gamma} > 0$. The exponent $\gamma = \frac{3+\alpha}{3}$ is known as the mass-critical exponent, arising from the balance between the internal and potential energies. Another important exponent is $\gamma = \frac{6}{6-\alpha}$, referred to as the energy-critical exponent, which is associated with the scaling of the potential energy. In the classical Newtonian case $\alpha = 1$, the range $\frac{6}{5} < \gamma < \frac{4}{3}$ is known as the mass-supercritical regime, while $\gamma > \frac{4}{3}$ corresponds to the mass-subcritical regime. In the context of the 3-D CEREs, these regimes generalise to $\frac{6}{6-\alpha} < \gamma < \frac{3+\alpha}{3}$ for the mass-supercritical regime, and $\gamma > \frac{3+\alpha}{3}$ for the mass-subcritical regime. A power-law pressure provides a good model for gases in the low-density regime. However, at high densities, the pressure may exhibit substantially different behaviour. For example, in the case of a white dwarf star with $\alpha = 1$, the pressure is given by
\begin{equation}\label{general pressure}
    p(\rho) = A \int^{B\rho^\frac{1}{3}}_0 \frac{\sigma^4}{\sqrt{D + \sigma^2}} \dd \sigma,
\end{equation}
for $\rho > 0$, where $A, B, D > 0$ are constants. The pressure law \eqref{general pressure} exhibits the following asymptotic behaviour in the low-density and high-density regimes:
\[
    \begin{cases}
        p(\rho) \sim \frac{A B^5 }{5\sqrt{D}}\rho^\frac{5}{3}, & \text{ as } \rho \to 0, \\
        p(\rho) \sim \frac{A B^4 }{4}\rho^\frac{4}{3}, & \text{ as } \rho \to \infty.
    \end{cases}
\]
Accordingly, it is natural to consider a more general pressure law $p : [0,\infty) \to [0,\infty)$ with $p \in C^1([0,\infty))$, satisfying the condition that
\begin{equation}\label{Press}
    p'(\rho) > 0,
\end{equation}
for $\rho > 0$.

\smallskip

Throughout the paper, we consider the Cauchy problem for \eqref{0.0} with the initial data 
$\rho_0 : \mathbb{R}^3 \to [0,\infty)$ and $\boldsymbol{\M}_0 : \mathbb{R}^3 \to \mathbb{R}^3$ 
that satisfy
 \begin{equation}\label{0.1}
     (\rho,\boldsymbol{\M})|_{t=0}
     = (\rho_0,\boldsymbol{\M}_0)(\boldsymbol{x}) \longrightarrow (0,\boldsymbol{0}),
     \qquad \mbox{ as $|\boldsymbol{x}| \to \infty$}.
 \end{equation}
 Since global solutions of the CEREs system may include vacuum regions
 $$\{ (t,\boldsymbol{x}) \, | \, \rho(t,\boldsymbol{x}) = 0 \},$$ where the fluid velocity is not well-defined, we instead formulate the problem in terms of the momentum variable $\boldsymbol{\M}(t,\boldsymbol{x})$. Unlike the velocity field $\boldsymbol{u}(t,\boldsymbol{x})$, the momentum remains globally well-defined and therefore provides a more robust framework for analysis. Under these assumptions, the energy associated with the CEREs is given by
\begin{equation*}
     E(\rho,\boldsymbol{\M})(t): 
     =\int_{\mathbb{R}^3}\Big(\rho(t, \boldsymbol{x}) e(\rho(t, \boldsymbol{x})) 
     + \frac{1}{2} \bigg|\frac{\boldsymbol{\M}}{\sqrt{\rho}}\bigg|^2(t, \boldsymbol{x}) 
     + \frac{1}{2}\rho(\boldsymbol{x})(\Phi_\alpha \ast \rho)(t, \boldsymbol{x})\Big)
     \dd\boldsymbol{x},
\end{equation*}
where
\[
    e(\rho) = \int^\rho_0 \frac{p(\tau)}{\tau^2} \dd \tau,
\]
is the internal energy density. Here, the first term represents the internal energy of the fluid, the second corresponds to the kinetic energy, and the final term describes the potential energy of the system.

\smallskip

Most well-posedness results for the $n$-D CEREs have so far been obtained only in the classical setting of the compressible Euler–Poisson equations (CEPEs). In the Newtonian case, the interaction potential is determined through the Poisson equation and therefore enjoys a local structure. By contrast, for general Riesz potentials this locality property is lost, and the resulting nonlocality introduces substantial analytical difficulties when extending the theory from the CEPEs to the CEREs. Moreover, the nonlocal interaction gives rise to new phenomena that do not appear in the classical CEPEs framework, necessitating the development of new analytical techniques. Only recently have preliminary results for the Euler–Riesz equations begun to emerge. Choi–Jeong \cite{Choi_2022a} established local well-posedness away from vacuum for both attractive and repulsive interactions in the super-Newtonian regime $\alpha \in (n-2,n)$, and further showed that classical solutions may break down in finite time under suitable conditions. Danchin–Ducomet \cite{Danchin_2022} proved the existence of global classical solutions for $\alpha \in [n-2,n-1]$ under dispersive spectral assumptions on the initial velocity field. Choi–Jung \cite{Choi_2022b} considered the pressureless damped system and proved the existence of unique global classical solutions for sufficiently small initial data in suitable Sobolev spaces in the attractive super-Newtonian regime $\alpha \in (n-2,n)$. This was subsequently extended by Choi–Jung–Lee \cite{Choi_2024} to non-damped polytropic flows in both the attractive and repulsive settings within the same super-Newtonian range. Whereas the above works focus primarily on the Newtonian and super-Newtonian regimes, Carrillo–Charles–Chen–Yuan \cite{Carrillo2025} proved the existence of global finite-energy weak solutions to the polytropic CEREs for $\alpha \in (-1,n-1)$. Here, the case $\alpha = 0$ corresponds to the logarithmic kernel $\Phi_0(\boldsymbol{x}) = \log |\boldsymbol{x}|$, which coincides with the Newtonian potential in two dimensions. These results were later extended in \cite{Carrillo_2026} to a class of general pressure laws exhibiting asymptotic power-law behaviour both near vacuum and at infinity. Finally, for the Manev potential $\alpha = n-1$, we refer to \cite{Bobylev,Cercignani,Illner} for related work on the Vlasov–Manev system, which shares many structural and dynamical similarities with the CEREs.

\smallskip

A fundamental problem in the study of the CEREs is to understand the behaviour of solutions near special classes of states, in particular steady states, namely solutions that are independent of time. In the attractive 3-D CEPEs setting, such steady states may be interpreted physically as equilibrium stellar configurations. It is therefore of considerable interest to investigate their nonlinear stability. More precisely, one asks whether a configuration that initially starts sufficiently close to a steady state remains close for all later times, or possibly converges to equilibrium. If this occurs, the steady state is said to be nonlinearly stable; conversely, it is called nonlinearly unstable if there exist solutions starting arbitrarily close to the steady state that eventually depart significantly from it. This question was first addressed by Rein \cite{Rein_2003}, who proved that minimisers of the free-energy functional
\begin{equation*}
\varrho
\mapsto \int_{\mathbb{R}^3} \Big(\varrho(\boldsymbol{x}) e(\varrho(\boldsymbol{x}))
+\frac{1}{2}\varrho(\boldsymbol{x})(\Phi_1\ast \varrho)(\boldsymbol{x})\Big)\,
\dd\boldsymbol{x},
\end{equation*}
are stationary states, steady states with vanishing velocity field, of the attractive 3-D CEPEs and have compact support. Moreover, under suitable assumptions on a general pressure law, corresponding to $\gamma > \frac{4}{3}$ in the polytropic case, these minimisers were shown to be nonlinearly stable. Such stationary states may be interpreted physically as non-rotating stars. The nonlinear stability of non-rotating stars has been studied extensively. Lin--Zeng \cite{Lin_2022} showed that a turning point principle can be applied to the linear stability of non-rotating stars. Based on a linear stability criterion, Lin--Wang--Zhu \cite{Lin_2023} obtained an unconditional stability result for the attractive CEPEs with a general pressure law. On the instability side, Deng--Liu--Yang--Yao \cite{Deng_2002} proved the instability of non-rotating stars for the three-dimensional polytropic CEPEs with $\gamma = \frac{4}{3}$, by showing that solutions with positive initial energy exhibit growth of their support; since stationary states have compact support, this implies instability. Jang \cite{Jang2008,Jang2014} proved nonlinear instability of Lane--Emden stars for $\gamma \in [\frac{6}{5},\frac{4}{3})$. More recently, Cheng--Cheng--Lin \cite{Cheng_2025} established nonlinear instability for non-rotating polytropic stars with $\gamma \in (\frac{6}{5},\frac{4}{3}]$, by exploiting the concavity of the free-energy with respect to the mass-preserving scaling parameter and using this to prove growth of the support of solutions. It is worth noting that, by considering the associated free-energy functional, stationary states of the CEREs can also be viewed as stationary states of aggregation--diffusion equations with a Riesz potential. In this context, results such as those of Carrillo--Hittmeir--Volzone--Yao \cite{Carrillo_2019}, who proved radial symmetry of stationary states for a broad class of Keller--Segel-type aggregation--diffusion equations, are also relevant to the CEREs; see also \cite{CCY19} for further properties of such equations. This observation was exploited in the works of Carrillo--Charles--Chen--Yuan \cite{Carrillo2025,Carrillo_2026}, where the authors established the existence of stationary states for the attractive CEREs with general pressure laws. Furthermore, they proved that these stationary states are nonlinearly stable, obtained quantitative estimates on their finite-time stability, and showed that the stability result is inherently local in nature.

\smallskip

Another class of steady states of considerable interest is that of rotating star solutions to the 3-D CEPEs, for which the velocity field is axisymmetric about the $z$-axis. The existence of such rotating stars was first established by Auchmuty--Beals \cite{Auchmuty_1971}, who constructed them as minimisers of a free-energy functional over a class of axisymmetric density distributions and showed that these minimisers satisfy the corresponding steady state equations. Subsequent works further developed the associated variational framework and investigated structural properties of rotating stars under various assumptions on the equation of state and admissible classes; see Caffarelli--Friedman \cite{Caffarelli_1980}, Friedman--Turkington \cite{Friedman_1980,Friedman_1981}, Auchmuty \cite{Auchmuty_1991}, Li \cite{Li_1991}, Chanillo--Yan \cite{Chanillo_1994}, Luo--Smoller \cite{Luo_2004}, McCann \cite{McCann_2006}, and Wu \cite{Wu_2016,Wu_2015}. Lions \cite{Lions_1981} proved the existence of rotating star solutions under a prescribed axisymmetric potential using variational methods. More recently, Strauss--Wu \cite{Strauss_2019,Strauss_2017} established existence results for prescribed angular velocity and for a broader class of admissible pressure laws exhibiting weaker decay near vacuum and slower growth at high densities. These results were subsequently extended to the non-constant entropy setting by Jang--Strauss--Wu \cite{Jang_2023}. Alternative constructions based on the Implicit Function Theorem have also been developed; see Lichtenstein \cite{Lichtenstein1933}, Heilig \cite{Heilig_1994}, and Jang--Makino \cite{Jang_2017,Jang_2019}. While the works above focus primarily on existence and structural properties, Luo--Smoller \cite{Luo_2008,Luo_2009} later proved nonlinear stability of rotating star solutions using a concentration compactness argument. In addition, Lin--Wang \cite{Lin_2023b} established linear spectral stability of rotating star solutions. For a broader overview of results and open problems concerning steady states of the CEPEs, we refer the reader to Luo \cite{Luo_2019}.

\smallskip

This paper is devoted to the study of rotating Riesz star solutions of the attractive CEREs, which generalise rotating star solutions of the attractive CEPEs. In \S\ref{stabilityyy} and \S\ref{stability:rota}, we establish both the existence and nonlinear stability of such solutions in the mass-subcritical regime. Our approach is based on a generalisation of Lions’ concentration compactness principle. By considering minimisers of the free-energy functional in axisymmetric coordinates, we prove tightness of the mass along minimising sequences, up to vertical translations. We further show that this tightness occurs around the origin, which allows us to pass to a strong limit and thereby establish the existence of rotating Riesz star solutions. To address stability, we introduce two distance functionals measuring the deviation between general solutions and rotating Riesz star configurations. Using a contradiction argument, we then prove nonlinear stability of the corresponding rotating Riesz star solutions. In \S\ref{existence:rotaingsuper}, we turn to the polytropic mass-supercritical regime and prove the existence of rotating Riesz star solutions. To the best of the author’s knowledge, such a result has not previously been established beyond the small angular velocity regime, even in the more extensively studied CEPEs setting. Our approach adapts the variational framework developed in \cite{Cheng_2025} to the rotational setting by constructing an energy minimisation problem whose minimisers correspond to rotating Riesz star solutions. The rotational term, however, introduces substantial additional difficulties, since it is not compatible with radially decreasing rearrangement techniques. Consequently, unlike in \cite[Lemma 4.6.]{Carrillo_2026}, a different argument is required to establish compactness of the potential term. Finally, in \S\ref{growing}, following the strategy developed in \cite{Carrillo_2026}, we perform a careful analysis of the concavity properties of the associated free-energy functional under mass-preserving scalings. This allows us to prove growth of the support for solutions starting arbitrarily close to a rotating Riesz star solution, thereby establishing nonlinear instability.

\smallskip

In the mass-subcritical regime, one of the principal difficulties arises in analysing minimisers of the associated free-energy functional $\mathcal{F}$ over the admissible class of axisymmetric densities $\mathcal{A}_M$. Unlike the stationary setting, translations cannot be applied freely to minimising sequences, since they generally destroy axisymmetry and therefore move the sequence outside $\mathcal{A}_M$. To address this issue, we develop a refined version of the concentration compactness framework of \cite{Lions_1984}, adapted specifically to the axisymmetric setting, in order to classify the possible behaviours of minimising sequences. We show that either the sequence undergoes a splitting phenomenon into two axisymmetric profiles whose supports drift apart, leading to a contradiction, or the mass is tight along rings in $\mathbb{R}^3$, up to vertical translations. An additional difficulty occurs when the radii of these rings become unbounded, since in this case concentration compactness arguments alone are insufficient to guarantee convergence of the potential energy. We prove that, unless the mass is tight around the origin, equivalently, unless the radii of the rings remain uniformly bounded, the potential energy vanishes in the limit. This would force the infimum of the energy to be non-negative, contradicting the underlying variational structure of the problem.

\smallskip

In the mass-supercritical regime, several difficulties arise that are absent in the stationary setting. In particular, unlike in \cite{Carrillo_2026}, it is no longer the case that every density admits a unique mass-preserving critical scaling placing the rescaled density into the natural admissible set; see Lemma \ref{smallalpha}. Indeed, such a critical scaling may fail to exist altogether, creating substantial obstacles, particularly in the analysis of growth properties of solutions. These difficulties are especially pronounced when $\alpha \in (0,2)$, whereas for $\alpha \in [2,3)$ the existence and instability theory for rotating Riesz star solutions becomes comparatively more tractable. In particular, the case $\alpha = 2$ appears to be the natural regime for the variational problem due to the scaling behaviour of the rotational kinetic energy. Nevertheless, in both regimes, the critical mass-preserving scaling associated with a given density is no longer available in explicit form, in contrast to the stationary setting, introducing an additional layer of analytical difficulty.

\smallskip

To overcome these difficulties, we exploit a structural feature of the mass-preserving scaling which allows the contribution of the rotational kinetic energy to be effectively decoupled when analysing the concavity properties of the free-energy functional along such scalings; see Lemma \ref{concave}. In the regime $\alpha \in (0,2)$, this leads to the introduction of a distinguished quantity $\kappa(\varrho)$, which separates the two possible mass-preserving critical scalings of a given density $\varrho$ into the natural admissible set. A key point is that, although the critical scalings themselves are only characterised implicitly, the quantity $\kappa(\varrho)$ admits an explicit representation. This allows us to circumvent the difficulties arising from the absence of an explicit formula for the critical scalings. By performing a careful analysis of the free-energy along each branch of critical scalings, we are able to identify the branch relevant for both the existence of minimisers and the instability mechanism. This, in turn, enables us to restrict the variational problem to a suitable subclass of the natural admissible set, which serves as the effective admissible class throughout the analysis.

\smallskip

In the regime $\alpha \in (0,2)$, further difficulties arise both in showing that minimisers of the free-energy correspond to rotating Riesz star solutions and in establishing growth of solutions arbitrarily close to such states. The underlying issue is the non-uniqueness of the mass-preserving critical scaling into the natural admissible set. This non-uniqueness can lead to delicate behaviour, for example when the two admissible critical scalings become arbitrarily close, or when a solution of the CEREs transitions from one scaling branch to the other. To overcome these difficulties, we impose a suitable superhomogeneity condition on the square of the angular momentum. Under this assumption, we prove that branch switching cannot occur and that the regime in which the two critical scalings approach one another does not produce pathological behaviour. This ensures that the variational framework remains well-posed and allows the analysis to be completed successfully.

\section{Preliminaries and results}

The CEREs form a strongly hyperbolic system and remain poorly understood in many respects. In particular, classical solutions are generally expected to develop singularities in finite time. As a result, the natural framework for studying global-in-time existence is that of weak solutions, and more specifically finite-energy weak solutions, which we define below.

\begin{definition}\label{definition}
A measurable vector-valued function $(\rho,\boldsymbol{\mathcal{M}})(t, \boldsymbol{x})$ is a finite-energy weak solution of the Cauchy problem \eqref{0.0} and \eqref{0.1} with pressure $p$ if
$(\rho,\boldsymbol{\mathcal{M}})$ satisfies the following conditions{\rm :}
\begin{enumerate}
\item[(a)] $\rho(t,\boldsymbol{x}) \geq 0$ {\it a.e.}, $(\boldsymbol{\mathcal{M}}, \frac{\boldsymbol{\mathcal{M}}}{\sqrt{\rho}})(t,\boldsymbol{x}) = \boldsymbol{0}$ {\it a.e.} on the vacuum states \\ $\{ (t,\boldsymbol{x}) \, | \, \rho(t,\boldsymbol{x}) = 0 \}$ and for any $T>0$,
\begin{align*}
			\rho\in L^{\infty}(0,T; L^1(\mathbb{R}^3)), \quad \rho e(\rho)\in L^{\infty}(0,T; L^1(\mathbb{R}^3)), \quad \frac{\boldsymbol{\mathcal{M}}}{\sqrt{\rho}}\in L^{\infty}(0,T; L^2(\mathbb{R}^3)).
    \end{align*}
    Moreover, for $\alpha \in (0,3)$ and any $T>0$,
    \[
         \Phi_\alpha * \rho\in L^\infty(0,T;L^{\frac{6}{\alpha}}(\mathbb{R}^3) ),\quad
			\nabla\Phi_\alpha * \rho\in L^\infty(0,T;L^{\frac{6}{\alpha + 2}}(\mathbb{R}^3)),
    \]
    with, for $\alpha \in [2,3)$,
    \[
        \rho \in L^\infty(0,T;C^{0,\beta}(\mathbb{R}^3)),
    \]
    for some $\beta \in (\alpha - 2,1)$.
\item[(b)] For {\it a.e.} $t>0$, the total energy is finite{\rm :}
\begin{align*}
\begin{cases}
\displaystyle
\int_{\mathbb{R}^3} \Big( \frac{1}{2} \Big| \frac{\boldsymbol{\mathcal{M}}}{\sqrt{\rho}}\Big|^2
+ \rho e(\rho) - \frac{1}{2} \rho (\Phi_\alpha * \rho) \Big)(t,\boldsymbol{x})
\dd \boldsymbol{x} \leq C(E_0,M), \\[4mm]
\displaystyle
\int_{\mathbb{R}^3} \Big( \frac{1}{2}\Big|\frac{\boldsymbol{\mathcal{M}}}{\sqrt{\rho}}\Big|^2
+ \rho e(\rho) + \frac{1}{2} \rho (\Phi_\alpha * \rho) \Big)(t,\boldsymbol{x})
\dd \boldsymbol{x} \leq E_0.
\end{cases}
\end{align*}
\item[(c)] For any $\zeta \in C^1_0(\mathbb{R}_+ \times \mathbb{R}^3)$,
\begin{equation*}
\int_{\mathbb{R}_+} \int_{\mathbb{R}^3} \big(\rho \zeta_t
+ \boldsymbol{\mathcal{M}} \cdot \nabla \zeta\big) \dd \boldsymbol{x} \dd t
+ \int_{\mathbb{R}^3} \rho_0 \zeta(0,\boldsymbol{x}) \dd \boldsymbol{x} = 0.
\end{equation*}
\item[(d)] For any $\boldsymbol{\psi} \in (C^1_0(\mathbb{R}_+ \times \mathbb{R}^3))^3$,
    \begin{align*}
        & \int_{\mathbb{R}_+} \int_{\mathbb{R}^3}
        \Big(\boldsymbol{\mathcal{M}} \cdot \boldsymbol{\psi}_t
        + \frac{\boldsymbol{\mathcal{M}}}{\sqrt{\rho}} \cdot \Big( \frac{\boldsymbol{\mathcal{M}}}{\sqrt{\rho}} \cdot \nabla \Big)
        \boldsymbol{\psi} + p(\rho)\nabla \cdot \boldsymbol{\psi} \Big)(t,\boldsymbol{x})
        \dd \boldsymbol{x} \dd t
        + \int_{\mathbb{R}^3} \boldsymbol{\mathcal{M}}_0(\boldsymbol{x}) \cdot \boldsymbol{\psi}(0,\boldsymbol{x}) \dd \boldsymbol{x} \\
        & = \int_{\mathbb{R}_+} \int_{\mathbb{R}^3} (\rho \nabla \mathcal{W}_\alpha \cdot \boldsymbol{\psi})(t,\boldsymbol{x}) \dd \boldsymbol{x} \dd t.
    \end{align*}
\end{enumerate}
\end{definition}

This paper focuses on the existence and stability of rotating Riesz star solutions, which constitute the analogue of rotating star solutions for the attractive CEPEs. By a rotating Riesz star solution, we mean a steady state solution of the attractive CEREs exhibiting rotation about a fixed axis and, in particular, possessing axisymmetry. We say that a solution $(\rho,\boldsymbol{\mathcal{M}})$ of \eqref{0.0} is axisymmetric if, with respect to the basis
    \[
        \boldsymbol{e}_r = \Big(\frac{x_1}{r},\frac{x_2}{r},0\Big), \qquad \boldsymbol{e}_\theta = \Big(-\frac{x_2}{r},\frac{x_1}{r},0\Big), \qquad \boldsymbol{e}_3 = (0,0,1),
    \]
     the density $\rho$ and momentum $\boldsymbol{\M}$ satisfy
     \begin{equation}\label{axiboy}
        \begin{cases}
            \rho(t,\boldsymbol{x}) = \rho(t,r,z), \\
            \boldsymbol{\M}(t,\boldsymbol{x}) = \mathcal{M}_r(t,r,z) \boldsymbol{e}_r + \mathcal{M}_\theta(t,r,z) \boldsymbol{e}_\theta + \mathcal{M}_3(t,r,z) \boldsymbol{e}_3,
        \end{cases} 
     \end{equation}
     where $r = r(\boldsymbol{x}) = \sqrt{x_1^2 + x_2^2}$ and $z = z(\boldsymbol{x}) = x_3$. Suppose the angular momentum of the solution is given by $J(m_\rho(r))$ with
\[
    m_{\rho}(r) \coloneq \int_{\sqrt{x_1^2 + x_2^2} \leq r} \rho(\boldsymbol{x}) \dd \boldsymbol{x} = 2\pi \int^r_0  \int^\infty_{-\infty} \rho(\eta,z) \eta \dd z \dd \eta,
\]
then, the rotational momentum is given by
\[
    \mathcal{M}_\theta(t,r) = \frac{\rho(t,r,z)J(m_{\rho(t)}(r))}{r}.
\]
Thus, more specifically, we say $(\bar{\rho}, \midbar{\boldsymbol{\M}})$ is a rotating Riesz star solution if the velocity field satisfies
$$\bar{\boldsymbol{u}}(\boldsymbol{x}) \coloneq \bigg (\frac{\midbar{\boldsymbol{\mathcal{M}}}}{\bar{\rho}} \bigg )(\boldsymbol{x}) = \frac{J(m_{\bar\rho}(r))}{r}\boldsymbol{e}_\theta.$$
Plugging this into \eqref{0.0}, we obtain
\begin{equation}\label{rotatingriesz2}
    \begin{cases}
    \displaystyle
        \partial_r p(\bar{\rho}) = \bar{\rho} \frac{L(m_{\bar{\rho}}(r))}{r^3} - \bar{\rho} \partial_r (\Phi_\alpha * \bar{\rho}), \\
        \displaystyle
        \partial_z p(\bar{\rho}) = - \bar{\rho} \partial_z (\Phi_\alpha * \bar{\rho}),
    \end{cases}
\end{equation}
where $L=J^2$. Thus, dividing \eqref{rotatingriesz2} by $\bar{\rho}$ and integrating with respect to $r$ and $z$ respectively, we conclude that there exists $\mu > 0$ such that, on the set where $\bar{\rho} > 0$,
\[
(\bar{\rho}e(\bar{\rho}))_{\bar{\rho}} = - \int^\infty_{r(\boldsymbol{x})} L(m_{\bar{\rho}}(s))s^{-3} \dd s - \Phi_\alpha\ast\bar{\rho} - \mu.
\]
Thus, it follows that
\[
(\bar{\rho}e(\bar{\rho}))_{\bar{\rho}} = \bigg (- \int^\infty_{r(\boldsymbol{x})} L(m_{\bar{\rho}}(s))s^{-3} \dd s - \Phi_\alpha\ast\bar{\rho} - \mu \bigg )_+.
\]
We shall consider $L \in C^1([0,M])$ such that
\begin{equation}\label{Lnice}
    \begin{cases}
        L(0)=0, \\
        L(m) \geq 0 \text{ for all } m \in [0,M],
    \end{cases}
\end{equation}
and have the following definition for rotating Riesz star solutions.
\begin{definition}\label{rotatingriesz}
$\bar{\rho} \in L_+^1(\mathbb{R}^3) \cap L^\infty(\mathbb{R}^3)$ is 
called a rotating Riesz star solution of \eqref{0.0}  
if $\bar{\rho}$ is axisymmetric, there exists an associated angular momentum squared $L$ and a constant $\mu > 0$ such that
\begin{align}\label{stationaxi}
\begin{cases}
\displaystyle
(\bar{\rho}e(\bar{\rho}))_{\bar{\rho}}+ \int^\infty_{r(\boldsymbol{x})} L(m_{\bar{\rho}}(s))s^{-3} \dd s +\Phi_\alpha\ast\bar{\rho}= - \mu,&  \boldsymbol{x}\in \Gamma,\\[2mm]
\displaystyle
\int^\infty_{r(\boldsymbol{x})} L(m_{\bar{\rho}}(s))s^{-3} \dd s + \Phi_\alpha\ast\bar{\rho}( \boldsymbol{x})\geq - \mu,&  \boldsymbol{x}\in \R^3 \setminus \Gamma,
\end{cases}
\end{align}
where
\[
        \Gamma \coloneq \{ \boldsymbol{x} \in \mathbb{R}^3 \, | \, \bar{\rho}(\boldsymbol{x}) > 0 \}.
\]
Then, the associated velocity field $\bar{\boldsymbol{u}}$, is given by
$$\bar{\boldsymbol{u}}(\boldsymbol{x}) = \bigg(\frac{\midbar{\boldsymbol{\mathcal{M}}}}{\bar{\rho}} \bigg )(\boldsymbol{x}) = \frac{J(m_{\bar\rho}(r))}{r}\boldsymbol{e}_\theta.$$
\end{definition}
In the mass-subcritical regime, to establish the existence and stability of rotating Riesz star solutions for general pressure laws, we adopt a variational framework similar to that developed in \cite{Carrillo_2026}. However, unlike the setting considered in \cite{Carrillo_2026}, the corresponding free-energy functional now includes an additional kinetic energy contribution arising from the velocity field $\bar{\boldsymbol{u}}$. Accordingly, we introduce the free-energy functional $\mathcal{G}$:
\[
\mathcal{G}(\varrho):=\int_{\mathbb{R}^3} \varrho(\boldsymbol{x})\Big( e(\varrho(\boldsymbol{x})) +\frac{1}{2} \frac{L(m_{\varrho}(r(\boldsymbol{x})))}{r(\boldsymbol{x})^{2}} +\frac{1}{2}(\Phi_\alpha\ast \varrho)(\boldsymbol{x})\Big)\,\dd\boldsymbol{x}.
\]
This is paired with the set of admissible functions $X_M${\rm:}
\[
X_M := \bigg \{ \varrho \in L^1_+(\mathbb{R}^3) \cap L_{\rm axi}^\frac{3 + \alpha}{3}(\mathbb{R}^3) \, \bigg | \, \int_{\mathbb{R}^3} \varrho \, \dd \boldsymbol{x} = M, \, \int_{\mathbb{R}^3} \frac{\varrho(\boldsymbol{x})L(m_{\varrho}(r(\boldsymbol{x})))}{r(\boldsymbol{x})^{2}} \dd\boldsymbol{x} < \infty \bigg \},
\]
where 
$$L_{\rm axi}^p(\mathbb{R}^3) \coloneq \{ \varrho \in L^p(\mathbb{R}^3) \, | \, \varrho \text{ is axisymmetric}\}.$$
We adapt the original concentration compactness framework of \cite{Lions_1984} to incorporate the rotational kinetic energy term. The axisymmetry constraint on the density introduces additional complications. For example, in the dichotomy scenario for minimising sequences, one must show that the two separating mass distributions are themselves axisymmetric. Moreover, due to the structure of the rotational kinetic term, it is no longer evident that radially decreasing rearrangements of minimisers of $\mathcal{G}$ over $X_M$ remain minimisers.

\smallskip

To apply a concentration compactness argument analogous to that of \cite{Lions_1984}, it is necessary that the free-energy functional $\mathcal{G}$ be bounded from below on $X_M$. In addition, we require the corresponding variational infimum $$g_M \coloneq \inf_{\varrho \in X_M} \mathcal{G}(\varrho),$$ to be strictly negative. In order to ensure that both of these properties hold, we impose the following assumptions:
\begin{enumerate}
    \item[(a)] \begin{equation}\label{higherpress}
    M < \bigg ( \frac{6 \liminf_{\rho \to \infty} p(\rho) \rho^{-\frac{3+\alpha}{3}}}{C_*} \bigg )^\frac{3}{3-\alpha},
\end{equation}
where $C_* > 0$ denotes the optimal constant in the Hardy--Littlewood--Sobolev (HLS) inequality stated in Lemma \ref{HLSineq}. This condition ensures that the internal energy dominates the potential energy at high densities, thereby guaranteeing that $g_M > -\infty$.
    \item[(b)] \begin{equation}\label{lower}
    M > \bigg ( \frac{6 \limsup_{\rho \to 0} p(\rho) \rho^{-\frac{3 +\alpha}{3}}}{ C_*} \bigg )^\frac{3}{3-\alpha}.
    \end{equation}
This condition ensures that the potential energy dominates the internal energy at low densities, thereby guaranteeing that $g_M < 0$.
\end{enumerate}

\begin{remark}
Note that condition \eqref{higherpress} is analogous to the Chandrasekhar mass limit. In the polytropic case with $\alpha = 1$ and $\gamma = \frac{4}{3}$, the corresponding critical mass for stationary gaseous stars is given by $M_{\rm ch}:= \big( \frac{6 a_0}{C_*} \big )^\frac{3}{2}$; see {\rm \cite{Chandrabook}}.
\end{remark} 
Furthermore, to exclude the possibility of dichotomy in minimising sequences, in the sense of Lemma \ref{comvandich}, we impose a subhomogeneity condition of order $\frac{5-\alpha}{3}$ on the square of the angular momentum, namely,
\begin{equation}\label{L}
    L(am) \geq a^\frac{5 - \alpha}{3}L(m),
\end{equation}
for all $0< a < 1$ and $0 \leq m \leq M$.
Given a global finite-energy weak solution $(\rho,\boldsymbol{\mathcal{M}})$ of the CEREs in the sense of Definition \ref{definition}, with total mass $M>0$, the momentum field admits the decomposition
\[
    \boldsymbol{\M}(t,\boldsymbol{x}) = \mathcal{M}_r(t,\boldsymbol{x}) \boldsymbol{e}_r + \mathcal{M}_\theta(t,\boldsymbol{x}) \boldsymbol{e}_\theta + \mathcal{M}_3(t,\boldsymbol{x}) \boldsymbol{e}_3.
\]
Thus, given a rotating Riesz star solution $\bar{\rho} \in X_M$ with associated angular momentum squared $L$, we consider the corresponding relative energy in order to identify the natural notion of distance relevant to the nonlinear stability analysis. In doing so, we obtain
\begin{equation}\label{Energydiffaxi}
\begin{split}
    E(\rho,\boldsymbol{\M}\, | \, \bar{\rho}, \bar {\boldsymbol{\mathcal{M}}})(t) & := \mathcal{G}(\rho(t)) - \mathcal{G}(\bar{\rho})
    + \frac{1}{2}\int_{\R^3} \bigg ( \bigg|\frac{\boldsymbol{\M}^r}{\sqrt{\rho}}\bigg|^2 + \bigg|\frac{\boldsymbol{\M}^\theta}{\sqrt{\rho}}\bigg|^2 + \bigg|\frac{\boldsymbol{\M}^3}{\sqrt{\rho}}\bigg|^2 \bigg)(t)\,\dd \boldsymbol{x} \\
    & \qquad \qquad \qquad \qquad \qquad \qquad \qquad \qquad \qquad - \frac{1}{2}\int_{\R^3} \frac{\rho(t) L(m_{\rho(t)}(r))}{r^2} \dd \boldsymbol{x}.
\end{split}
\end{equation}
In particular, in order to exploit the variational framework, we require the rotational velocity field to have the same structure as that associated with the rotating Riesz star solution $(\bar{\rho},\midbar{\boldsymbol{\mathcal{M}}})$. Thus, we require that
\begin{equation}\label{Lform}
    u^\theta(t,\boldsymbol{x}) = \bigg(\frac{\mathcal{M}^\theta}{\rho} \bigg )(t,\boldsymbol{x}) = \frac{J(m_{\rho(t)}(r))}{r},
\end{equation}
with $J^2= L$. To ensure that this property holds, we impose the following assumptions on the initial data:
\begin{enumerate}
    \item[(I$_1$)] The mass of the solution is the same as the stationary state
    $$
        \int_{\mathbb{R}^3} \rho_0 \dd \boldsymbol{x} = \int_{\mathbb{R}^3} \bar{\rho} \dd \boldsymbol{x} = M.
    $$
    \item[(I$_2$)] The initial angular momentum $r u_0^\theta(\boldsymbol{x}) = r u_0^\theta(r,z)$ depends only on the mass in the cylinder $\{\boldsymbol{y} \in \mathbb{R}^3 \, | \, r(\boldsymbol{y}) \leq r(\boldsymbol{x})\}$, that is, there exists $J$ such that
    $$
        u^\theta_0(\boldsymbol{x}) = \bigg(\frac{\mathcal{M}_0^\theta}{\rho_0} \bigg )(\boldsymbol{x}) = \frac{J(m_{\rho_0}(r))}{r}.
    $$
    \item[(I$_3$)] The angular momentum squared is $L$,
    $$
        J^2(m) = L(m), \qquad 0\leq m \leq M.
    $$
\end{enumerate}
As in \cite{Luo_2008}, we have the following assumptions on the weak solutions:
     \begin{enumerate}
         \item[(A$_1$)] For $D_t = \{\boldsymbol{x} \in \mathbb{R}^3 \, | \, \rho(t,\boldsymbol{x}) > 0 \}$, there exists a measurable set $G_t \subset D_t$ such that $|D_t \setminus G_t| = 0$ and a unique backwards particle path $\xi(\cdot,\cdot;t):[0,t] \times G_t \to \mathbb{R}^3$ such that
         \[
            \partial_\tau \xi(\tau,\boldsymbol{x};t) = \boldsymbol{u}(\tau, \xi(\tau,\boldsymbol{x};t)), \qquad \xi(t,\boldsymbol{x};t) = \boldsymbol{x},
         \]
         where $\boldsymbol{u} = \frac{\boldsymbol{\mathcal{M}}}{\rho}$. We define $\xi_{-t}(\boldsymbol{x}) \coloneq \xi(0,\boldsymbol{x};t)$.
         \item[(A$_2$)] Conservation of the angular momentum along the particle path \\ $j (t,\boldsymbol{x}) = j_0(\xi_{-t}(\boldsymbol{x}))$ for all $\boldsymbol{x} \in G_t$.
         \item[(A$_3$)] Conservation of the mass along the particle path $m_{\rho(t)} (r(\boldsymbol{x})) = m_{\rho_0} (r(\xi_{-t}(\boldsymbol{x})))$ for all $\boldsymbol{x} \in G_t$.
    \end{enumerate}
    Thus, as in \cite{Luo_2008}, assumptions (I$_1$)--(I$_3$) together with (A$_1$)--(A$_3$) imply that \eqref{Lform} holds. Condition (A$_1$) is a natural assumption, while conditions (A$_2$) and (A$_3$) were shown in \cite{Tassoul_1978} to hold for sufficiently regular solutions. Under these assumptions, the identity \eqref{Energydiffaxi} reduces to
    \[
        E(\rho,\boldsymbol{\M}\, | \, \bar{\rho}, \bar {\boldsymbol{\mathcal{M}}})(t) = \mathcal{G}(\rho(t)) - \mathcal{G}(\bar{\rho})
    + \frac{1}{2}\int_{\R^3} \bigg ( \bigg|\frac{\boldsymbol{\M}^r}{\sqrt{\rho}}\bigg|^2 + \bigg|\frac{\boldsymbol{\M}^3}{\sqrt{\rho}}\bigg|^2 \bigg)(t)\,\dd \boldsymbol{x}.
    \]
    Considering the difference of the free-energy functionals, we obtain
    \begin{align}\label{differenceaxi}
    \begin{split}
        \mathcal{G}(\rho(t)) - \mathcal{G}(\bar{\rho}) & = d_1(\rho(t),\bar{\rho}) + d_2(\rho(t),\bar{\rho}) + \frac{1}{2} \int_{\mathbb{R}^3} (\rho(t) - \bar{\rho})\, \Phi_\alpha* (\rho(t) - \bar{\rho}) \dd \boldsymbol{x},
    \end{split}
\end{align}
where the relative free-energy functionals $d_1(\varrho,\bar{\rho})$ and $d_2(\varrho,\bar{\rho})$, for $\varrho,\bar{\rho} \in X_M$, are defined by
\[
    \begin{cases}
        \displaystyle
        d_1(\varrho,\bar{\rho}) \coloneq \int_{\mathbb{R}^3}\bigg((\varrho e(\varrho)-\bar{\rho}e(\bar{\rho}))+ (\varrho-\bar{\rho}) \bigg(\Phi_\alpha\ast\bar{\rho}+\int_{r(\boldsymbol{x})}^\infty \frac{L(m_{\bar{\rho}}(s))}{s^3} \dd s\bigg) \bigg )\dd \boldsymbol{x}, \vspace{0.2cm} \\
        \displaystyle
        d_2(\varrho,\bar{\rho}) \coloneq \frac{1}{2} \int_{\mathbb{R}^3} \frac{\varrho L(m_{\varrho}(r(\boldsymbol{x}))) - \bar{\rho}L(m_{\bar{\rho}}(r(\boldsymbol{x})))}{r(\boldsymbol{x})^{2}} \dd \boldsymbol{x} - \int_{\mathbb{R}^3} (\varrho - \bar{\rho}) \int_{r(\boldsymbol{x})}^\infty \frac{L(m_{\bar{\rho}}(s))}{s^3} \dd s \dd \boldsymbol{x}.
    \end{cases}
\]
In the case where $\bar{\rho}$ is a rotating Riesz star solution of \eqref{0.0}, it follows from \eqref{stationaxi} that there exists $\mu > 0$ such that
\begin{align*}
    d_1(\varrho,\bar{\rho}) & = \int_{\mathbb{R}^3}\bigg((\varrho e(\varrho)-\bar{\rho}e(\bar{\rho}))+ (\varrho-\bar{\rho}) \bigg(\Phi_\alpha\ast\bar{\rho}+\int_{r(\boldsymbol{x})}^\infty \frac{L(m_{\bar{\rho}}(s))}{s^3} \dd s + \mu \bigg) \bigg )\dd \boldsymbol{x} \\
    & \geq \int_{\mathbb{R}^3}\big((\varrho e(\varrho)-\bar{\rho}e(\bar{\rho}))+(\varrho-\bar{\rho})(\bar{\rho}e(\bar{\rho}))_{\bar{\rho}}\big)\,\dd \boldsymbol{x}.
\end{align*}
By \eqref{Press}, the mapping $\varrho \mapsto \varrho e(\varrho)$ is convex, and hence $d_1(\varrho,\bar{\rho}) \geq 0$. On the other hand, to ensure non-negativity of $d_2(\varrho,\bar{\rho})$, we must impose the additional technical assumption
\smallskip
    \begin{enumerate}
         \item[(A$_4$)] $\displaystyle \lim_{r \to 0^+} L(m_\rho(r) + m_{\bar{\rho}}(r)) m_{\rho - \bar{\rho}}(r) r^{-2} = 0.$
     \end{enumerate}

     \begin{lemma}[\text{\cite[Lemma 3.8]{Luo_2008}}]\label{d2}
         Suppose $L$ satisfies \eqref{Lnice} and
         \begin{equation}\label{Lstable}
             L'(m) \geq 0,
         \end{equation}
         for all $0\leq m \leq M$, then for any $\varrho \in X_M$ satisfying
         \[
            \displaystyle \lim_{r \to 0^+} L(m_\varrho(r) + m_{\bar{\rho}}(r)) m_{\varrho - \bar{\rho}}(r) r^{-2} = 0,
         \]
         we have that
         $$d_2(\varrho,\bar{\rho}) \geq 0.$$
     \end{lemma}
     The term on the right-hand side of \eqref{differenceaxi} can be estimated using Hölder’s inequality together with the Hardy--Littlewood--Sobolev inequality from Lemma \ref{simplehls}, yielding
     \[
    \Big|\int_{\mathbb{R}^3} (\rho(t) - \bar{\rho}) \,\Phi_\alpha* (\rho(t) - \bar{\rho}) \dd \boldsymbol{x}\Big| \leq C \| \rho(t) - \bar{\rho} \|_{L^\frac{6}{6 - \alpha}}^2.
    \]
    This leads us to the following stability theorem.

\begin{theorem}[Nonlinear Stability of Rotating Riesz Star Solutions]\label{nsa}
Suppose $\bar{\rho}$ is a unique minimiser (up to translation in the $z$-axis) of the functional $\mathcal{G}$ in $X_M$. Suppose $\alpha \in (0,2)$ with $p \in C^1([0,\infty))$ satisfying \eqref{Press}, \eqref{higherpress}--\eqref{lower} and $L$ satisfies \eqref{Lnice}, \eqref{L} and \eqref{Lstable}. Let $(\rho,\boldsymbol{\M})$ be a global finite-energy weak solution of the CEREs in the sense of {\rm Definition {\ref{definition}}} satisfying {\rm(I}$_1)$--{\rm(I}$_3)$ and {\rm(A}$_1)$--{\rm(A}$_4)$ with the square of the angular momentum given by $L(m_\rho(r(\boldsymbol{x})))$. Then, for any $\v>0,$ there exists $\delta>0,$ such that if
\begin{align*}
d_1(\rho_0,\bar{\rho}) + d_2(\rho_0,\bar{\rho}) + \| \rho_0 - \bar{\rho} \|_{L^\frac{6}{6 - \alpha}}^2+\frac{1}{2}\int_{\R^3} \bigg ( \bigg|\frac{\boldsymbol{\M}^r_0}{\sqrt{\rho_0}}\bigg|^2 + \bigg|\frac{\boldsymbol{\M}^3_0}{\sqrt{\rho_0}}\bigg|^2 \bigg)\,\dd \boldsymbol{x}<\delta,
\end{align*}
there exists a translation $\boldsymbol{y}\in \{\boldsymbol{0}\} \times \mathbb{R},$ such that
\[
d_1(\rho(t),T^{\boldsymbol{y}}\bar{\rho}) + d_2(\rho(t),T^{\boldsymbol{y}}\bar{\rho}) + \| \rho(t) - T^{\boldsymbol{y}}\bar{\rho} \|_{L^{\frac{6}{6 - \alpha}}}^2+\frac{1}{2}\int_{\R^3} \bigg ( \bigg|\frac{\boldsymbol{\M}^r}{\sqrt{\rho}}\bigg|^2 + \bigg|\frac{\boldsymbol{\M}^3}{\sqrt{\rho}}\bigg|^2 \bigg)(t)\,\dd\boldsymbol{x}<\v,
\]
for {\it a.e.} $t>0$.
\end{theorem}

\begin{remark}
    We adopt an alternative approach to that of {\rm\cite{Luo_2008}} in order to avoid difficulties associated with the estimate
    \[
         \bigg \| \int_{|\boldsymbol{x} - \boldsymbol{y}| \leq 1} |\boldsymbol{x}-\boldsymbol{y}|^{-1} \rho(\boldsymbol{y}) \dd y \bigg \|_{L^4(\mathbb{R}^3)} \leq \|\mathds{1}_{B_1(\boldsymbol{x})} \rho \|^b_{L^1(\mathbb{R}^3)} \|\rho\|^{1-b}_{L^\frac{4}{3}(\mathbb{R}^3)} + \|\mathds{1}_{B_1(\boldsymbol{x})} \rho \|^c_{L^1(\mathbb{R}^3)} \|\rho\|^{1-c}_{L^\frac{4}{3}(\mathbb{R}^3)},
    \]
    for some $b,c \in (0,1)$, since the right-hand side depends on the variable $\boldsymbol{x}$, so the displayed bound does not define a scalar $L^4$-estimate as written. In {\rm\cite{Luo_2008}}, this estimate is used in the argument excluding vanishing of minimising sequences and establishing tightness, up to the natural vertical translations. Instead, we prove that if a minimising sequence fails to be tight near the origin after vertical translations, then the corresponding potential energy necessarily vanishes in the limit, leading to a contradiction. We also emphasise that our analysis extends beyond the Newtonian setting to sub-Manev Riesz potentials. Moreover, our framework yields an existence result for minimisers in a substantially more general variational setting.
\end{remark}

\smallskip

In the polytropic mass-supercritical regime ($\frac{6}{6-\alpha} < \gamma < \frac{3 + \alpha}{3}$), the free-energy functional $\mathcal{G}$ is no longer bounded from below on $X_M$. Consequently, following the approach of \cite{Carrillo_2026}, we introduce an alternative free-energy functional of the form
\[
    S_\mu(\varrho) \coloneq \frac{1}{2} \int_{\mathbb{R}^3} \frac{\varrho L(m_{\varrho}(r))}{r^2} \dd \boldsymbol{x} + \frac{a_0}{\gamma - 1} \int_{\mathbb{R}^3} \varrho^\gamma \dd \boldsymbol{x} + \frac{1}{2} \int_{\mathbb{R}^3} \varrho \Phi_\alpha * \varrho \dd \boldsymbol{x} + \mu \int_{\mathbb{R}^3} \varrho \dd \boldsymbol{x},
\]
for $\mu > 0$. If, following \cite{Carrillo_2026}, we analyse the behaviour of $S_\mu$ under mass-preserving scalings,
\begin{equation}\label{rescaling}
    \lambda \mapsto \varrho_\lambda(\boldsymbol{x}) := \lambda^\frac{3}{2} \varrho(\lambda^\frac{1}{2} \boldsymbol{x}),
\end{equation}
we observe substantially different behaviour. Differentiating the rescaled functional $S_{\mu,\lambda}(\varrho) := S_{\mu}(\varrho_\lambda)$ with respect to the mass-preserving scaling parameter yields the functional
\[
    Q(\varrho) \coloneq \int_{\mathbb{R}^3} \frac{\varrho L(m_{\varrho}(r))}{r^2} \dd \boldsymbol{x} + 3 a_0 \int_{\mathbb{R}^3} \varrho^\gamma \dd \boldsymbol{x} + \frac{\alpha}{2} \int_{\mathbb{R}^3} \varrho \Phi_\alpha * \varrho \dd \boldsymbol{x}.
\]
While in the stationary setting ($L \equiv 0$), for any $\varrho \in L^1_+(\mathbb{R}^3) \cap L^\gamma(\mathbb{R}^3)$, there exists a unique scaling parameter $\lambda^*(\varrho)$ such that $Q(\varrho_{\lambda^*(\varrho)}) = 0$. However, due to the presence and scaling behaviour of the rotational kinetic energy, this property no longer holds in general in the rotating setting. More precisely, when $\alpha \in (2,3)$, one can still prove the existence of a unique scaling $\lambda^*(\varrho)$ satisfying $Q(\varrho_{\lambda^*(\varrho)}) = 0$. By contrast, for $\alpha \in (0,2]$, the situation becomes considerably more delicate. In this regime, it is no longer guaranteed that there exists any $\lambda>0$ such that $Q(\varrho_{\lambda}) = 0$. Moreover, when $\alpha \in (0,2)$, even if such a scaling exists, uniqueness may fail. In fact, there may exist up to two distinct critical scalings $0< \lambda^*_1(\varrho) \leq \lambda^*_2(\varrho)$ such that $Q(\varrho_{\lambda_i^*(\varrho)}) = 0$ for $i=1,2$. While, in the stationary setting, the scaling $\lambda^*(\varrho)$ represents the global maximum of $\lambda \mapsto S_{\mu,\lambda}(\varrho)$, in the rotating case with $\alpha \in (0,2)$, if the two critical scalings $\lambda_1^*(\varrho)$ and $\lambda_2^*(\varrho)$ exist and are distinct, then $\lambda_1^*(\varrho)$ corresponds to a local maximum of $\lambda \mapsto S_{\mu,\lambda}(\varrho)$, whereas $\lambda_2^*(\varrho)$ corresponds to a local minimum. If the two critical scalings coincide, that is, $\lambda^*_1(\varrho) = \lambda^*_2(\varrho)$, then the corresponding point is an inflection point. Consequently, minimising the functional $S_\mu$ over the class of densities satisfying $Q(\varrho)=0$ is insufficient for constructing rotating Riesz star solutions. Indeed, $S_\mu$ is not even bounded below on this set. This necessitates the introduction of a more refined admissible class. Since $\lambda_1^*(\varrho)$ corresponds to the local maximum branch, one expects, in analogy with the stationary framework of \cite{Carrillo_2026}, that minimising over densities satisfying $\lambda^*_1(\varrho) = 1$ should produce rotating Riesz star solutions. Furthermore, as we will show, this condition is equivalent to requiring that $Q(\varrho) = 0$ and $Q(\varrho_\lambda) > 0$, for all $0<\lambda<1$. Motivated by this observation, we introduce the admissible set
\[
    \mathcal{K} \coloneq \{ \varrho \in L^1_+(\mathbb{R}^3) \cap L_{\rm axi}^\gamma(\mathbb{R}^3) \, | \, Q(\varrho) = 0, \, Q(\varrho_\lambda) > 0 \text{ for all } 0 < \lambda < 1, \,  \varrho \not\equiv 0\}.
\]
Note that, in the case $L \equiv 0$ and without the axisymmetry constraint, the admissible set $\mathcal{K}$ reduces to the corresponding admissible class introduced in \cite{Carrillo_2026}. Observe also that the definition of $\mathcal{K}$ imposes no restriction on the mass of a density $\varrho$. Consequently, it is no longer appropriate to consider angular momentum functions $L$ of the form \eqref{Lnice}. Instead, we assume that $L \in C^1([0,\infty))$ satisfies the following conditions:
\begin{equation}\label{Lnice2}
    \begin{cases}
        L(0)=0, \\
        L(m) \geq 0 \text{ for all } m \in [0,\infty).
    \end{cases}
\end{equation}
Note that any function $L$ of the form \eqref{Lnice} can be extended to a function $\bar{L}$ satisfying \eqref{Lnice2}.

\smallskip

In the polytropic setting, if $L(m) = m^{\omega^*}$, where
\begin{equation}\label{omegastar}
    \omega_* \coloneq \frac{7-\alpha -(5-\alpha)\gamma}{3 + \alpha - 3\gamma},
\end{equation}
with ${\bar{\rho}}_1$ a rotating Riesz star solution in the sense of Definition \ref{rotatingriesz} corresponding to $\mu = 1$, then under the energy-critical scaling
\begin{equation}\label{energycritical}
    \leftindex^\xi{\varrho}(\boldsymbol{x}) := \xi^{3 - \alpha} \rho \big (\xi^{2-\gamma} \boldsymbol{x} \big ),
    \end{equation}
    for $\xi > 0$, we have that $\bar{\rho}_\mu = \leftindex^{\mu^\frac{1}{(\gamma-1)(3-\alpha)}}{\bar{\rho}}_1$ satisfies Definition \ref{rotatingriesz} for any $\mu > 0$. That is,
    \begin{equation*}
        \begin{split}
            \bar{\rho}_\mu(\boldsymbol{x}) & = \Big ( \frac{\gamma - 1}{a_0 \gamma} \Big )^\frac{1}{\gamma - 1} \bigg (- \mu \int^\infty_r L(m_{\bar{\rho}_1}(s)) s^{-3} \dd s - \mu \Phi_\alpha * \bar{\rho}_1 - \mu \bigg)_+^\frac{1}{\gamma - 1}\Big (\mu^\frac{2-\gamma}{(\gamma - 1)(3-\alpha)}\boldsymbol{x} \Big) \\
            & = \Big ( \frac{\gamma - 1}{a_0 \gamma} \Big )^\frac{1}{\gamma - 1} \bigg (- \int^\infty_r L(m_{\bar{\rho}_\mu}(s)) s^{-3} \dd s - \Phi_\alpha * \bar{\rho}_\mu - \mu \bigg )_+^\frac{1}{\gamma - 1}(\boldsymbol{x}).
        \end{split}
    \end{equation*}
    Thus, there exists a scaling between rotating Riesz star solutions, this highlights the criticality of $\omega_*$. However, in the mass-supercritical regime with $\alpha \in (0,2]$, in order to establish the existence of minimisers of $S_\mu$ over $\mathcal{K}$ for angular momentum functions $L$ satisfying \eqref{Lnice2}, we must impose a stronger superhomogeneity condition of order $\omega > \omega_*$ on the square of the angular momentum, namely,
\begin{equation}\label{Linstab}
L(am) \leq a^\omega L(m),
\end{equation}
for all $0<a<1$ and $m \geq 0$. Such an assumption is necessary not only to ensure that the admissible set $\mathcal{K}$ is non-empty, but also, under the stronger condition $\omega > \omega_* + \bar{\omega}$, to guarantee that minimisers of the variational problem indeed correspond to rotating Riesz star solutions. Here,
\[
    \bar{\omega} = \frac{(2-\alpha)((6-\alpha)\gamma - 6)(5-3\gamma)}{3\alpha (\gamma - 1)(3+\alpha -3\gamma)} \geq 0.
\]
We also provide evidence that the condition $\omega > \omega_* + \bar{\omega}$ is intrinsic to the problem rather than merely a technical artefact; see Remark \ref{nottechnical}.
\begin{remark}
        The additional term $\bar{\omega}$ reflects the stabilising effect introduced by rotation in the steady state configuration. In particular, as $\gamma$ approaches the mass-critical exponent $\frac{3+\alpha}{3}$, stronger superhomogeneity assumptions on $L$ are required. This stabilising effect is specific to the regime $\alpha \in (0,2)$, and becomes weaker as $\alpha$ approaches $2$. Indeed, the corresponding restriction on the superhomogeneity of $L$ relaxes in this limit, and in particular, $\bar{\omega} \to 0$ as $\alpha \to 2$. 
    \end{remark}
    
To establish the existence of minimisers of $S_\mu$ over $\mathcal{K}$, we additionally require weak lower semi-continuity of the rotational kinetic energy term
\[
    \varrho \mapsto \int_{\mathbb{R}^3} \frac{\varrho L(m_{\varrho}(r))}{r^2} \dd \boldsymbol{x}.
\]
To ensure this property, we impose the condition that
\begin{equation}\label{Lstable2}
    L'(m) \geq 0,
\end{equation}
for all $m \geq 0$. Thus, we have the following theorem.

\begin{theorem}\label{Existofrot}
    Suppose $\alpha\in(0,3)$, $\frac{6}{6-\alpha}<\gamma<\frac{3+\alpha}{3}$, $L$ satisfies \eqref{Lnice2}, \eqref{Lstable2} and, for $\alpha \in (0,2]$, \eqref{Linstab} for $\omega > \omega_* + \bar{\omega}$. If $\bar{\rho}$ is a minimiser of $S_\mu$ over $\mathcal{K}$, then $\bar{\rho} \in C(\mathbb{R}^3)$ and is a rotating Riesz star solution as given by Definition \ref{rotatingriesz} with compact support.
\end{theorem}

Furthermore, following the approach of \cite{Carrillo_2026}, we prove that these rotating Riesz star solutions are unstable under an additional assumption.

\section{Existence in the mass-subcritical regime}\label{stabilityyy}

In this section, we consider the existence of rotating Riesz star solutions of the CEREs \eqref{0.0} for a general pressure, with sufficiently nice asymptotic behaviour at the origin and infinity, which are inclusive of the mass-subcritical regime $\gamma > \frac{3+\alpha}{3}$ in the polytropic case. We recall the definition for the admissible set of functions
\begin{align*}
X_M = \bigg \{ \varrho \in L^1_+(\mathbb{R}^3) \cap L_{\rm axi}^\frac{3 + \alpha}{3}(\mathbb{R}^3) \, \Big | \, \int_{\mathbb{R}^3} \varrho \, \dd \boldsymbol{x} = M, \, \int_{\mathbb{R}^3} \frac{\varrho(\boldsymbol{x})L(m_{\varrho}(r(\boldsymbol{x})))}{r(\boldsymbol{x})^{2}} \dd\boldsymbol{x} < \infty \bigg \}.
\end{align*}
For $\varrho\in X_M,$ we recall the definition of the free-energy functional
\begin{align*}
\mathcal{G}(\varrho)=\int_{\mathbb{R}^3} \varrho(\boldsymbol{x})\Big( e(\varrho(\boldsymbol{x})) +\frac{1}{2} \frac{L(m_{\varrho}(r(\boldsymbol{x})))}{r(\boldsymbol{x})^{2}} +\frac{1}{2}(\Phi_\alpha\ast \varrho)(\boldsymbol{x})\Big)\,\dd\boldsymbol{x}.
\end{align*}

We develop a substantial adaptation of the concentration compactness framework of \cite{Lions_1984} in order to establish the existence of minimisers for the variational problem:
\begin{equation*}
g_M = \inf_{\varrho \in X_M} \mathcal{G}(\varrho).
\end{equation*}
We first prove some facts about minimisers of general free-energies. To do this, we consider a general non-local interaction kernel of the form 
\begin{equation}\label{potentialaxi}
        \Psi \in L^p_{\rm w}(\mathbb{R}^3) \, \text{{\rm (}weak $L^p$ space{\rm )}}, \text{ for } p \in (1,\infty), \text{ with } \Psi \geq 0 \text{ {\it a.e.}},
    \end{equation}
    such that there exists $\beta \in (0,3)$ satisfying
    \begin{equation}\label{Psiaxi}
        \Psi(t \xi) \geq t^{-\beta} \Psi(\xi),
    \end{equation}
    for all $t \geq 1$ and {\it a.e.} $\xi \in \mathbb{R}^3$. We take a function $\Pi$ satisfying
    \begin{equation}\label{convex}
        \begin{cases}
        \Pi:[0,\infty) \to [0,\infty) \text{ strictly convex,} \\
        \lim_{t \to 0^+} \Pi(t)t^{-1} = 0, \\
        \liminf_{t \to \infty} \Pi(t)t^{-q} > 0,
        \end{cases}
    \end{equation}
    where $q = \frac{p + 1}{p}$. 
    We consider a function $I$ such that
\begin{equation}\label{Iconditions}
    \begin{cases}
        I \in C^1([0,M]), \\
        I(0) = 0,\\
        \text{$I'(m) \geq 0$ for all $m \in [0,M]$}, \\
    \end{cases}
\end{equation}
we also require subhomogeneity of order $\frac{5-\beta}{3}$, that is
\begin{equation}\label{Iconditions2}
    I(am) \geq a^\frac{5 - \beta}{3}I(m),
\end{equation}
for all $0< a < 1$ and $0 \leq m \leq M$.
    We define the associated free-energy
\[
    \mathcal{F}(\varrho):=\int_{\mathbb{R}^3}\Big(\Pi(\varrho(\boldsymbol{x})) + \frac{\varrho(\boldsymbol{x}) I(m_{\varrho}(r(\boldsymbol{x})))}{r(\boldsymbol{x})^{2}} - \frac{1}{2}\varrho(\boldsymbol{x})(\Psi \ast \varrho)(\boldsymbol{x})\Big)\,\dd\boldsymbol{x},
\]
with the corresponding admissible set
\begin{equation*}
        \mathcal{A}_M := \bigg \{ \varrho \in L^1_+(\mathbb{R}^3) \cap L_{\rm axi}^q(\mathbb{R}^3) \, \Big | \, \int_{\mathbb{R}^3} \varrho \, \dd \boldsymbol{x} = M,\, \int_{\mathbb{R}^3} \frac{\varrho(\boldsymbol{x})I(m_{\varrho}(r(\boldsymbol{x})))}{r(\boldsymbol{x})^{2}} \dd\boldsymbol{x} < \infty \bigg \}.
    \end{equation*}
We define the minimum of $\mathcal{F}$ over $\mathcal{A}_M$ by
\begin{equation*}
    F_M := \inf_{\varrho \in \mathcal{A}_M} \mathcal{F}(\varrho).
\end{equation*}

To establish compactness of minimising sequences, we require assumptions on the mass of the distributions analogous to \eqref{higherpress} and \eqref{lower}. We begin by characterising how these mass conditions determine both the sign and finiteness of the minimised energy.

\begin{lemma}\label{negativeaxi}
Let $p \in (1,\infty)$, with $q=\frac{p+1}{p}$. Suppose that $\Psi$ satisfies \eqref{potentialaxi}, $\Pi$ satisfies \eqref{convex} and $I$ satisfies \eqref{Iconditions}.
\begin{enumerate}
    \item[{\rm(a)}] Suppose
    \begin{equation}\label{masscond}
    M < \bigg ( \frac{2 \liminf_{t \to \infty} \Pi(t) t^{-q}}{C_*^\Psi} \bigg )^\frac{1}{2-q},
    \end{equation}
    where $C_*^\Psi>0$ is the sharp constant from the weak Young's convolution inequality {\rm Lemma \ref{precised}}, then we have that $F_M > - \infty$.

    \smallskip

    \item[{\rm(b)}] Suppose that $p > \frac{3}{2}$, $\Psi$ satisfies \eqref{Psiaxi} for $\beta \leq  \frac{3}{p}$ and is axisymmetric with $r \mapsto \Psi(r,z)$ a radially decreasing function for {\it a.e.} $z$. If $\beta = \frac{3}{p}$, suppose 
    \begin{equation*}
    M > \bigg ( \frac{2 \limsup_{t \to 0} \Pi(t) t^{-q}}{C_*^\Psi} \bigg )^\frac{1}{2-q},
    \end{equation*}
    is satisfied, then we have that $F_M < 0$.
\end{enumerate}
\end{lemma}

\begin{proof}
    \begin{enumerate}
        \item[(a)] We can take $\rho^*$ large enough such that
        \[
            \Pi(\varrho) \varrho^{-q} > \frac{C_*^\Psi}{2} M^{2-q},
        \]
        for all $\varrho \geq \rho^*$. Then, for any $\varrho \in \mathcal{A}_M$, applying the weak Young's convolution inequality, we have that
        \begin{align*}
                \mathcal{F}(\varrho) & \geq \int_{\{ \varrho \geq \rho^* \}} \Pi(\varrho(\boldsymbol{x}))\,\dd\boldsymbol{x} - \frac{C_*^\Psi}{2} \| \varrho \|_{L^1}^{2-q} \| \varrho \|_{L^q}^q \nonumber\\
                & \geq \int_{\{ \varrho \geq \rho^* \}} \varrho^q \Big ( \Pi(\varrho(\boldsymbol{x})) \varrho^{-q} - \frac{C_*^\Psi}{2} M^{2-q} \Big ) \,\dd\boldsymbol{x} - \frac{C_*^\Psi}{2} M^{2-q} \int_{\{ \varrho < \rho^* \}} \!\!\!\!\varrho^q \,\dd\boldsymbol{x} \nonumber\\
                & \geq - \frac{C_*^\Psi}{2} M^{2-q} (\varrho^*)^{q-1} \int_{\{ \varrho < \rho^* \}} \varrho \,\dd\boldsymbol{x} \\
                & \geq - \frac{C_*^\Psi}{2} M^{3-q} (\rho^*)^{q-1} \nonumber,
        \end{align*}
        where $\Pi, I\geq 0$ was used in the first inequality. Thus, $F_M > - \infty.$

        \smallskip
        
        \item[(b)] There exists $\theta \in (0,1)$ such that
        \[
            \delta_{\beta, \frac{3}{p}} \limsup_{t \to 0} \Pi(t) t^{-q} < \frac{\theta C_*^\Psi}{2} M^{2-q},
        \]
        where $$\delta_{\beta, \frac{3}{p}} = \begin{cases}
            1, & \beta = \frac{3}{p}, \\
            0, & \beta < \frac{3}{p}.
        \end{cases}$$
        By Lemma \ref{precised} and the sharpness of $C_*^\Psi > 0$, there exists $\varrho \in L^1_+(\mathbb{R}^3)\cap L^q(\mathbb{R}^3)$ such that
        \begin{equation}\label{critib}
            \int_{\mathbb{R}^3} \varrho \Psi * \varrho \dd \boldsymbol{x} > \theta C_*^\Psi \| \varrho \|_{L^q}^q \| \varrho \|_{L^1}^{2-q}.
        \end{equation}
       We take the radial decreasing rearrangement $\varrho^\#$ in the first two variables $(x_1,x_{2})$ of $\varrho$ and, since $\Psi$ is axisymmetric with $r \mapsto \Psi(r,z)$ a radially decreasing function for {\it a.e.} $z$, we obtain that the radial decreasing rearrangement $\Psi^{\#}$ in the first two variables $(x_1,x_{2})$ of $\Psi$ satisfies $\Psi^{\#} = \Psi$. Thus, applying Lemma \ref{rearrangement}, we obtain that $\varrho^\#$ satisfies \eqref{critib} and is axisymmetric. Moreover, taking $\varepsilon>0$, we can then define $\varrho^*_\varepsilon = \varrho^\# \mathds{1}_{B_\varepsilon^c \times \mathbb{R}}$, where $B_\varepsilon$ represents the open ball of radius $\varepsilon$ centred at $\boldsymbol{0} \in \mathbb{R}^2$, that satisfies $$\int_{\mathbb{R}^3} \frac{\varrho^*_\varepsilon(\boldsymbol{x})I(m_{\varrho^*_\varepsilon}(r(\boldsymbol{x})))}{r(\boldsymbol{x})^{2}} \dd\boldsymbol{x} < \infty.$$
        Taking $\varepsilon>0$ sufficiently small, we have that $\varrho^*_\varepsilon$ satisfies \eqref{critib}. So without loss of generality, we may assume that $\varrho$ is axisymmetric and satisfies
        $$\int_{\mathbb{R}^3} \frac{\varrho(\boldsymbol{x})I(m_{\varrho}(r(\boldsymbol{x})))}{r(\boldsymbol{x})^{2}} \dd\boldsymbol{x} < \infty.$$
        By density of $C_{\rm c}(\mathbb{R}^3)$ in $L^1(\mathbb{R}^3)\cap L^q(\mathbb{R}^3)$, we may assume that $\varrho \in C_{\rm c}(\mathbb{R}^3)$ and thus is a bounded function. Moreover, by employing the rescaling $$\tilde{\varrho}(\boldsymbol{x}) := \varrho\bigg (\boldsymbol{x}\bigg (\frac{\|\varrho\|_{L^1}}{M} \bigg)^\frac{1}{3} \bigg),$$
       we have that $\tilde{\varrho} \in \mathcal{A}_M$. Now, employing the mass-preserving scaling given in \eqref{rescaling}, we have that for $\upsilon = \lambda^\frac{2}{3}$, $\tilde{\varrho}_\upsilon = \lambda \tilde{\varrho}(\lambda^\frac{1}{3} \cdot) \in \mathcal{A}_M$ for all $\lambda > 0$. Considering the kinetic energy term, by integration by parts, we have that
        \begin{align*}
                \int_{\mathbb{R}^3} \frac{\tilde{\varrho}_\upsilon I(m_{\tilde{\varrho}_\upsilon}(r(\boldsymbol{x})))}{r(\boldsymbol{x})^{2}} \dd \boldsymbol{x} & = \int_{\mathbb{R}^3} \frac{\tilde{\varrho}_\upsilon I \Big (M \|\varrho\|_{L^1(\mathbb{R}^3)}^{-1} m_{\varrho}\big (r \big (\lambda^\frac{1}{3} \|\varrho\|_{L^1(\mathbb{R}^3)}^\frac{1}{3} M^{-\frac{1}{3}}\boldsymbol{x}\big ) \big ) \Big)}{r(\boldsymbol{x})^{2}} \dd \boldsymbol{x} \\
                & = \lambda^\frac{2}{3} \bigg ( \frac{M}{\| \varrho \|_{L^1(\mathbb{R}^3)}} \bigg )^\frac{1}{3} \int_{\mathbb{R}^3} \frac{\varrho I \big (M \|\varrho\|_{L^1}^{-1} m_{\varrho} (r (\boldsymbol{x} )) \big)}{r(\boldsymbol{x})^{2}} \dd \boldsymbol{x}.
        \end{align*}

       Since $\Psi$ satisfies \eqref{Psiaxi} for $\beta \leq \frac{3}{p}$, we combine this with \eqref{critib} to see for $\lambda < \frac{M}{\|\varrho\|_{L^1}}$ that
        \begin{align*}
          \int_{\mathbb{R}^3} \tilde{\varrho}_\upsilon \Psi * \tilde{\varrho}_\upsilon \dd \boldsymbol{x} 
                & =  \bigg (\frac{M}{\|\varrho\|_{L^1}} \bigg )^2 \int_{\mathbb{R}^3} \int_{\mathbb{R}^3} \varrho(\boldsymbol{x}) \varrho(\boldsymbol{y}) \Psi \bigg (\bigg(\frac{M}{\lambda \|\varrho\|_{L^1}}\bigg)^{\frac{1}{3}} (\boldsymbol{x} - \boldsymbol{y}) \bigg ) \dd \boldsymbol{x} \dd \boldsymbol{y} \\
                & \geq \lambda^\frac{\beta}{3} \bigg (\frac{M}{\|\varrho\|_{L^1}} \bigg )^{2 - \frac{\beta}{3}} \int_{\mathbb{R}^3} \varrho \Psi * \varrho \dd \boldsymbol{x} \\
                & \geq \theta C_*^\Psi M^{2-\frac{\beta}{3}} \lambda^\frac{\beta}{3} \| \varrho \|_{L^q}^q \| \varrho \|_{L^1}^{\frac{\beta}{3} - q}.
        \end{align*}
        In order for the internal energy and potential energy terms to dominate as $\lambda \to 0$, we require that $\beta < 2$. Since $p > \frac{3}{2}$, we have that $\beta \leq \frac{3}{p} < 2$. Thus, taking $\lambda< \frac{M}{\|\varrho\|_{L^1}}$, we obtain
        \begin{align*}
                \mathcal{F}(\tilde{\varrho}_\upsilon) \lambda^{-\frac{\beta}{3}} \bigg ( \frac{M}{\| \varrho \|_{L^1}} \bigg)^{\frac{\beta-3q}{3}} & \leq \bigg ( \frac{M}{\lambda\| \varrho \|_{L^1}} \bigg)^{\frac{\beta}{3}- \frac{1}{p}} \int_{\mathbb{R}^3} \frac{\Pi(\lambda \varrho)}{(\lambda \varrho)^q} \varrho^q \dd \boldsymbol{x} - \frac{\theta C_*^\Psi}{2} \| \varrho \|_{L^q}^q M^{2-q} \\
                & \,\,\,\,\, + \frac{\lambda^\frac{2 - \beta}{3}}{2} \bigg ( \frac{M}{\| \varrho \|_{L^1}} \bigg )^\frac{\beta - 3q + 1}{3} \int_{\mathbb{R}^3} \frac{\varrho I \big (M \|\varrho\|_{L^1}^{-1} m_{\varrho} (r (\boldsymbol{x} )) \big)}{r(\boldsymbol{x})^{2}} \dd \boldsymbol{x} \\
                & = \int_{\{\varrho>0\}} \bigg ( \bigg ( \frac{M}{\lambda \| \varrho \|_{L^1}} \bigg)^{\frac{\beta}{3}- \frac{1}{p}} \frac{\Pi(\lambda \varrho)}{(\lambda \varrho)^q} - \frac{\theta C_*^\Psi}{2} M^{2-q} \bigg ) \varrho^q \dd \boldsymbol{x} \\
                & \,\,\,\,\, + \frac{\lambda^\frac{2 - \beta}{3}}{2} \bigg ( \frac{M}{\| \varrho \|_{L^1}} \bigg )^\frac{\beta - 3q + 1}{3} \int_{\mathbb{R}^3} \frac{\varrho I \big (M \|\varrho\|_{L^1}^{-1} m_{\varrho} (r (\boldsymbol{x} )) \big)}{r(\boldsymbol{x})^{2}} \dd \boldsymbol{x}.
        \end{align*}
        Since $\varrho$ is a bounded function and 
        $$\limsup_{\lambda \to 0} \lambda^{\frac{1}{p} - \frac{\beta}{3}}\frac{\Pi(\lambda \varrho)}{(\lambda \varrho)^q} \leq  \frac{C_*^\Psi}{2} M^{2-q} ,$$
        for $\varrho>0$, the term $\lambda^{\frac{1}{p} - \frac{\beta}{3}}\frac{\Pi(\lambda \varrho)}{(\lambda \varrho)^q}$ is bounded for $\lambda<\frac{M}{\|\varrho\|_{L^1}}$. Hence, there exists $C>0$ such that $\lambda^{\frac{1}{p} - \frac{\beta}{3}}\frac{\Pi(\lambda \varrho)}{(\lambda \varrho)^q} \leq C$ for $\lambda<\frac{M}{\|\varrho\|_{L^1}}$, and so we can apply the Reverse Fatou Lemma to
        obtain that
        \[
            \limsup_{\lambda \to 0} \mathcal{F}(\tilde{\varrho}_\upsilon) \lambda^{- \frac{\beta}{3}} \bigg ( \frac{M}{\| \varrho \|_{L^1}} \bigg)^{\frac{\beta-3q}{3}} \leq \| \varrho \|^q_{L^q} \Big ( \delta_{\beta, \frac{3}{p}} \limsup_{t \to 0} \Pi(t) t^{-q} - \frac{\theta C_*^\Psi}{2}  M^{2-q} \Big ) < 0.
        \]
        Hence, we have that $F_M < 0$.
    \end{enumerate}
\end{proof}

We are now in a position to state the theorem characterising the existence of minimisers of $\mathcal{F}$ over $\mathcal{A}_M$, together with the convergence of minimising sequences toward such minimisers.

\begin{theorem}\label{minimiseraxi}
    Let $p \in (1,\infty)$, with $q=\frac{p+1}{p}$. Suppose $F_M < 0$, \eqref{masscond} is satisfied, $\Psi$ satisfies \eqref{potentialaxi}--\eqref{Psiaxi}, $\Pi$ satisfies \eqref{convex} and $I$ satisfies \eqref{Iconditions}--\eqref{Iconditions2}. Then for any minimising sequence $\{\varrho_k\}_k \subset \mathcal{A}_M$, there exists $\boldsymbol{y}_k \in \{\boldsymbol{0}\} \times \mathbb{R}$ and a minimiser $\bar{\rho} \in \mathcal{A}_M$ to $\mathcal{F}$ $(${\rm i.e.} $\mathcal{F}(\bar{\rho}) = F_M)$ with $T^{\boldsymbol{y}_k} \varrho_k$ compact in $L^1(\mathbb{R}^3) \cap L^q(\mathbb{R}^3)$, with $T^{\boldsymbol{y}_k} \varrho_k \to \bar{\rho}$ in $L^1(\mathbb{R}^3) \cap L^q(\mathbb{R}^3)$.
\end{theorem}

The proof of Theorem \ref{minimiseraxi} is inspired by \cite[Theorem II.1]{Lions_1984}; however, the presence of the rotational kinetic energy introduces additional difficulties that require new analytical arguments. We begin by proving that minimising sequences are uniformly bounded in $L^q$.

\begin{lemma}\label{boundy}
    Under the assumptions of Theorem \ref{minimiseraxi}, let $m \in (0,M]$ and $\varrho_k \in \mathcal{A}_m$ be a minimising sequence $($i.e. $\lim_{k\to\infty} \mathcal{F}(\varrho_k) = F_m)$. Then, $\|\varrho_k\|_{L^q(\mathbb{R}^3)}$ and $\int_{\mathbb{R}^3} \varrho_k \Psi * \varrho_k \dd \boldsymbol{x}$ are bounded.
\end{lemma}

\begin{proof}
    By \eqref{masscond}, there exists $\delta>0$ and $\varrho^*>0$ such that
        \[
            \Pi(\varrho) \varrho^{-q} > \frac{C^\Psi_* M^{2-q}}{2} + \delta,
        \]
        for all $\varrho \geq \varrho^*$. By Lemma \ref{precised}, we have that
     \begin{align*}
            \begin{split}
                \mathcal{F}(\varrho_k) & \geq \int_{\{\varrho_k \geq \varrho^*\}} \Pi(\varrho_k) \dd \boldsymbol{x} - \frac{1}{2} \int_{\mathbb{R}^3} \varrho_k \Psi*\varrho_k \dd \boldsymbol{x} \\
                & \geq \int_{\{\varrho_k \geq \varrho^*\}} \Pi(\varrho_k) \dd \boldsymbol{x} - \frac{C_*^\Psi m^{2-q}}{2} \int_{\mathbb{R}^3} \varrho_k^q \dd \boldsymbol{x} \\
                & = \int_{\{ \varrho_k \geq \varrho^* \}} \varrho_k^q \Big ( \Pi(\varrho_k(\boldsymbol{x})) \varrho_k^{-q} - \frac{C_*^\Psi m^{2-q}}{2}   \Big ) \,\dd\boldsymbol{x} - \frac{C_*^\Psi m^{2-q}}{2} \int_{\{ \varrho_k < \varrho^* \}} \varrho_k^q \dd \boldsymbol{x} \\
                & \geq \delta \|\varrho_k\|_{L^q(\mathbb{R}^3)}^q - \frac{C_*^\Psi m^{3-q} (\varrho^*)^{q-1}}{2}. \\
            \end{split}
        \end{align*}
        Hence, we have that
        \begin{equation}\label{qbound}
            \|\varrho_k\|_{L^q(\mathbb{R}^3)}^q \leq C \big (m^{3-q} + \mathcal{F}(\varrho_k)\big ),
        \end{equation}
        where $C > 0$ only depends on $\Psi,\varrho^*,\Pi$. Equally, since by Lemma \ref{precised},
        \[
            \int_{\mathbb{R}^3} \varrho \Psi * \varrho \dd \boldsymbol{x} \leq C^\Psi_* \| \varrho \|_{L^q(\mathbb{R}^3)}^q \| \varrho \|_{L^1(\mathbb{R}^3)}^{2-q},
        \]
        we have that
        \begin{equation}\label{conbound}
        \int_{\mathbb{R}^3} \varrho_k \Psi*\varrho_k \dd \boldsymbol{x} \leq C m^{2-q}\big (m^{3-q} + \mathcal{F}(\varrho_k)\big).
        \end{equation}
        The result follows.
\end{proof}

To exclude behaviours such as dichotomy of minimising sequences, as described in Lemma \ref{comvandich}, we require continuity of the minimised energy together with the condition that
\[
    F_M < F_m + F_{M-m},
\]
for all $m \in (0,M)$. To establish this, we employ the following lemma, which compares the minimised energies corresponding to different masses.

\begin{lemma}\label{compareminb}
    Under the assumptions of Theorem \ref{minimiseraxi}, we have that $F_l \geq \big (\frac{l}{m}\big)^\frac{6 - \beta}{3} F_m$ for all $0<l\leq m \leq M$.
\end{lemma}

\begin{proof}
    The proof follows a similar argument to that of \cite{Lions_1984} and \cite{Luo_2008}. Let $\varrho_m \in \mathcal{A}_m$, we take $b = \big(\frac{m}{l}\big)^\frac{1}{3}$ and define $\varrho_l = \varrho_m(b\,\cdot) \in \mathcal{A}_l$. Then by integration by parts and $\eqref{Iconditions2}$, we have that
    \begin{align*}
        \int_{\mathbb{R}^3} \frac{\varrho_l I(m_{\varrho_l}(r(\boldsymbol{x})))}{r(\boldsymbol{x})^{2}} \dd \boldsymbol{x} & = \Big(\frac{l}{m} \Big)^\frac{1}{3}\int_{\mathbb{R}^3} \frac{\varrho_m I\big(\frac{l}{m}m_{\varrho_m}(r(\boldsymbol{x}))\big)}{r(\boldsymbol{x})^{2}} \dd \boldsymbol{x} \\
        & \geq \Big(\frac{l}{m} \Big)^\frac{6 - \beta}{3}\int_{\mathbb{R}^3} \frac{\varrho_m I(m_{\varrho_m}(r(\boldsymbol{x})))}{r(\boldsymbol{x})^{2}} \dd \boldsymbol{x}.
    \end{align*}
    Employing \eqref{Psiaxi}, we have that
    \[
    \begin{split}
        \int_{\mathbb{R}^3} \varrho_l(\boldsymbol{x}) \varrho_l(\boldsymbol{y}) \Psi(\boldsymbol{x} - \boldsymbol{y}) \dd \boldsymbol{x} \dd \boldsymbol{y} & = \Big(\frac{l}{m}\Big)^2 \int_{\mathbb{R}^3} \varrho_m(\boldsymbol{x}) \varrho_m(\boldsymbol{y}) \Psi\Big(\Big(\frac{l}{m}\Big)^\frac{1}{3}(\boldsymbol{x} - \boldsymbol{y})\Big) \dd \boldsymbol{x} \dd \boldsymbol{y} \\
        & \leq \Big(\frac{l}{m}\Big)^\frac{6 - \beta}{3} \int_{\mathbb{R}^3} \varrho_m(\boldsymbol{x}) \varrho_m(\boldsymbol{y}) \Psi(\boldsymbol{x} - \boldsymbol{y}) \dd \boldsymbol{x} \dd \boldsymbol{y}.
    \end{split}
    \]
    Thus, combining, we see that
    \[
    \begin{split}
        \mathcal{F}(\varrho_l) & \geq \frac{l}{m} \int_{\mathbb{R}^3} \Pi(\varrho_m) \dd \boldsymbol{x} + \Big(\frac{l}{m} \Big)^\frac{6 - \beta}{3}\int_{\mathbb{R}^3} \frac{\varrho_m I(m_{\varrho_m}(r(\boldsymbol{x})))}{r(\boldsymbol{x})^{2}} \dd \boldsymbol{x} \\
        & \qquad - \Big(\frac{l}{m}\Big)^\frac{6 - \beta}{3} \int_{\mathbb{R}^3} \varrho_m(\boldsymbol{x}) \varrho_m(\boldsymbol{y}) \Psi(\boldsymbol{x} - \boldsymbol{y}) \dd \boldsymbol{x} \dd \boldsymbol{y} \\
        & \geq \Big(\frac{l}{m}\Big)^\frac{6 - \beta}{3} \mathcal{F}(\varrho_m).
    \end{split}
    \]
    Since $\varrho_m\mapsto\varrho_l$ is a one-to-one mapping between $\mathcal{A}_m$ and $\mathcal{A}_l$, the result follows.
\end{proof}

We can now establish that the minimised energy satisfies the required inequality.

\begin{lemma}\label{mininequal}
    Under the assumptions of Theorem \ref{minimiseraxi}, 
    \begin{equation}\label{niceineq}
        F_M < F_m + F_{M-m},
    \end{equation}
    for all $m \in (0,M)$.
\end{lemma}

\begin{proof}
    Suppose that $m \in (0,M)$ and let $\varrho \in \mathcal{A}_m$, such that $\mathcal{F}(\varrho) < \infty$, $\varrho \in C_{\rm c}(\mathbb{R}^3)$ and thus a bounded function. Using $\upsilon = \lambda^\frac{2}{3}$ and $\varrho_\upsilon = \lambda \varrho(\lambda^\frac{1}{3} \cdot) \in \mathcal{A}_m$, by $\eqref{convex}_2$ ($\lim_{t \to 0^+} \Pi(t)t^{-1} = 0$) and the dominated convergence Theorem, we obtain
    \begin{align*}
        \int_{\mathbb{R}^3} \Pi(\varrho_\upsilon) \dd \boldsymbol{x} + \int_{\mathbb{R}^3} \frac{\varrho_\upsilon I(m_{\varrho_\upsilon}(r(\boldsymbol{x})))}{r(\boldsymbol{x})^{2}} \dd \boldsymbol{x} & = \int_{\mathbb{R}^3} \frac{\Pi(\lambda \varrho)}{\lambda} \dd \boldsymbol{x} + \lambda^\frac{2}{3} \int_{\mathbb{R}^3} \frac{\varrho I(m_{\varrho}(r(\boldsymbol{x})))}{r(\boldsymbol{x})^{2}} \dd \boldsymbol{x} \\
        & \longrightarrow 0, \qquad \mbox{as $\lambda \to 0$.}
    \end{align*}
    Hence, we have that $F_m \leq 0$. Suppose that \eqref{niceineq} does not hold for some $m \in (0,M)$, that is
    \[
        0 > F_M \geq F_m + F_{M-m}.
    \]
    Then without loss of generality, we have that $F_m < 0$. Thus, by Lemma \ref{compareminb} and the fact that $\beta \in (0,3)$, we have that
    \[
    \frac{m}{M} F_M \leq \Big ( \frac{M}{m} \Big )^\frac{\beta - 3}{3} F_m < F_m.
    \]
    Equally, since $F_{M-m} \leq 0$, we get
    \[
        \frac{M - m}{M}F_M \leq F_{M-m}.
    \]
    Combining, we see that
    \[
        F_M = \frac{m}{M} F_M + \frac{M - m}{M} F_M < F_m + F_{M-m}.
    \]
    Thus, \eqref{niceineq} must hold for all $m \in (0,M)$.
\end{proof}

The following lemma proves the continuity of $m \mapsto F_m$.

\begin{lemma}\label{contofmin}
    Under the assumptions of Theorem \ref{minimiseraxi}, we have that $m \mapsto F_m$ is continuous for all $m \in (0,M]$.
\end{lemma}

\begin{proof}
We split the proof into two steps.
        \begin{step}
            We first prove that $\limsup_{l \to m} (F_{l} - F_m) \leq 0$.
            \begin{case}
                Suppose that $l > m$, applying Lemma \ref{compareminb}, we have that
                \[
                    \limsup_{l \to m^+} (F_{l} - F_m) \leq \limsup_{l \to m^+} \Big (\Big ( \frac{l}{m}\Big)^\frac{6 - \beta}{3} - 1 \Big)F_m = 0.
                \]
            \end{case}
            \begin{case}
                Suppose that $l < m$, we have that for all $\varepsilon > 0$, there exists $\varrho_m \in \mathcal{A}_m$ such that $F_m \geq \mathcal{F}(\varrho_m) - \varepsilon$. Then, we see that $F_{l} - F_m \leq \mathcal{F}\big(\frac{l\varrho_m}{m}\big) - \mathcal{F}(\varrho_m) + \varepsilon$. Since $I'(m) \geq 0$ for all $m \in [0,M]$ and $\Pi:[0,\infty) \to [0,\infty)$ is strictly convex with $\Pi(0) = 0$, $\Pi$ must be increasing, we have that
                \[
                    \mathcal{F}\Big(\frac{l\varrho_m}{m}\Big) - \mathcal{F}(\varrho_m) \leq \Big ( 1 - \Big ( \frac{l}{m} \Big )^2 \Big) \int_{\mathbb{R}^3} \varrho_m \Psi * \varrho_m \dd \boldsymbol{x}.
                \]
                Thus,
                \[
                    \limsup_{l \to m^-} (F_{l} - F_m) \leq \varepsilon + \limsup_{l \to m^-} \Big ( 1 - \Big ( \frac{l}{m} \Big )^2 \Big) \int_{\mathbb{R}^3} \varrho_m \Psi * \varrho_m \dd \boldsymbol{x} = \varepsilon.
                \]
                Hence, since $\varepsilon > 0$ is arbitrary, we see that $ \limsup_{l \to m^-} (F_{l} - F_m) \leq 0$.
            \end{case}
             \setcounter{case}{0}
        \end{step}
        \begin{step}
            We now prove that $\liminf_{l \to m} (F_{l} - F_m) \geq 0$.
            \begin{case}
                Suppose that $l < m$, applying Lemma \ref{compareminb}, we have that
                \[
                    \liminf_{l \to m^-} (F_{l} - F_m) \geq \liminf_{l \to m^-} \Big (\Big ( \frac{l}{m}\Big)^\frac{6 - \beta}{3} - 1 \Big)F_m = 0.
                \]
            \end{case}
            \begin{case}
                Suppose that $l > m$, we have that for all $\varepsilon > 0$, there exists $\varrho_l \in \mathcal{A}_l$ such that $F_l \geq \mathcal{F}(\varrho_l) - \varepsilon$. Then, we see that $F_{l} - F_m \geq \mathcal{F}(\varrho_l) - \mathcal{F}\big(\frac{m\varrho_l}{l}\big) - \varepsilon$. Since $I$ and $\Pi$ are increasing, we have that
                \[
                    \mathcal{F}(\varrho_m) - \mathcal{F}\Big(\frac{m\varrho_l}{l}\Big) \geq \Big ( \Big ( \frac{m}{l} \Big )^2 - 1 \Big) \int_{\mathbb{R}^3} \varrho_l \Psi * \varrho_l \dd \boldsymbol{x}.
                \]
                We need to show that $\int_{\mathbb{R}^3} \varrho_l \Psi * \varrho_l \dd \boldsymbol{x}$ remains bounded as $l \to m
                $ in order to pass the limit. By employing Lemma \ref{compareminb} and a similar argument to that used to obtain \eqref{qbound} and \eqref{conbound} in Lemma \ref{boundy}, we have that
        \[
            \|\varrho_l\|_{L^q(\mathbb{R}^3)}^q \leq C \big (l^{3-q} + \mathcal{F}(\varrho_l)\big ) \leq C \Big (\varepsilon + l^{3-q} + \Big ( \frac{l}{m} \Big )^\frac{6 - \beta}{3} F_m\Big ),
        \]
        and
        \[
        \int_{\mathbb{R}^3} \varrho_l \Psi*\varrho_l \dd \boldsymbol{x} \leq C l^{2-q} \big (l^{3-q} + \mathcal{F}(\varrho_l)\big ) \leq C l^{2-q} \Big ( \varepsilon + l^{3-q} + \Big ( \frac{l}{m} \Big )^\frac{6 - \beta}{3} F_m\Big ).
        \]
        Thus,
                \[
                    \liminf_{l \to m^+} (F_{l} - F_m) \geq - \varepsilon + \liminf_{l \to m^+} \Big ( \Big ( \frac{m}{l} \Big )^2 - 1 \Big) \int_{\mathbb{R}^3} \varrho_m \Psi * \varrho_m \dd \boldsymbol{x} = -\varepsilon.
                \]
                Hence, since $\varepsilon > 0$ is arbitrary, we see that $ \liminf_{l \to m^+} (F_{l} - F_m) \geq 0$.
            \end{case}
            \setcounter{case}{0}
        \end{step}
\end{proof}
\setcounter{step}{0}

We now turn to the convergence of minimising sequences for $\mathcal{F}$ over $\mathcal{A}_M$. The following lemma allows us to establish weak convergence of such minimising sequences.

\begin{lemma}\label{weaklowersemi}
    Suppose that $q > 1$, $\varrho_k \in L^1_+(\mathbb{R}^3) \cap L^q (\mathbb{R}^3)$ is a bounded sequence. Then, there exists $\tilde{\varrho} \in L^1_+(\mathbb{R}^3) \cap L^q (\mathbb{R}^3)$ such that  $\varrho_k \rightharpoonup \tilde{\varrho}$ in $L^q(\R^3)$ {\rm(}up to a subsequence{\rm)} and
    \[
    \|\tilde{\varrho}\|_{L^1} \leq \liminf_{k \to \infty} \|\varrho_k\|_{L^1}.
    \]
\end{lemma}
The proof is the same as in \cite{Cheng_2025}. 

\smallskip

To show that the weak limit of a minimising sequence is indeed a minimiser of $\mathcal{F}$ over $\mathcal{A}_M$, we require weak lower semicontinuity of the kinetic energy term. This property will also play an important role in establishing the existence of minimisers in \S\ref{existence:rotaingsuper}. We prove this in the following lemma. To this end, it is convenient to introduce the definition
\[
    \mathbb{R}^3_* \coloneq \mathbb{R}^3 \setminus (\{\boldsymbol{0}\} \times \mathbb{R}).
\]

\begin{lemma}\label{lowerkinetic}
    Suppose that $I$ satisfies \eqref{Iconditions}, $q > 1$, $\varrho_k \in \mathcal{A}_M$ and there exists $\tilde{\varrho} \in L^1_+(\mathbb{R}^3) \cap L^q_{\rm axi} (\mathbb{R}^3)$ such that $\varrho_k \rightharpoonup \tilde{\varrho}$ as $k \to \infty$ in $L^q(\mathbb{R}^3)$. Then, we have that
    \[
        \int_{\mathbb{R}^3} \frac{\tilde{\varrho} I(m_{\tilde{\varrho}}(r))}{r^2} \dd \boldsymbol{x} \leq \liminf_{k \to \infty} \int_{\mathbb{R}^3} \frac{\varrho_k I(m_{\varrho_k}(r))}{r^2} \dd \boldsymbol{x}.
    \]
\end{lemma}

\begin{proof}
    If $\liminf_{k \to \infty} \int_{\mathbb{R}^3} \frac{\varrho_k I(m_{\varrho_k}(r))}{r^2} \dd \boldsymbol{x} = \infty$, the result follows trivially. So, without loss of generality, we assume that $\int_{\mathbb{R}^3} \frac{\varrho_k I(m_{\varrho_k}(r))}{r^2} \dd \boldsymbol{x}$ is a bounded sequence. Let $S>0$ and define
    $$m^S_{\varrho_k}(r) \coloneq \int^S_{-S} \int_{B_r} \varrho_k(\boldsymbol{y},z) \dd \boldsymbol{y} \dd z.$$
    Since $\varrho_k \rightharpoonup \tilde{\varrho}$ in $L^q(\mathbb{R}^3)$ as $k \to \infty$,
    \[
        m^S_{\varrho_k}(r) \longrightarrow m^S_{\tilde{\varrho}}(r), \qquad \mbox{as $k \to \infty$.}
    \]
    As in \cite{Luo_2008}, since $m^S_{\tilde{\varrho}}(r)$ is continuous by the Lebesgue dominated convergence Theorem, we use the variation of Dini's Theorem in \cite[p.167]{Rudin_76} to obtain that for any $R>0$, that
    \[
        \| m^S_{\varrho_k}(r) - m^S_{\tilde{\varrho}}(r) \|_{L^\infty([0,R])} \longrightarrow 0, \qquad \mbox{as $k \to \infty$.}
    \]
    Since $I \in C^1([0,\infty))$, we have that
    \begin{equation}\label{Lconverge}
        \| I(m^S_{\varrho_k}(r)) - I(m^S_{\tilde{\varrho}}(r)) \|_{L^\infty([0,R])} \longrightarrow 0, \qquad \mbox{as $k \to \infty$.}
    \end{equation}
    Thus, since $I'(m) \geq 0$ for all $m \geq 0$, we have that
    \[
        \int_{\mathbb{R}^3} \frac{\varrho_k I(m^S_{\varrho_k}(r))}{r^2} \dd \boldsymbol{x} \leq \int_{\mathbb{R}^3} \frac{\varrho_k I(m_{\varrho_k}(r))}{r^2} \dd \boldsymbol{x},
    \]
    is a bounded sequence. Hence, by weak compactness in the space of signed Radon measures $\mathcal{M}(\mathbb{R}^3)$, there exists $\nu^S \in \mathcal{M}(\mathbb{R}^3)$ such that
    \[
        \frac{\varrho_k I(m^S_{\varrho_k}(r))}{r^2} \longrightharpoonup \nu^S, \qquad \mbox{in $\mathcal{M}(\mathbb{R}^3)$ as $k \to \infty$.}
    \]
    By properties of weak convergence of signed Radon measures, as detailed in \cite[Theorem 1.1.3]{Evans_1990}, we have that
    \begin{equation}\label{nu0}
        \nu^S(\{\boldsymbol{0}\} \times \mathbb{R}) \geq \limsup_{k \to \infty} \int_{\{\boldsymbol{0}\}\times \mathbb{R}} \frac{\varrho_k I(m^S_{\varrho_k}(r))}{r^2} \dd \boldsymbol{x} = 0,
    \end{equation}
    and
    \begin{equation}\label{muR3}
        \nu^S(\mathbb{R}^3) \leq \liminf_{k \to \infty} \int_{\mathbb{R}^3} \frac{\varrho_k I(m^S_{\varrho_k}(r))}{r^2} \dd \boldsymbol{x} \leq \liminf_{k \to \infty} \int_{\mathbb{R}^3} \frac{\varrho_k I(m_{\varrho_k}(r))}{r^2} \dd \boldsymbol{x} < \infty.
    \end{equation}
    Let $g \in C_{\rm c}(\mathbb{R}^3_*)$, due to \eqref{Lconverge}, we have that since $\varrho_k \rightharpoonup \tilde{\varrho}$ in $L^q(\mathbb{R}^3)$ as $k \to \infty$, that
    \[
        \int_{\mathbb{R}^3} \frac{\varrho_k(\boldsymbol{x}) I(m^S_{\varrho_k}(r))}{r^2} g(\boldsymbol{x}) \dd \boldsymbol{x} \longrightarrow \int_{\mathbb{R}^3} \frac{\tilde{\varrho}(\boldsymbol{x}) I(m^S_{\tilde{\varrho}}(r))}{r^2} g(\boldsymbol{x}) \dd \boldsymbol{x}.
    \]
    This implies that on $\mathbb{R}^3_*$, $\nu^S \equiv \frac{\tilde{\varrho} I(m^S_{\tilde{\varrho}}(r))}{r^2}$. Thus, combining this with \eqref{nu0} and \eqref{muR3}, we have that
    \begin{align*}
        \int_{\mathbb{R}^3} \frac{\tilde{\varrho}(\boldsymbol{x}) I(m^S_{\tilde{\varrho}}(r))}{r^2} \dd \boldsymbol{x} 
        & = \nu^S(\mathbb{R}^3_*) \\
        & = \nu^S(\mathbb{R}^3) - \nu^S(\{ \boldsymbol{0} \} \times \mathbb{R}) \\
        & \leq \liminf_{k \to \infty} \int_{\mathbb{R}^3} \frac{\varrho_k I(m_{\varrho_k}(r))}{r^2} \dd \boldsymbol{x}.
    \end{align*}
    Thus, applying the Monotone convergence Theorem and the fact that $I'(m) \geq 0$ for all $m \geq 0$, we have that
    \[
    \begin{split}
        \int_{\mathbb{R}^3} \frac{\tilde{\varrho}(\boldsymbol{x}) I(m_{\tilde{\varrho}}(r))}{r^2} \dd \boldsymbol{x} & = \lim_{S \to \infty} \int_{\mathbb{R}^3} \frac{\tilde{\varrho}(\boldsymbol{x}) I(m^S_{\tilde{\varrho}}(r))}{r^2} \dd \boldsymbol{x}\\
        & \leq \liminf_{k \to \infty} \int_{\mathbb{R}^3} \frac{\varrho_k I(m_{\varrho_k}(r))}{r^2} \dd \boldsymbol{x}. \\
    \end{split}
    \]
\end{proof}

We now state the following lemma, which characterises the possible behaviour of minimising sequences. Using the framework developed above, we will show that phenomena such as vanishing and dichotomy cannot occur for minimising sequences.

\begin{lemma}[\text{\cite[Lemma I.1]{Lions_1984}}]\label{comvandich}
    Let $\varrho_k \in L^1(\mathbb{R}^n)$ be a sequence satisfying{\rm:}
    \[
        \varrho_k \geq 0 \text{ in } \mathbb{R}^n, \qquad \int_{\Omega} \varrho_k \dd \boldsymbol{x} = M,
    \]
    for fixed $M>0$. Then there exists a subsequence $\varrho_{k_l}$ satisfying one of the three possibilities{\rm:}
    \begin{enumerate}
        \item[{\rm(a)}] {\rm(}compactness{\rm)} There exists $\boldsymbol{y}_l \in \mathbb{R}^n$ such that $\varrho_{k_l}(\cdot + \boldsymbol{y}_{l})$ is tight, that is, for all $\varepsilon > 0$, there exists $R>0$ such that
        \[
            \int_{\boldsymbol{y}_l + B_{R}} \varrho_{k_l} \dd \boldsymbol{x} \geq M - \varepsilon;
        \]
        \item[{\rm(b)}] {\rm(}vanishing{\rm)} For all $R < \infty$,
        \[
            \lim_{l \to \infty} \sup_{\boldsymbol{y} \in \mathbb{R}^n} \int_{\boldsymbol{y} + B_R} \varrho_{k_l} \dd \boldsymbol{x} = 0;
        \]
        \item[{\rm(c)}] {\rm(}dichotomy{\rm)} There exists $m \in (0,M)$ such that for all $\varepsilon > 0$, there exist $l_0 \geq 1$ and $\varrho^1_l,\varrho^2_l \in L^1_+(\mathbb{R}^n)$ satisfying for $l \geq l_0$
        \[
        \begin{cases}
            \displaystyle
            \|\varrho_{k_l} - (\varrho^1_l + \varrho^2_l) \|_{L^1(\mathbb{R}^n)} \leq \varepsilon, \vspace{0.2cm}\\
            \displaystyle
            \bigg | \int_{\mathbb{R}^n} \varrho^1_l \dd \boldsymbol{x} - m \bigg | \leq \varepsilon, \vspace{0.2cm} \\ 
            \displaystyle \bigg | \int_{\mathbb{R}^n} \varrho^2_l \dd \boldsymbol{x} - (M - m) \bigg | \leq \varepsilon, \vspace{0.2cm}\\
            \displaystyle
            { \rm dist}(\supp(\varrho^1_l),\supp(\varrho^2_l)) \longrightarrow \infty, \text{ as } l \to \infty.
        \end{cases}
        \]
    \end{enumerate}
\end{lemma}

Unlike the setting of \cite{Lions_1984}, where one may work with radially symmetric minimising sequences, the presence of angular momentum in our problem prevents us from concluding that the radially decreasing rearrangement of a minimising sequence remains minimising. Consequently, Lemma \ref{comvandich} must be applied in a different manner, leading to an additional possible scenario for the behaviour of minimising sequences in $L^1$.

\begin{proof}[Proof of Theorem \ref{minimiseraxi}]
We split the proof into five steps.
    \begin{step}
     From Lemma \ref{negativeaxi} (a), we have that $F_M > - \infty$, as such, there exists a minimising sequence $\varrho_k \in \mathcal{A}_M$ such that $\mathcal{F}(\varrho_k) \to F_M$ as $k \to \infty$. From Lemma \ref{boundy}, we have that the sequence $\varrho_k$ is bounded in $L^1_+(\mathbb{R}^3) \cap L^q(\mathbb{R}^3)$. We can apply Lemma \ref{comvandich} to our minimising sequence, our aim is to show that vanishing and dichotomy cannot occur. For $\overline{r}$, the extension of $r>0$ to $\mathbb{R}$, we define $$f_k(\overline{r},z) \coloneq \begin{cases}
         2\pi \varrho_k(\overline{r},z)\overline{r}, & \overline{r}>0, \\
         0, & \overline{r} \leq 0.
     \end{cases} $$
     We must have that $f_k \in L^1(\mathbb{R}^2)$. We apply Lemma \ref{comvandich} to $f_k$.
     \end{step}
    \begin{step}
        Suppose that (c) (dichotomy) in Lemma \ref{comvandich} occurs, then (up to a subsequence), for each $\varepsilon>0$, there exists $f^1_k, f^2_k, g_k \in L^1_+(\mathbb{R}^2)$ such that $0 \leq f^1_k, f^2_k, g_k \leq f_k$ and $f^1_k f^2_k = f^1_k g_k = f^2_k g_k = 0$ almost everywhere. We let $d_k = {\rm dist}(\supp (f^1_k), \supp( f^2_k))$ and
        \[
            m^1_k = \int_{\mathbb{R}^2} f^1_k \dd \overline{r} \dd z, \qquad m^2_k = \int_{\mathbb{R}^2} f^2_k \dd \overline{r} \dd z.
        \]
        Then, without loss of generality, we may assume that $d_k \to \infty$ as $k \to \infty$, $m^1_k \to m^1(\varepsilon)$, $m^2_k \to m^2(\varepsilon)$ for $m^1(\varepsilon) \in (0,M)$ and $m^2(\varepsilon) \in (0,M-m^1(\varepsilon)]$ with $|m^1(\varepsilon) - m| \leq \varepsilon$ and $|m^2(\varepsilon) - (M-m)| \leq \varepsilon$ for some $m \in (0,M)$. Defining
        \begin{align*}
            \varrho^1_k(r,z) \coloneq \frac{f^1_k(r,z)}{2\pi r}, \\
            \varrho^2_k(r,z) \coloneq \frac{f^2_k(r,z)}{2\pi r}, \\
            \sigma_k(r,z) \coloneq \frac{g_k(r,z)}{2\pi r},
        \end{align*}
        $\varrho^1_k, \varrho^2_k, \sigma_k \in L^1_+(\mathbb{R}^3)$ are axisymmetric such that $0 \leq \varrho^1_k, \varrho^2_k, \sigma_k \leq \varrho_k$ and $\varrho^1_k \varrho^2_k = \varrho^1_k \sigma_k = \varrho^2_k \sigma_k = 0$ almost everywhere. They also satisfy
        \[
            m^1_k = \int_{\mathbb{R}^3} \varrho^1_k \dd \boldsymbol{x}, \qquad m^2_k = \int_{\mathbb{R}^3} \varrho^2_k \dd \boldsymbol{x},
        \]
        ${\rm dist}(\supp (\varrho^1_k), \supp( \varrho^2_k))= d_k \to \infty$ as $k \to \infty$. As in the proof of \cite[Theorem II.1]{Lions_1984}, one can show that
        \[
            \bigg | \int_{\mathbb{R}^3} \int_{\mathbb{R}^3} \varrho^1_k(\boldsymbol{x}) \varrho^2_k(\boldsymbol{y}) \Psi(\boldsymbol{x} - \boldsymbol{y}) \dd \boldsymbol{x} \dd \boldsymbol{y} \bigg | \longrightarrow 0, \qquad \mbox{as $k \to \infty$,}
        \]
        and there exists a $C>0$ such that for $i=1,2$,
        \[
        \bigg | \int_{\mathbb{R}^3} \int_{\mathbb{R}^3} \varrho^i_k(\boldsymbol{x}) \sigma_k(\boldsymbol{y}) \Psi(\boldsymbol{x} - \boldsymbol{y}) \dd \boldsymbol{x} \dd \boldsymbol{y} \bigg |^2 + \bigg | \int_{\mathbb{R}^3} \int_{\mathbb{R}^3} \sigma_k(\boldsymbol{x}) \sigma_k(\boldsymbol{y}) \Psi(\boldsymbol{x} - \boldsymbol{y}) \dd \boldsymbol{x} \dd \boldsymbol{y} \bigg | \leq C \varepsilon^{2-q}.
        \]
        Moreover, since $I'(m) \geq 0$ for all $m \in [0,M]$ and $\Pi:[0,\infty) \to [0,\infty)$ is strictly convex with $\Pi(0) = 0$, $\Pi$ must be increasing, so it can be seen that
        \begin{align*}
            \int_{\mathbb{R}^3}\Big(\Pi(\varrho_k(\boldsymbol{x})) + \frac{\varrho_k(\boldsymbol{x}) I(m_{\varrho_k}(r(\boldsymbol{x})))}{r(\boldsymbol{x})^2} \Big)\,\dd\boldsymbol{x} & \geq  \int_{\mathbb{R}^3}\Big(\Pi(\varrho^1_k(\boldsymbol{x})) + \frac{\varrho^1_k(\boldsymbol{x}) I(m_{\varrho^1_k}(r(\boldsymbol{x})))}{r(\boldsymbol{x})^2} \Big)\,\dd\boldsymbol{x} \\
            & \qquad + \int_{\mathbb{R}^3}\Big(\Pi(\varrho^2_k(\boldsymbol{x})) + \frac{\varrho^2_k(\boldsymbol{x}) I(m_{\varrho^2_k}(r(\boldsymbol{x})))}{r(\boldsymbol{x})^2} \Big)\,\dd\boldsymbol{x}.
        \end{align*}
        Thus, we obtain that
        \begin{align*}
        F_M & \geq \liminf_{k \to \infty} \bigg ( \int_{\mathbb{R}^3}\Big(\Pi(\varrho^1_k(\boldsymbol{x})) + \frac{\varrho^1_k(\boldsymbol{x}) I(m_{\varrho^1_k}(r(\boldsymbol{x})))}{r(\boldsymbol{x})^2} - \frac{1}{2} \varrho^1_k \Psi * \varrho^1_k \Big)\,\dd\boldsymbol{x} \bigg ) \\
        & \quad + \liminf_{k \to \infty} \bigg ( \int_{\mathbb{R}^3}\Big(\Pi(\varrho^2_k(\boldsymbol{x})) + \frac{\varrho^2_k(\boldsymbol{x}) I(m_{\varrho^2_k}(r(\boldsymbol{x})))}{r(\boldsymbol{x})^2} - \frac{1}{2} \varrho^2_k \Psi * \varrho^2_k \Big)\,\dd\boldsymbol{x} \bigg ) + O(\varepsilon^{1-\frac{q}{2}}) \\
        & \geq F_{m^1(\varepsilon)} + F_{m^2(\varepsilon)} + O(\varepsilon^{1-\frac{q}{2}}).
        \end{align*}
        Letting $\varepsilon \to 0$, by Lemma \ref{contofmin}, we obtain that
        \[
            F_M \geq F_m + F_{M-m},
        \]
        a contradiction to Lemma \ref{mininequal}.
     \end{step}
     \begin{step}
         Suppose that (b) (vanishing) in Lemma \ref{comvandich} occurs, then as in the proof of \cite[Theorem II.1]{Lions_1984} (up to a subsequence), we have that
         \[
            \int_{\mathbb{R}^3} \int_{\mathbb{R}^3} \varrho_k(\boldsymbol{x}) \varrho_k(\boldsymbol{y}) \Psi(\boldsymbol{x} - \boldsymbol{y}) \dd \boldsymbol{x} \dd \boldsymbol{y} \longrightarrow 0, \qquad \mbox{as $k \to \infty$.}
        \]
        But, this would imply that $F_M \geq 0$, a contradiction.
     \end{step}
     \begin{step}
         Thus, we must have that (a) (compactness) in Lemma \ref{comvandich} occurs, then (up to a subsequence) there exists $(\overline{r}_k,z_k) \in \mathbb{R}^2$ such that for any $\varepsilon > 0$, there exists $R>0$ such that
         \[
            \int_{B_R(\overline{r}_k,z_k)^c} f_k \dd \overline{r} \dd z \leq \varepsilon.
         \]
         In particular, due to the form of $f_k$, we can assume that $\overline{r}_k \geq 0$ and thus, use the notation $r_k = \overline{r}_k$. We can assume $z_k = 0$ by taking $\tilde{f}_k = f_k(\cdot,\cdot -z_k)$ and $\tilde{\varrho}_k = \varrho_k(\cdot,\cdot - z_k)$. Suppose that $r_k$ does not have a bounded subsequence, in particular, we have that $r_k \to \infty$ as $k \to \infty$. Then, we define $\tilde{f}^1_k(r,z;R) = \mathds{1}_{B_R(r_k,0)} \tilde{f}_k$ and $\tilde{f}^2_k(r,z;R) = \mathds{1}_{B_R(r_k,0)^c} \tilde{f}_k$. Then, as before, we define
         \begin{align*}
            \tilde{\varrho}^1_k(r,z;R) \coloneq \frac{\tilde{f}^1_k(r,z;R)}{2\pi r}, \\
            \tilde{\varrho}^2_k(r,z;R) \coloneq \frac{\tilde{f}^2_k(r,z;R)}{2\pi r},
        \end{align*}
        we have that $\tilde{\varrho}_k = \tilde{\varrho}^1_k + \tilde{\varrho}^2_k$, $0 \leq \tilde{\varrho}^1_k, \tilde{\varrho}^2_k, \leq \tilde{\varrho}_k$ and $\tilde{\varrho}^1_k \tilde{\varrho}^2_k = 0$ almost everywhere. For convenience, we suppress the $R$ in the notation so that $\tilde{\varrho}^1_k(\boldsymbol{x};R) = \tilde{\varrho}^1_k(\boldsymbol{x})$. We claim that
        \[
                \int_{\mathbb{R}^3} \int_{\mathbb{R}^3} \tilde{\varrho}_k(\boldsymbol{x}) \tilde{\varrho}_k(\boldsymbol{y}) \Psi(\boldsymbol{x} - \boldsymbol{y}) \dd \boldsymbol{x} \dd \boldsymbol{y} \longrightarrow 0, \qquad \mbox{as $k \to \infty$.}
        \]
        This comes from the dissipation of the density, meaning that as $k\to \infty$, the amount of local interactions decrease. By Lemma \ref{precised}, we have that
        \begin{equation}\label{otherterms}
        \begin{split}
            &\int_{\mathbb{R}^3} \int_{\mathbb{R}^3} \tilde{\varrho}^2_k(\boldsymbol{x}) \tilde{\varrho}^2_k(\boldsymbol{y}) \Psi(\boldsymbol{x} - \boldsymbol{y}) \dd \boldsymbol{x} \dd \boldsymbol{y} \leq C^\Psi_* \varepsilon^{2-q} \|\tilde{\varrho}_k\|_{L^q(\mathbb{R}^3)}^q, \\
            &\int_{\mathbb{R}^3} \int_{\mathbb{R}^3} \tilde{\varrho}^2_k(\boldsymbol{x}) \tilde{\varrho}^1_k(\boldsymbol{y}) \Psi(\boldsymbol{x} - \boldsymbol{y}) \dd \boldsymbol{x} \dd \boldsymbol{y} \leq C^\Psi_* (\varepsilon M)^\frac{2-q}{2} \|\tilde{\varrho}_k\|_{L^q(\mathbb{R}^3)}^q.
        \end{split}
        \end{equation}
        For $\delta > 0$, we have that
        \begin{equation}\label{boundpsi}
            \int_{\mathbb{R}^3} \int_{\mathbb{R}^3} \tilde{\varrho}^1_k(\boldsymbol{x}) \tilde{\varrho}^1_k(\boldsymbol{y}) \mathds{1}_{\{\Psi \leq \delta\}}\Psi(\boldsymbol{x} - \boldsymbol{y}) \dd \boldsymbol{x} \dd \boldsymbol{y} \leq \delta M^2.
        \end{equation}
        Defining, $\Psi_\delta \coloneq \mathds{1}_{\{\Psi \geq \delta\}}\Psi \in L^r(\mathbb{R}^3)$ for all $1 \leq r < p$, we have that
        \[
            \begin{split}
                \int_{\mathbb{R}^3} \int_{\mathbb{R}^3} \tilde{\varrho}^1_k(\boldsymbol{x}) \tilde{\varrho}^1_k(\boldsymbol{y}) \Psi_\delta(\boldsymbol{x} - \boldsymbol{y}) \dd \boldsymbol{x} \dd \boldsymbol{y} &= \int_{\mathbb{R}^3} \int_{\{|\boldsymbol{x}-\boldsymbol{y}| < \sqrt{r_k}\}} \tilde{\varrho}^1_k(\boldsymbol{x}) \tilde{\varrho}^1_k(\boldsymbol{y}) \Psi_\delta(\boldsymbol{x} - \boldsymbol{y}) \dd \boldsymbol{x} \dd \boldsymbol{y} \\
                &\qquad  + \int_{\mathbb{R}^3} \int_{\{|\boldsymbol{x}-\boldsymbol{y}| \geq \sqrt{r_k}\}} \tilde{\varrho}^1_k(\boldsymbol{x}) \tilde{\varrho}^1_k(\boldsymbol{y}) \Psi_\delta(\boldsymbol{x} - \boldsymbol{y}) \dd \boldsymbol{x} \dd \boldsymbol{y}.
            \end{split}
        \]
        By Young's convolution inequality, taking $\frac{q}{2(q-1)} = \frac{pq}{2} \leq r < p$, we have that
        \begin{align*}
            \int_{\mathbb{R}^3} \int_{\{|\boldsymbol{x}-\boldsymbol{y}| \geq \sqrt{r_k}\}} \tilde{\varrho}^1_k(\boldsymbol{x}) \tilde{\varrho}^1_k(\boldsymbol{y}) \Psi_\delta(\boldsymbol{x} - \boldsymbol{y}) \dd \boldsymbol{x} \dd \boldsymbol{y} \leq C \|\Psi_\delta\|_{L^r((B_{\sqrt{r_k}} )^c  )} M^\frac{q(2r-1) - 2r}{r(q-1)} \|\tilde{\varrho}_k\|_{L^q(\mathbb{R}^3)}^\frac{q}{r(q-1)}.
        \end{align*}
        Note, we have that $\|\Psi_\delta\|_{L^r( (B_{\sqrt{r_k}} )^c )} \to 0$ as $k \to \infty$ by the dominated convergence Theorem. Thus,
        \begin{equation}\label{uppersigmapsi}
            \int_{\mathbb{R}^3} \int_{\{|\boldsymbol{x}-\boldsymbol{y}| \geq  \sqrt{r_k}\}} \tilde{\varrho}^1_k(\boldsymbol{x}) \tilde{\varrho}^1_k(\boldsymbol{y}) \Psi_\delta(\boldsymbol{x} - \boldsymbol{y}) \dd \boldsymbol{x} \dd \boldsymbol{y} \longrightarrow 0, \qquad \mbox{as $k \to \infty$.}
        \end{equation}
        For $\boldsymbol{x} \in \mathbb{R}^3$, we define $\boldsymbol{x}^2 = (x_1,x_{2})$. For fixed $R > 0$, letting $k \to \infty$, there exists $C = C(R) > 0$ such that the amount of open balls $B_{\sqrt{r_k}}(\boldsymbol{x}) = \{\boldsymbol{y} \in \mathbb{R}^3 \, | \, |\boldsymbol{x} - \boldsymbol{y}| \leq \sqrt{r_k} \}$ with $|\boldsymbol{x}^2| = r_k$ and $x_3 = 0$ needed to cover the set $\{ \boldsymbol{y} \in \mathbb{R}^3 \, | \, (|\boldsymbol{y}^2|,y_3) \in B_R(r_k,0) \}$ is
        \[
            \sim C r_k^\frac{1}{2},
        \]
        as $k \to \infty$. Hence, for $k$ large enough, there exists a covering set
        $\text{Cov}_k(R) \subset \{ \boldsymbol{x} \in \mathbb{R}^3 \, | \, |\boldsymbol{x}^{2}|=r_k, \, x_{3}= 0 \}$ such that $|\text{Cov}_k(R)| \leq C r_k^\frac{1}{2}$ and
        \[
            \supp(\tilde{\varrho}^1_k) \subset \{ \boldsymbol{y} \in \mathbb{R}^3 \, | \, (|\boldsymbol{y}^2|,y_3) \in B_R(r_k,0) \} \subset \bigcup_{\boldsymbol{x} \in \text{Cov}_k(R)} B_{\sqrt{r_k}}(\boldsymbol{x}).
        \]
        Due to the axisymmetry of $\tilde{\varrho}^1_k$, we have that there exists $C = C(R)>0$ such that
        \begin{equation}\label{coveringbeh}
            \begin{split}
                & \|\tilde{\varrho}^1_k\|_{L^1\big(B_{2\sqrt{r_k}}(\boldsymbol{z})\big)} \leq C 
                r_k^{-\frac{1}{2}} \|\tilde{\varrho}^1_k\|_{L^1(\mathbb{R}^3)}, \\ & \|\tilde{\varrho}^1_k\|^q_{L^q\big(B_{2\sqrt{r_k}}(\boldsymbol{z})\big)} \leq C 
                r_k^{-\frac{1}{2}} \|\tilde{\varrho}^1_k\|^q_{L^q(\mathbb{R}^3)}.
            \end{split}
        \end{equation}
        Thus, applying Lemma \ref{precised} and \eqref{coveringbeh}, we have that
        \[
            \begin{split}
                 & \int_{\mathbb{R}^3} \int_{\{|\boldsymbol{x}-\boldsymbol{y}| <  \sqrt{r_k}\}} \tilde{\varrho}^1_k(\boldsymbol{x}) \tilde{\varrho}^1_k(\boldsymbol{y}) \Psi_\delta(\boldsymbol{x} - \boldsymbol{y}) \dd \boldsymbol{x} \dd \boldsymbol{y}\\
                 & \leq \sum_{\boldsymbol{z} \in \text{Cov}_k(R)} \int_{B_{\sqrt{r_k}}(\boldsymbol{z})} \int_{\{|\boldsymbol{x}-\boldsymbol{y}| <  \sqrt{r_k}\}} \tilde{\varrho}^1_k(\boldsymbol{x}) \tilde{\varrho}^1_k(\boldsymbol{y}) \Psi_\delta(\boldsymbol{x} - \boldsymbol{y}) \dd \boldsymbol{x} \dd \boldsymbol{y} \\
                 & \leq \sum_{\boldsymbol{z} \in \text{Cov}_k(R)} \int_{B_{\sqrt{r_k}}(\boldsymbol{z})} \int_{B_{2\sqrt{r_k}}(\boldsymbol{z})} \tilde{\varrho}^1_k(\boldsymbol{x}) \tilde{\varrho}^1_k(\boldsymbol{y}) \Psi_\delta(\boldsymbol{x} - \boldsymbol{y}) \dd \boldsymbol{x} \dd \boldsymbol{y} \\
                 & \leq C^\Psi_* \sum_{\boldsymbol{z} \in \text{Cov}_k(R)} \|\tilde{\varrho}^1_k\|^{2-q}_{L^1(B_{2\sqrt{r_k}}(\boldsymbol{z}))} \|\tilde{\varrho}^1_k\|^{q}_{L^q(B_{2\sqrt{r_k}}(\boldsymbol{z}))} \\
                 & \leq C r_k^{-\frac{(2-q)}{2}} M^{2-q} \|\tilde{\varrho}_k\|^{q}_{L^q(\mathbb{R}^3)}.
            \end{split}
        \]
        Thus, we have that
        \begin{equation}\label{convolutedballcovering}
            \int_{\mathbb{R}^3} \int_{\{|\boldsymbol{x}-\boldsymbol{y}| < 2 \sqrt{r_k}\}} \tilde{\varrho}^1_k(\boldsymbol{x}) \tilde{\varrho}^1_k(\boldsymbol{y}) \Psi_\delta(\boldsymbol{x} - \boldsymbol{y}) \dd \boldsymbol{x} \dd \boldsymbol{y} \longrightarrow 0, \qquad \mbox{as $k \to \infty$.}
        \end{equation}
        Thus, combining \eqref{otherterms}--\eqref{convolutedballcovering}, we have that
        \[
                \int_{\mathbb{R}^3} \int_{\mathbb{R}^3} \tilde{\varrho}_k(\boldsymbol{x}) \tilde{\varrho}_k(\boldsymbol{y}) \Psi(\boldsymbol{x} - \boldsymbol{y}) \dd \boldsymbol{x} \dd \boldsymbol{y} \longrightarrow 0, \qquad \mbox{as $k \to \infty$.}
        \]
        This implies that $F_M \geq 0$, which is a contradiction. Thus our argument shows that $r_k$ must remain bounded and hence we can assume without loss of generality that $r_k = 0$ for all $k$. 
        \end{step}
        \begin{step}
        Hence for all $\varepsilon>0$, there exists $R>0$ such that
         \begin{equation}\label{concentration}
             \int_{B_R^c} \tilde{\varrho}_k \dd \boldsymbol{x} \leq \varepsilon.
         \end{equation}
         Since $\tilde{\varrho}_k \in L^1_+(\mathbb{R}^3)\cap L^q(\mathbb{R}^3)$ is a bounded sequence in $L^1_+(\mathbb{R}^3)\cap L^q(\mathbb{R}^3)$, by Lemma \ref{weaklowersemi}, there exists $\bar{\rho} \in L^1_+(\mathbb{R}^3)\cap L^q(\mathbb{R}^3)$ such that $\tilde{\varrho}_k \rightharpoonup \bar{\rho}$ in $L^q(\mathbb{R}^3)$. Since $\tilde{\varrho}_k$ is axisymmetric, so is $\bar{\rho}$. Due to the weak semi-continuity of the $L^1$ norm and \eqref{concentration}, we have that
        \begin{equation}\label{smalltilde}
            \int_{B_R^c} \bar{\rho} \dd \boldsymbol{x} \leq \varepsilon.
        \end{equation}
        By weak convergence in $L^q$, we see that
        \begin{equation}\label{localconb}
            \int_{B_R} \tilde{\varrho}_k \dd \boldsymbol{x} \longrightarrow \int_{B_R} \bar{\rho} \dd \boldsymbol{x}, \qquad \mbox{as $k \to \infty$.}
        \end{equation}
        Tying together \eqref{smalltilde} and \eqref{localconb}, we see that $\int_{\mathbb{R}^3} \bar{\rho} \dd \boldsymbol{x} = M$. By Lemma \ref{lowerkinetic}, we obtain
        \[
            \int_{\mathbb{R}^3} \frac{\bar{\rho}(\boldsymbol{x}) I(m_{\bar{\rho}}(r(\boldsymbol{x})))}{r(\boldsymbol{x})^2} \dd\boldsymbol{x} \leq \liminf_{k \to \infty} \int_{\mathbb{R}^3} \frac{\tilde{\varrho}_k(\boldsymbol{x}) I(m_{\tilde{\varrho}_k}(r(\boldsymbol{x})))}{r(\boldsymbol{x})^2} \dd\boldsymbol{x}.
        \] 
        Hence, $\bar{\rho} \in \mathcal{A}_M$. We now aim to show that $\bar{\rho}$ is in fact a minimiser of $\mathcal{F}$ over $\mathcal{A}_M$. We have that since the functional $\varrho \mapsto \int_{\mathbb{R}^3} \Pi(\varrho) \dd \boldsymbol{x}$ is convex and strongly lower semi-continuous by Fatou's Lemma, it must be weakly lower semi-continuous. Hence,
         \[
             \int_{\mathbb{R}^3} \Pi(\bar{\rho}) \dd \boldsymbol{x} \leq \liminf_{k \to \infty} \int_{\mathbb{R}^3} \Pi(\tilde{\varrho}_k) \dd \boldsymbol{x}.
         \]
         As in \cite{Lions_1984}, one can prove that
        \[
            \int_{\mathbb{R}^3} \int_{\mathbb{R}^3} \tilde{\varrho}_k(\boldsymbol{x}) \tilde{\varrho}_k(\boldsymbol{y}) \Psi(\boldsymbol{x} - \boldsymbol{y}) \dd \boldsymbol{x} \dd \boldsymbol{y} \longrightarrow \int_{\mathbb{R}^3} \int_{\mathbb{R}^3} \bar{\rho}(\boldsymbol{x}) \bar{\rho}(\boldsymbol{y}) \Psi(\boldsymbol{x} - \boldsymbol{y}) \dd \boldsymbol{x} \dd \boldsymbol{y}, \quad \mbox{as $k \to \infty$.}
        \]
        Thus, we see that $\bar{\rho}$ is indeed a minimiser of $\mathcal{F}$ over $\mathcal{A}_M$ and in fact that
        \[
             \int_{\mathbb{R}^3} \Pi(\tilde{\varrho}_k) \dd \boldsymbol{x} \longrightarrow \int_{\mathbb{R}^3} \Pi(\bar{\rho}) \dd \boldsymbol{x}, \qquad \text{as } k \to \infty.
         \]
         As in \cite{Lions_1984}, owing to the strict convexity of $\Pi$, we have that $\tilde{\varrho}_k \to \bar{\rho}$ in measure. Thus, we can prove that $\tilde{\varrho}_k \to \bar{\rho}$ in $L^1(\mathbb{R}^3) \cap L^q(\mathbb{R}^3)$.
     \end{step}
\end{proof}
\setcounter{step}{0}

We now apply the general framework to our particular minimisation problem. Taking $\Psi = - \Phi_\alpha$, which satisfies \eqref{potentialaxi} for $p = \frac{3}{\alpha}$, $q = \frac{p + 1}{p} = \frac{3 + \alpha}{3}$ and satisfies \eqref{Psiaxi} for $\beta = \alpha$. For the internal energy function $\Pi(\varrho) := \varrho e(\varrho)$, we have that taking the second derivative, that $\Pi''(\varrho) = \frac{p'(\varrho)}{\varrho} > 0$ for $\varrho > 0$ from the assumption \eqref{Press} on $p$. This implies that $\Pi$ is strictly convex on $[0,\infty)$ and we can clearly check from the formula for $e(\varrho)$, $\eqref{Press}$, \eqref{higherpress} and \eqref{lower}, that it is non-negative and satisfies \eqref{convex}. Moreover, $I = \frac{1}{2}L$ satisfies \eqref{Iconditions} and \eqref{Iconditions2} for $\beta = \alpha$.
If $\alpha \in (0,2)$ and
\[
    M > \bigg ( \frac{6 \limsup_{\rho \to 0} p(\rho) \rho^{-\frac{3 +\alpha}{3}}}{ C_*} \bigg )^\frac{3}{3-\alpha},
\]
we have $g_M < 0$ by Lemma \ref{negativeaxi}. If
\[
    M < \bigg ( \frac{6 \liminf_{\rho \to \infty} p(\rho) \rho^{-\frac{3 +\alpha}{3}}}{ C_*} \bigg )^\frac{3}{3-\alpha},
\]
we have $g_M > -\infty$ by Lemma \ref{negativeaxi}.
By Theorem \ref{minimiseraxi} we deduce that there exists a minimiser $\bar{\rho}$ to $\mathcal{G}$ over $X_M$. Thus, we obtain the following corollary.

\begin{corollary}\label{stateexistaxi}
    Suppose $\alpha \in (0,2)$, $M>0$ with $p \in C^1([0,\infty))$ satisfying \eqref{Press}, \eqref{higherpress}--\eqref{lower} and $L$ satisfies \eqref{Lnice}, \eqref{L} and \eqref{Lstable}. Then there exists a minimiser $\bar{\rho} \in X_M$ to $\mathcal{G}$.
\end{corollary}

Combining this with the approach of \cite{Auchmuty_1971,Luo_2008,Luo_2009}, one obtains the following theorem.

\begin{theorem}\label{lmulaxi}
Under the conditions of Corollary \ref{stateexistaxi}, if $\bar{\rho} \in X_M$ is a minimiser of $\mathcal{G}$ over $X_M$, then $\bar{\rho} \in C(\mathbb{R}^3)$ and is a rotating Riesz star solution as given by Definition \ref{rotatingriesz} with compact support.
\end{theorem}

\begin{proof}
We split the proof into five steps.
\begin{step}
    We first show that the Euler-Lagrange equations of the minimisation problem correspond to \eqref{stationaxi}. As such, for $\varepsilon > 0$, we define $\mathbb{R}^3_\varepsilon \coloneq \{\boldsymbol{x} \in \mathbb{R}^3 \, | \,r(\boldsymbol{x}) > \varepsilon\}$ and consider the set
    \[
    \Gamma_\varepsilon := \Big \{ \boldsymbol{x} \in \mathbb{R}^3_\varepsilon \, \Big | \, \varepsilon < \bar{\rho}(\boldsymbol{x}) < \frac{1}{\varepsilon} \Big \}.
    \]
    Let $w \in L^\infty(\mathbb{R}^3)$ be such that $w$ has compact support in $\mathbb{R}^3_\varepsilon$ and is non-negative on $\mathbb{R}^3 \setminus \Gamma_\varepsilon$. Then, defining
    \[
        v = w - \frac{\mathds{1}_{\Gamma_\varepsilon}}{|\Gamma_\varepsilon|}\int_{\mathbb{R}^3} w \dd \boldsymbol{x},
    \]
    and
    \[
        \bar{\rho}_\tau \coloneq \bar{\rho} + \tau v,
    \]
    for $\tau \geq 0$ small enough, we have $\bar{\rho}_\tau \in X_M$. Since $\bar{\rho}$ is a minimiser of $\mathcal{G}$ over $X_M$, we have that
    \begin{equation}\label{minimisethingy}
        0 \leq \lim_{\tau \to 0^+} \frac{\d \mathcal{G}(\bar{\rho}_\tau)}{\d \tau}  = \frac{1}{2} \int_{\mathbb{R}^3} \frac{v L(m_{\bar{\rho}}(r)) + \bar{\rho} L'(m_{\bar{\rho}}(r)) m_v(r)}{r^2} \dd \boldsymbol{x} + \int_{\mathbb{R}^3} v \bigg ( (\bar{\rho}e(\bar{\rho}))_{\bar{\rho}} + \Phi_\alpha * \bar{\rho} \bigg ) \dd \boldsymbol{x}.
    \end{equation}
    We focus upon the kinetic terms. We have that by the co-area formula and integration by parts
        \begin{align}\label{3.22new}
                \int_{\mathbb{R}^3} \frac{\bar{\rho} L'(m_{\bar{\rho}}(r)) m_v(r)}{r^2} \dd \boldsymbol{x} & = \int_0^\infty \frac{(L(m_{\bar{\rho}}(r)))_r m_v(r)}{r^2} \dd r \nonumber \\
                & = 2\pi \int_0^\infty \int_\mathbb{R} v(r,z) r \dd z \int^\infty_r\frac{(L(m_{\bar{\rho}}(s)))_s}{s^2} \dd s \dd r \nonumber \\
                & \qquad \qquad \qquad \qquad - \bigg [ m_v(r) \int^\infty_r\frac{(L(m_{\bar{\rho}}(s)))_s}{s^2} \dd s\bigg ]^\infty_{r=0} \nonumber \\
                & = \int_{\mathbb{R}^3} v \int^\infty_r\frac{(L(m_{\bar{\rho}}(s)))_s}{s^2} \dd s \dd \boldsymbol{x}.
        \end{align}
        Note that applying integration by parts, we obtain
        \begin{equation}\label{3.23new}
            \int^\infty_r\frac{(L(m_{\bar{\rho}}(s)))_s}{s^2} \dd s = \bigg [ \frac{L(m_{\bar{\rho}}(s))}{s^2} \bigg ]^\infty_{s=r} + 2 \int^\infty_r\frac{L(m_{\bar{\rho}}(s))}{s^3} \dd s.
        \end{equation}
        Thus, combining \eqref{minimisethingy}-\eqref{3.23new}, we obtain that
        \[
            0 \leq \int_{\mathbb{R}^3} w \bigg ( (\bar{\rho}e(\bar{\rho}))_{\bar{\rho}} + \int^\infty_r\frac{L(m_{\bar{\rho}}(s))}{s^3} \dd s + \Phi_\alpha * \bar{\rho} + \mu_\varepsilon \bigg ) \dd \boldsymbol{x},
        \]
        where
        \[
            \mu_\varepsilon \coloneq - \frac{1}{|\Gamma_\varepsilon|} \int_{\Gamma_\varepsilon} \bigg ( (\bar{\rho}e(\bar{\rho}))_{\bar{\rho}} + \int^\infty_r\frac{L(m_{\bar{\rho}}(s))}{s^3} \dd s + \Phi_\alpha * \bar{\rho} \bigg ) \dd \boldsymbol{x}.
        \]
        Thus, since this is true for all such $w$, we obtain that
        \begin{align*}
            & (\bar{\rho}e(\bar{\rho}))_{\bar{\rho}} + \int^\infty_r\frac{L(m_{\bar{\rho}}(s))}{s^3} \dd s + \Phi_\alpha * \bar{\rho} + \mu_\varepsilon = 0, \qquad  \text{ {\it a.e.} on } \Gamma_\varepsilon,\\
            & (\bar{\rho}e(\bar{\rho}))_{\bar{\rho}} + \int^\infty_r\frac{L(m_{\bar{\rho}}(s))}{s^3} \dd s + \Phi_\alpha * \bar{\rho} + \mu_\varepsilon \geq 0, \qquad \text{ {\it a.e.} on } \mathbb{R}^3_\varepsilon \setminus \Gamma_\varepsilon.
        \end{align*}
        Since this is true for all $\varepsilon>0$, we must have that $\mu_\varepsilon$ is constant in $\varepsilon$. Hence, there exists $\mu \in \mathbb{R}$ such that $\mu = \mu_\varepsilon$ for all $\varepsilon>0$. Thus, letting $\varepsilon \to 0$, we obtain
    \[
    \begin{split}
        & (\bar{\rho}e(\bar{\rho}))_{\bar{\rho}} + \int^\infty_r\frac{L(m_{\bar{\rho}}(s))}{s^3} \dd s + \Phi_\alpha * \bar{\rho} + \mu = 0, \qquad  \text{ {\it a.e.} on } \Gamma,\\
        & \int^\infty_r\frac{L(m_{\bar{\rho}}(s))}{s^3} \dd s + \Phi_\alpha * \bar{\rho} + \mu \geq 0, \qquad \qquad \qquad \,\,\,\, \text{ {\it a.e.} on } \mathbb{R}^3 \setminus \Gamma.
    \end{split}
    \]
    Thus $\bar{\rho}$ satisfies {\it a.e.} on $\mathbb{R}^3$,
    \begin{equation}\label{rotsteadynew}
        (\bar{\rho}e(\bar{\rho}))_{\bar{\rho}} = \bigg ( -\int^\infty_r\frac{L(m_{\bar{\rho}}(s))}{s^3} \dd s - \Phi_\alpha * \bar{\rho} - \mu \bigg )_+.
    \end{equation}
\end{step}
\begin{step}
    Next we show that $\mu > 0$. By \eqref{lower}, we have that there exists $\rho_* > 0$ and $C>0$ such that $(\bar{\rho}e(\bar{\rho}))_{\bar{\rho}} \leq C \bar{\rho}^\frac{\alpha}{3}$ for all $\bar{\rho} \leq \rho_*$. There exists $R > 0$ such that
    \[
        \int_{B_R} \bar{\rho} \dd \boldsymbol{x} > \frac{M}{2}.
    \]
    Hence, we have that for $|\boldsymbol{x}|>R$
    \[
        - (\Phi_\alpha * \bar{\rho})(\boldsymbol{x}) \geq \frac{1}{\alpha} \int_{B_R} \bar{\rho}(\boldsymbol{y}) |\boldsymbol{x} - \boldsymbol{y}|^{-\alpha} \geq C(|\boldsymbol{x}| + R)^{-\alpha}.
    \]
    Thus, there exists $C_1>0$ such that for $|\boldsymbol{x}|$ large enough,
    \[
        - (\Phi_\alpha * \bar{\rho})(\boldsymbol{x}) \geq C_1 |\boldsymbol{x}|^{-\alpha}.
    \]
    Using the fact that $L$ is increasing, we have that there exists $C_2>0$ such that
    \[
        \int^\infty_r\frac{L(m_{\bar{\rho}}(s))}{s^3} \dd s \leq \int^\infty_r\frac{L(M)}{s^3} \dd s \leq C_2 r^{-2}.
    \]
    Combining this with \eqref{rotsteadynew}, we obtain for $\bar{\rho} \leq \rho_*$ and {\it a.e.} $|\boldsymbol{x}|$ large enough, that
    \begin{equation}\label{notintegrable}
        \begin{split}
            \bar{\rho} & \geq C\big ((\bar{\rho}e(\bar{\rho}))_{\bar{\rho}}\big )^\frac{3}{\alpha} \\
            & = C\bigg ( -\int^\infty_r\frac{L(m_{\bar{\rho}}(s))}{s^3} \dd s - \Phi_\alpha * \bar{\rho} - \mu \bigg )_+^\frac{3}{\alpha} \\
            & \geq C\big ( C_1|\boldsymbol{x}|^{-\alpha} - C_2r^{-2} -  \mu \big )_+^\frac{3}{\alpha}.
        \end{split}
    \end{equation}
    Since $\alpha < 2$, if $\mu \leq 0$, the final line of \eqref{notintegrable} is not integrable at the far field. Since $\bar{\rho} \in X_M$, we must have that
    \[
        |\{ \bar{\rho} > \rho_* \}| < \infty.
    \]
    Hence, this combined with \eqref{notintegrable}, implies that $\mu > 0$.
\end{step}
\begin{step}
    We now show that $\bar{\rho} \in L^\infty(\mathbb{R}^3)$. By \eqref{higherpress}, there exists $\rho^* > 0$ and $C>0$ such that $(\bar{\rho}e(\bar{\rho}))_{\bar{\rho}} \geq C \bar{\rho}^\frac{\alpha}{3}$ for all $\bar{\rho} \geq \rho^*$. Thus, since $\mu>0$, for $\bar{\rho} \geq \rho^*$ and {\it a.e.} on $\mathbb{R}^3$, we have that
    \begin{equation}\label{Linfty}
        \begin{split}
            \bar{\rho} & \leq C\big ((\bar{\rho}e(\bar{\rho}))_{\bar{\rho}}\big )^\frac{3}{\alpha} \\
            & = C\bigg ( -\int^\infty_r\frac{L(m_{\bar{\rho}}(s))}{s^3} \dd s - \Phi_\alpha * \bar{\rho} - \mu \bigg )_+^\frac{3}{\alpha} \\
            & \leq C\big ( - \Phi_\alpha * \bar{\rho} \, \big )^\frac{3}{\alpha}.
        \end{split}
    \end{equation}
    Since $\bar{\rho} \in L^1(\mathbb{R}^3)$, if $\bar{\rho} \in L^p(\mathbb{R}^3)$ for some $p > \frac{3}{3-\alpha}$, then by Lemma \ref{simplehls}, $\Phi_\alpha * \bar{\rho} \in C(\mathbb{R}^3) \cap L^\infty(\mathbb{R}^3)$. From \eqref{Linfty}, this would imply that $\bar{\rho} \in L^\infty(\mathbb{R}^3)$. Thus, we just need to show that $\bar{\rho} \in L^p(\mathbb{R}^3)$ for some $p > \frac{3}{3-\alpha}$. If this is not true, then since $\bar{\rho} \in X_M$, we have that $\bar{\rho} \in L^q(\mathbb{R}^3)$ for some $\frac{3+\alpha}{3} \leq q < \frac{3}{3-\alpha}$. From the HLS inequality, we have that $\Phi_\alpha * \bar{\rho} \in L^r(\mathbb{R}^3)$ for $r$ satisfying
    \[
        \frac{1}{r} = \frac{1}{q} - \frac{3-\alpha}{3}.
    \]
    Owing to \eqref{Linfty} and the fact that $\bar{\rho} \in L^1(\mathbb{R}^3)$, this implies that $\bar{\rho} \in L^{q'}(\mathbb{R}^3)$ for $q'$ satisfying
    \[
        \frac{1}{q'} = \frac{1}{q} - \frac{3-\alpha}{3}\Big(1-\frac{1}{q}\Big).
    \]
    Thus, if we repeat this process finitely many times, we obtain that $\bar{\rho} \in L^p(\mathbb{R}^3)$ for some $p > \frac{3}{3-\alpha}$ and we are done.
\end{step}
\begin{step}
    We now show that $\bar{\rho} \in C(\mathbb{R}^3)$. Since $\bar{\rho} \in X_M$ we have that $\int^\infty_r L(m_{\bar{\rho}}(s))s^{-3} \dd s$ is continuous for $r(\boldsymbol{x}) > 0$ and is continuous on $\mathbb{R}^3$ if
    \begin{equation}\label{boundedkin}
        \int^\infty_0 L(m_{\bar{\rho}}(s))s^{-3} \dd s < \infty.
    \end{equation}
    Suppose \eqref{boundedkin} is not true, then since $L$ is non-negative, there exists $r'>0$ such that if $0 \leq r(\boldsymbol{x})< r'$, then
    \[
        \int^\infty_r L(m_{\bar{\rho}}(s))s^{-3} \dd s > \|\Phi_\alpha * \bar{\rho}\|_{L^\infty(\mathbb{R}^3)} - \mu.
    \]
    Thus, by \eqref{rotsteadynew}, for {\it a.e.} $0 \leq r(\boldsymbol{x})< r'$, we have that
    \[
        (\bar{\rho}e(\bar{\rho}))_{\bar{\rho}} \leq \bigg ( -\int^\infty_r\frac{L(m_{\bar{\rho}}(s))}{s^3} \dd s + \|\Phi_\alpha * \bar{\rho}\|_{L^\infty(\mathbb{R}^3)} - \mu \bigg )_+ = 0.
    \]
    Thus, $\bar{\rho} = 0$ for {\it a.e.} $0 \leq r(\boldsymbol{x})< r'$. Hence, since $L(0)=0$,
    \[
        \int^\infty_0 L(m_{\bar{\rho}}(s))s^{-3} \dd s = \int^\infty_{r'} L(m_{\bar{\rho}}(s))s^{-3} \dd s < \infty.
    \]
    Thus, \eqref{boundedkin} must be satisfied. Hence, $\int^\infty_r L(m_{\bar{\rho}}(s))s^{-3} \dd s$ is continuous on $\mathbb{R}^3$. Since $(\tau e(\tau))_{\tau}$ is continuous and increasing on $\tau \geq 0$, we must have that it is invertible with the inverse being continuous. Since the right-hand side of \eqref{rotsteadynew} is continuous, $\bar{\rho} \in C(\mathbb{R}^3)$ and is a rotating Riesz star solution as given by Definition \ref{rotatingriesz}.
\end{step}
\begin{step}
        We now show that $\bar{\rho}$ has compact support. Since $\bar{\rho} \in L^1(\mathbb{R}^3) \cap L^\infty(\mathbb{R}^3)$, we have that $(\Phi_\alpha * \bar{\rho})(\boldsymbol{x}) \to 0$ uniformly as $|\boldsymbol{x}| \to \infty$. Rearranging \eqref{stationaxi} and using the fact that $L$ is non-negative, we obtain that
        \[
            \begin{split}
                (\bar{\rho}e(\bar{\rho}))_{\bar{\rho}} & = \bigg ( -\int^\infty_r\frac{L(m_{\bar{\rho}}(s))}{s^3} \dd s - \Phi_\alpha * \bar{\rho} - \mu \bigg )_+ \\
                & \leq (- \Phi_\alpha * \bar{\rho} - \mu)_+.
            \end{split}
        \]
        Thus, since $\mu > 0$, we obtain that there exists $R>0$ such that $\Gamma \subset B_R$. Hence, $\bar{\rho}$ has compact support.
    \end{step}
\end{proof}
\setcounter{step}{0}

\section{Stability in the mass-subcritical regime}\label{stability:rota}

To establish stability, we consider axisymmetric solutions $(\rho,\boldsymbol{\mathcal{M}})$ of \eqref{0.0}, in the sense of \eqref{axiboy}, whose initial rotational momentum is of the form
\[
\mathcal{M}^\theta_0(\boldsymbol{x}) = \frac{\rho_0(\boldsymbol{x}) J(m_{\rho_0}(r))}{r}\boldsymbol.
\]
Under assumptions (A$_1$)--(A$_4$), the rotational momentum field takes the form
\[
\mathcal{M}^\theta(t,\boldsymbol{x}) = \frac{\rho(t,\boldsymbol{x}) J(m_{\rho(t)}(r))}{r}.
\]

\smallskip

We can now prove the stability of rotating Riesz star solutions.

     \begin{proof}[Proof of Theorem \ref{nsa}]
        Suppose the statement in Theorem \ref{nsa} is not true, then there exist $\varepsilon > 0$ and a sequence of global weak solutions $(\rho^j,\boldsymbol{\mathcal{M}}^j)$ of the CEREs associated to the initial data $(\rho_0^j,\boldsymbol{\mathcal{M}}_0^j)$ and a sequence of times $t_j > 0$ such that
\begin{align}\label{contradictaxi}
d_1(\rho^j_0,\bar{\rho})+ d_2(\rho^j_0,\bar{\rho})+ \| \rho^j_0 - \bar{\rho} \|_{L^\frac{6}{6 - \alpha}}^2
+\frac{1}{2}\int_{\mathbb{R}^3}\bigg (\bigg|\frac{\mathcal{M}^j_{r,0}}{\sqrt{\rho^j_0}}\bigg|^2 + \bigg|\frac{\mathcal{M}^j_{3,0}}{\sqrt{\rho^j_0}}\bigg|^2 \bigg )\,\dd \boldsymbol{x}<\frac{1}{j},
\end{align}
and
\begin{equation}\label{auxxaxi}
d_1(\rho^j(t_j),T^{\boldsymbol{y}}\bar{\rho}) + d_2(\rho^j(t_j),T^{\boldsymbol{y}}\bar{\rho})
+ \| \rho^j(t_j) - T^{\boldsymbol{y}} \bar{\rho} \|_{L^\frac{6}{6 - \alpha}}^2 +\frac{1}{2}\int_{\mathbb{R}^3}\bigg (\bigg|\frac{\mathcal{M}^j_{r}}{\sqrt{\rho^j}}\bigg|^2 + \bigg|\frac{\mathcal{M}^j_{3}}{\sqrt{\rho^j}}\bigg|^2 \bigg )(t_j) \,\dd \boldsymbol{x} \geq \v,
\end{equation}
for all $\boldsymbol{y} \in \{\boldsymbol{0}\} \times \mathbb{R}$. Notice that the sequence of times $t_j$ can be chosen outside a zero measure set, where the terms in \eqref{auxxaxi} are well-defined according to Definition \ref{definition}. By Lemma \ref{d2}, we have that $d_2(\rho(t),\bar{\rho}) \geq 0$ for all $t \geq 0$. Then, by \eqref{differenceaxi} and \eqref{contradictaxi}, we have
\begin{equation*}
    \lim_{j \to \infty} E(\rho^j_0,\boldsymbol{\mathcal{M}}^j_0) = E(\bar{\rho},\midbar{\boldsymbol{\mathcal{M}}}) = \mathcal{G}(\bar{\rho}).
\end{equation*}
Since the energy $E$ for the weak solutions is non-increasing from the initial energy by definition, we have 
\begin{equation*}
    \limsup_{j \to \infty} \mathcal{G}(\rho^j(t_j)) \leq \limsup_{j \to \infty} E(\rho^j(t_j),\boldsymbol{\mathcal{M}}^j(t_j)) \leq \lim_{j \to \infty} E(\rho^j_0,\boldsymbol{\mathcal{M}}^j_0) = \mathcal{G}(\bar{\rho}).
\end{equation*}
Thus, $\rho^j(t_j) \in X_M$ is a minimising sequence of $\mathcal{G}$ in $X_M$ and, by Theorem \ref{minimiseraxi}, there exists a sequence $\boldsymbol{y}_j \in \{\boldsymbol{0}\} \times \mathbb{R}$ such that $T^{\boldsymbol{y}_j}\rho^j(t_j) \to \bar{\rho}$ in $L^1(\mathbb{R}^3) \cap L^\frac{3 + \alpha}{3}(\mathbb{R}^3)$. Then
\begin{equation*}
     \| \rho^j(t_j) - T^{-\boldsymbol{y}_j}\bar{\rho} \|_{L^\frac{6}{6 - \alpha}} = \| T^{\boldsymbol{y}_j}\rho^j(t_j) - \bar{\rho} \|_{L^\frac{6}{6 - \alpha}}  
     \longrightarrow 0,
\end{equation*}
as $j \to \infty$, since $\frac{6}{6 - \alpha} \in (1, \frac{3 + \alpha}{3})$. 
We see that, for any $i$, 
\begin{equation*}
 \lim_{j \to \infty} E(\rho^j(t_j),\boldsymbol{\mathcal{M}}^j(t_j)) = \mathcal{G}(\bar{\rho}) = \mathcal{G}(T^{-\boldsymbol{y}_i}\bar{\rho}).
\end{equation*}
Therefore, by \eqref{differenceaxi}, 
\begin{equation*}
    d_1(\rho^j(t_j),T^{-\boldsymbol{y}_j}\bar{\rho}) + d_2(\rho^j(t_j),T^{-\boldsymbol{y}_j}\bar{\rho})
    + \frac{1}{2}\int_{\mathbb{R}^3}\bigg (\bigg|\frac{\mathcal{M}^j_{r}}{\sqrt{\rho^j}}\bigg|^2 + \bigg|\frac{\mathcal{M}^j_{3}}{\sqrt{\rho^j}}\bigg|^2 \bigg )(t_j) \,\dd \boldsymbol{x} \longrightarrow 0,
\end{equation*}
as $j \to \infty$. This is a contradiction to \eqref{auxxaxi}, which completes the proof.
     \end{proof}

    \begin{remark}
        Combining Theorem \ref{lmulaxi} with {\rm\cite[Theorem 2.6.]{Carrillo_2026}}, we conclude that Theorem \ref{nsa} is, in fact, a local stability result.
    \end{remark}

    \begin{remark}
        As in {\rm\cite[Corollary 5.9]{Carrillo2025}}, one can also establish an analogous stability result in the case where rotating Riesz star solutions are not unique (up to translation).
    \end{remark}

     \section{Existence in the mass-supercritical regime}\label{existence:rotaingsuper}
    In this section we prove existence of rotating Riesz star solutions in the polytropic mass-supercritical case $(\frac{6}{6-\alpha}<\gamma <\frac{3+\alpha}{3})$. Given $\varrho \in L^1_+(\R^3) \cap L^\gamma(\mathbb{R}^3)$,  and $p(\varrho) = a_0 \varrho^\gamma$, with $\frac{6}{6 - \alpha} < \gamma < \frac{3 + \alpha}{3}$, we recall the definition of the free-energy functional $Q(\varrho)$:
    \[
    Q(\varrho) = \int_{\mathbb{R}^3} \frac{\varrho L(m_{\varrho}(r))}{r^2} \dd \boldsymbol{x} + 3 a_0 \int_{\mathbb{R}^3} \varrho^\gamma \dd \boldsymbol{x} + \frac{\alpha}{2} \int_{\mathbb{R}^3} \varrho \Phi_\alpha * \varrho \dd \boldsymbol{x}.
    \]
    This quantity will play a fundamental role in the subsequent analysis. We also recall the associated free-energy functional $S_\mu(\varrho)$:
    \[
    S_\mu(\varrho) = \frac{1}{2} \int_{\mathbb{R}^3} \frac{\varrho L(m_{\varrho}(r))}{r^2} \dd \boldsymbol{x} + \frac{a_0}{\gamma - 1} \int_{\mathbb{R}^3} \varrho^\gamma \dd \boldsymbol{x} + \frac{1}{2} \int_{\mathbb{R}^3} \varrho \Phi_\alpha * \varrho \dd \boldsymbol{x} + \mu \int_{\mathbb{R}^3} \varrho \dd \boldsymbol{x},
    \]
    for $\mu>0$. In the axisymmetric setting, under the assumption of purely rotational velocity, the functional $S_\mu(\varrho)$ represents the sum of $\mu$ times the mass of $\varrho$ and the total energy associated with the density distribution $\varrho$ and velocity field
    \[
        \boldsymbol{u}(\boldsymbol{x}) = \frac{J(m_\varrho(r))}{r} \boldsymbol{e}_\theta,
    \]
    with $J^2 = L$. Defining $S_{\mu,\lambda}(\varrho) \coloneq S_\mu(\varrho_\lambda)$, we have that
    \[
        S_{\mu,\lambda}(\varrho) \coloneq \frac{\lambda}{2} \int_{\mathbb{R}^3} \frac{\varrho L(m_{\varrho}(r))}{r^2} \dd \boldsymbol{x} + \frac{a_0 \lambda^\frac{3(\gamma - 1)}{2}}{\gamma - 1} \int_{\mathbb{R}^3} \varrho^\gamma \dd \boldsymbol{x} + \frac{\lambda^\frac{\alpha}{2}}{2} \int_{\mathbb{R}^3} \varrho \Phi_\alpha * \varrho \dd \boldsymbol{x} + \mu \int_{\mathbb{R}^3} \varrho \dd \boldsymbol{x},
    \]
    and
    \begin{align*}
    \frac{\d S_{\mu,\lambda}(\varrho)}{\d \lambda} &= \frac{1}{2} \int_{\mathbb{R}^3} \frac{\varrho L(m_{\varrho}(r))}{r^2} \dd \boldsymbol{x} + \frac{3 a_0 \lambda^{\frac{3\gamma - 5}{2}}}{2} \int_{\mathbb{R}^3} \varrho^\gamma \dd \boldsymbol{x} + \frac{\alpha \lambda^{\frac{\alpha - 2}{2}}}{4} \int_{\mathbb{R}^3} \varrho \Phi_\alpha * \varrho \dd \boldsymbol{x} \\ 
    & = \frac{Q(\varrho_\lambda)}{2\lambda}.
    \end{align*}
    The analysis is substantially more involved than in the stationary setting ($L \equiv 0$) considered in \cite{Carrillo_2026}, due to the interaction between the rotational kinetic energy, the internal energy, and the potential energy. In the stationary state case, for every $\varrho \in L^1_+ \cap L^\gamma$, there exists a unique scaling parameter $\lambda^* = \lambda^*(\varrho) > 0$ such that $Q(\varrho_{\lambda^*}) = 0$. However, this property no longer necessarily holds once the rotational kinetic energy is introduced. The possible behaviours are characterised in the following two lemmas.
    
    \begin{lemma}\label{largealpha}
        Suppose $\alpha \in [2,3)$, $\frac{6}{6 - \alpha} < \gamma < \frac{3 + \alpha}{3}$, $L$ satisfies \eqref{Lnice2} and $\varrho \in L^1_+(\mathbb{R}^3) \cap L^\gamma(\mathbb{R}^3)$, $\varrho \not\equiv 0$, 
        $$\int_{\mathbb{R}^3} \frac{\varrho L(m_{\varrho}(r))}{r^2} \dd \boldsymbol{x} < \infty,$$
        with, for $\alpha = 2$,
        $$\int_{\mathbb{R}^3} \frac{\varrho L(m_{\varrho}(r))}{r^2} \dd \boldsymbol{x} < -\int_{\mathbb{R}^3} \varrho \Phi_\alpha * \varrho \dd \boldsymbol{x},$$
        then the following statements are true{\rm :}
    \begin{enumerate}
        \item[{\rm(}a{\rm)}] There exists a unique $\lambda^*(\varrho) > 0$ such that $Q(\varrho_{\lambda^*(\varrho)}) = 0$.
        \item[{\rm(}b{\rm)}] $\lambda^*(\varrho) > 1 \iff Q(\varrho) > 0$.
        \item[{\rm(}c{\rm)}] $\lambda \mapsto S_{\mu,\lambda}(\varrho)$ is strictly increasing {\rm(}decreasing{\rm)} for $\lambda \leq \lambda^*(\varrho)$ {\rm(}$\lambda \geq \lambda^*(\varrho)${\rm)}. Moreover, $S_{\mu,\lambda}(\varrho) < S_{\mu,\lambda^*(\varrho)}(\varrho)$, for all $\lambda \neq \lambda^*(\varrho)$ and $\lambda > 0$.
    \end{enumerate}
    \end{lemma}
    \begin{proof}
    \begin{case}[$\alpha \in (2,3)$]
        Since $\alpha > 2,3(\gamma -1)$ and $\int_{\mathbb{R}^3} \varrho \Phi_\alpha * \varrho \dd \boldsymbol{x} < 0$, we have that
        \[
            Q(\varrho_\lambda) = \lambda \int_{\mathbb{R}^3} \frac{\varrho L(m_{\varrho}(r))}{r^2} \dd \boldsymbol{x} + 3 a_0 \lambda^{\frac{3(\gamma-1)}{2}} \int_{\mathbb{R}^3} \varrho^\gamma \dd \boldsymbol{x} + \frac{\alpha \lambda^\frac{\alpha}{2}}{2}  \int_{\mathbb{R}^3} \varrho \Phi_\alpha * \varrho \dd \boldsymbol{x} \longrightarrow - \infty,
        \]
        as $\lambda \to \infty$. Similarly, due to the positivity of the kinetic energy and the internal energy, we obtain that
        \[
            \frac{Q(\varrho_\lambda)}{\lambda^\frac{\alpha}{2}} \longrightarrow \infty, \qquad \mbox{as $\lambda \to 0$.}
        \]
        Thus, since $\lambda \mapsto Q(\varrho_\lambda)$ is continuous, there exists $\lambda^*=\lambda^*(\varrho) > 0$ such that $Q(\varrho_{\lambda^*})=0$. 
        \end{case}
        \begin{case}[$\alpha = 2$]
        Similar to the previous case, as $2 > 3(\gamma-1)$ and 
        $$\int_{\mathbb{R}^3} \frac{\varrho L(m_{\varrho}(r))}{r^2} \dd \boldsymbol{x} < -\int_{\mathbb{R}^3} \varrho \Phi_\alpha * \varrho \dd \boldsymbol{x},$$
        we obtain that
        \[
            Q(\varrho_\lambda) \longrightarrow - \infty, \qquad \mbox{as $\lambda \to \infty$,}
        \]
        and
        \[
            \frac{Q(\varrho_\lambda)}{\lambda} \longrightarrow \infty, \qquad \mbox{as $\lambda \to 0$.}
        \]
        Thus, there exists $\lambda^*=\lambda^*(\varrho) > 0$ such that $Q(\varrho_{\lambda^*})=0$.
    \end{case}

        Now, for $\alpha \in [2,3)$, calculating the derivative of $\lambda \mapsto Q(\rho_\lambda)$, we see that
        \[
        \begin{split}
            \frac{\d Q(\varrho_\lambda)}{\d \lambda} & = \int_{\mathbb{R}^3} \frac{\varrho L(m_{\varrho}(r))}{r^2} \dd \boldsymbol{x} + \frac{9 a_0 (\gamma - 1) \lambda^{\frac{3\gamma-5}{2}}}{2} \int_{\mathbb{R}^3} \varrho^\gamma \dd \boldsymbol{x} + \frac{\alpha^2 \lambda^{\frac{\alpha-2}{2}}}{4}  \int_{\mathbb{R}^3} \varrho \Phi_\alpha * \varrho \dd \boldsymbol{x} \\
            & = \frac{\alpha Q(\rho_\lambda)}{2\lambda} + \Big(\frac{2 - \alpha}{2}\Big)\int_{\mathbb{R}^3} \frac{\varrho L(m_{\varrho}(r))}{r^2} \dd \boldsymbol{x} + \frac{3 a_0( 3\gamma - 3  - \alpha)}{2} \lambda^{\frac{3\gamma-5}{2}} \int_{\mathbb{R}^3} \varrho^\gamma \dd \boldsymbol{x}.
        \end{split}
        \]
        Hence,
        \begin{align*}
        \frac{\d Q(\varrho_\lambda)}{\d \lambda} \bigg |_{\lambda = \lambda^*} & = \Big(\frac{2-\alpha}{2}\Big) \int_{\mathbb{R}^3} \frac{\varrho L(m_{\varrho}(r))}{r^2} \dd \boldsymbol{x} + \frac{3 a_0( 3\gamma - 3  - \alpha)}{2} (\lambda^*)^{\frac{3(\gamma-1)-2}{2}} \int_{\mathbb{R}^3} \varrho^\gamma \dd \boldsymbol{x} \\
        & < 0.
        \end{align*}
        Thus, $\lambda^*$ must be unique with $Q(\rho_\lambda) > 0$ for $0<\lambda<\lambda^*$ and $Q(\rho_\lambda) < 0$ for $\lambda>\lambda^*$. Thus, (a), (b), and (c) follow.
    \end{proof}
    \setcounter{case}{0}

    However, when $\alpha \in (0,2)$, uniqueness of the scaling parameter may fail; in general, there need not exist a unique $\lambda^*>0$ satisfying $Q(\varrho_{\lambda^*}) = 0$.
    
    \begin{lemma}\label{smallalpha}
        Suppose $\alpha \in (0,2)$, $\frac{6}{6 - \alpha} < \gamma < \frac{3 + \alpha}{3}$, $L$ satisfies \eqref{Lnice2} and $\varrho \in L^1_+(\mathbb{R}^3) \cap L^\gamma(\mathbb{R}^3)$, $\varrho \not\equiv 0$, 
        $$\int_{\mathbb{R}^3} \frac{\varrho L(m_{\varrho}(r))}{r^2} \dd \boldsymbol{x} < \infty,$$
        with $Q(\varrho_\lambda) \leq 0$ for some $\lambda > 0$, then the following statements are true{\rm :}
    \begin{enumerate}
        \item[{\rm(}a{\rm)}] There exists a $\lambda^*(\varrho)> 0$ such that $Q(\varrho_{\lambda^*(\varrho)}) = 0$. Moreover, there are at most two such $\lambda^*(\varrho)$, label these $0 < \lambda_1^*(\varrho) \leq \lambda_2^*(\varrho) < \infty$.
        \item[{\rm(}b{\rm)}] $\lambda_1^*(\varrho) < 1 < \lambda_2^*(\varrho) \iff Q(\varrho) < 0$.
        \item[{\rm(}c{\rm)}] $\lambda \mapsto S_{\mu,\lambda}(\varrho)$ is strictly increasing {\rm(}decreasing{\rm)} for $\lambda \leq \lambda^*_1(\varrho)$ and $\lambda \geq \lambda^*_2(\varrho)$ {\rm(}$\lambda_1^*(\varrho) \leq \lambda \leq \lambda_2^*(\varrho)${\rm)}.
        \item[{\rm(}d{\rm)}] $\lambda_1^*(\varrho) \leq \kappa(\varrho) \leq \lambda_2^*(\varrho)$ with $Q(\varrho_\kappa) \leq 0$, where
        \begin{equation}\label{kappa}
            \kappa = \kappa(\varrho) \coloneq \bigg( \frac{6 a_0 (5 - 3\gamma )\int_{\mathbb{R}^3}\varrho^\gamma \dd \boldsymbol{x}}{\alpha (\alpha-2) \int_{\mathbb{R}^3} \varrho \Phi_\alpha * \varrho \dd \boldsymbol{x}} \bigg)^\frac{2}{3 + \alpha - 3\gamma},
        \end{equation}
        with equality in either of the inequalities if and only if $\lambda_1^*(\varrho) = \lambda_2^*(\varrho)$.
    \end{enumerate}
    \end{lemma}
    \begin{proof}
    We split the proof into two steps.
    \begin{step}
        Since $3(\gamma-1) < \alpha < 2$, we obtain that 
        \[
            Q(\varrho_\lambda) \longrightarrow \infty, \qquad \mbox{as $\lambda \to \infty$,}
        \]
        and
        \[
            \frac{Q(\varrho_\lambda)}{\lambda} \longrightarrow \infty, \qquad \mbox{as $\lambda \to 0$.}
        \]
        Thus, since there exists $\lambda > 0$ such that $Q(\varrho_\lambda) \leq 0$, by continuity of $\lambda \mapsto Q(\varrho_\lambda)$, there exists $\lambda^*(\varrho) > 0$ such that $Q(\varrho_{\lambda^*(\varrho)}) = 0$. Now, considering the derivative of $\lambda \mapsto Q(\varrho_\lambda)$, as before, we see that
        \[
            \frac{\d Q(\varrho_\lambda)}{\d \lambda} = \frac{Q(\rho_\lambda)}{\lambda} + \frac{3 a_0( 3\gamma - 5)}{2} \lambda^{\frac{3(\gamma-1)-2}{2}} \int_{\mathbb{R}^3} \varrho^\gamma \dd \boldsymbol{x} + \frac{\alpha(\alpha - 2)}{4} \lambda^{\frac{\alpha - 2}{2}} \int_{\mathbb{R}^3} \varrho \Phi_\alpha * \varrho \dd \boldsymbol{x}.
        \]
        Thus, we obtain that
        \[
            \frac{\d Q(\varrho_\lambda)}{\d \lambda} \bigg |_{\lambda = \lambda^*} = \frac{3a_0( 3\gamma - 5)}{2} (\lambda^*)^{\frac{3(\gamma-1)-2}{2}} \int_{\mathbb{R}^3} \varrho^\gamma \dd \boldsymbol{x} + \frac{\alpha(\alpha - 2)}{4} (\lambda^*)^{\frac{\alpha - 2}{2}} \int_{\mathbb{R}^3} \varrho \Phi_\alpha * \varrho \dd \boldsymbol{x}.
        \]
        Hence, we see that
        \begin{equation}\label{differentbehave}
        \begin{cases}
            \frac{\d Q(\varrho_\lambda)}{\d \lambda} \big |_{\lambda = \lambda^*} > 0, & \text{ for } \lambda^* > \kappa, \\
            \frac{\d Q(\varrho_\lambda)}{\d \lambda} \big |_{\lambda = \lambda^*} = 0, & \text{ for } \lambda^* = \kappa, \\
            \frac{\d Q(\varrho_\lambda)}{\d \lambda} \big |_{\lambda = \lambda^*} < 0, & \text{ for } \lambda^* < \kappa,
        \end{cases}
        \end{equation}
        where $\kappa$ is defined in \eqref{kappa}. Thus, there can be at most one $\lambda^* > 0$ in each of $\{\lambda > \kappa \}$, $\{\lambda = \kappa \}$ and $\{\lambda < \kappa \}$ such that $Q(\varrho_{\lambda^*}) = 0$. Suppose that all three sets contain a vanishing point of $\lambda \mapsto Q(\varrho_\lambda)$, label these $0 < \lambda_1^* < \lambda^*_2 < \lambda^*_3$, where $\lambda^*_2 = \kappa$. Then, due to \eqref{differentbehave}, we see that $Q(\varrho_\lambda) \leq 0$ for all $\lambda_1^* < \lambda < \lambda^*_3$. However, we see that
        \[
            \frac{\d^2 Q(\varrho_\lambda)}{\d \lambda^2} \bigg |_{\lambda = \lambda_2^*} = \frac{3a_0 (\lambda^*_2)^{\frac{3\gamma - 7}{2}}}{2} \int_{\mathbb{R}^3}\varrho^\gamma \dd \boldsymbol{x} > 0.
        \]
        This implies by Taylor's Theorem that there exists $\lambda_2^* < \lambda < \lambda^*_3$ such that $Q(\varrho_\lambda) > 0$. A contradiction. Hence, (a), (b) and (c) follow.
        \end{step}
        \begin{step}
        Suppose $Q(\varrho_{\kappa(\varrho)}) = 0$. Considering, the map $\lambda \mapsto \frac{Q(\varrho_\lambda)}{2\lambda} = \frac{\d S_{\mu,\lambda}(\varrho)}{\d \lambda}$, looking at the derivative, we see that
        \begin{equation}\label{secondderi}
            \frac{\d^2 S_{\mu,\lambda}(\varrho)}{\d \lambda^2} = \frac{3 a_0 (3\gamma - 5) }{4} \lambda^{\frac{3\gamma - 7}{2}} \int_{\mathbb{R}^3} \varrho^\gamma \dd \boldsymbol{x} + \frac{\alpha (\alpha - 2)}{8} \lambda^{\frac{\alpha - 4}{2}}\int_{\mathbb{R}^3} \varrho \Phi_\alpha * \varrho \dd \boldsymbol{x}.
        \end{equation}
    Hence, $\lambda =  \kappa(\varrho)$ is the only minimum of $\lambda \mapsto \frac{Q(\varrho_\lambda)}{2\lambda}$. Thus, we see that $Q(\varrho_\lambda) > 0$ for $\lambda \neq \kappa(\varrho)$. Hence, (d) follows.
        \end{step}
    \end{proof}
    \setcounter{step}{0}

    \begin{remark}
        In the proof of Lemma \ref{smallalpha}, it is important that we use the mass-preserving scaling $\varrho_\lambda(\boldsymbol{x}) = \lambda^\frac{3}{2}\varrho(\lambda^\frac{1}{2}\boldsymbol{x})$ rather than $\lambda\varrho(\lambda^\frac{1}{3}\boldsymbol{x})$, since, under the former scaling, the kinetic term disappears from the second derivative of $S_{\mu,\lambda}$. This cancellation is essential for characterising the behaviour of the critical points $\lambda_1^*(\varrho)$ and $\lambda_2^*(\varrho)$.
    \end{remark}

    Thus, for $\alpha \in (0,2)$, considering the behaviour of $S_{\mu,\lambda}(\varrho)$, we see that if there exists $\lambda > 0$ such that $Q(\varrho_\lambda) < 0$, then $S_{\mu,\lambda^*_1}(\varrho)$ is a local maximum of $\lambda \mapsto S_{\mu,\lambda}(\varrho)$, while $S_{\mu,\lambda^*_2}(\varrho)$ is a local minimum of $\lambda \mapsto S_{\mu,\lambda}(\varrho)$. Moreover, it is not clear whether $S_{\mu,\lambda}(\varrho)$ remains bounded below on the set of $\varrho$ such that $\lambda^*_2(\varrho) = 1$. As in \cite{Carrillo_2026}, we prove existence when $S_{\mu,\lambda^*}(\varrho)$ is a local maximum or inflection point of $\lambda \mapsto S_{\mu,\lambda}(\varrho)$. As such, for $\alpha \in (0,3)$ we look to minimise $S_\mu$ on the following admissible set
    \[
        \mathcal{K} \coloneq \{ \varrho \in L^1_+(\mathbb{R}^3) \cap L_{\rm axi}^\gamma(\mathbb{R}^3) \, | \, Q(\varrho) = 0, \, Q(\varrho_\lambda) > 0 \text{ for all } 0 < \lambda < 1, \,  \varrho \not\equiv 0\}.
    \]
    For $\alpha \in (0,2)$, the admissible set $\mathcal{K}$ is equivalent to $\lambda^*_1(\varrho) = 1$. For $\alpha \in (0,2]$, to show that this set is non-empty, we require that $L$ is superhomogeneous of order $\omega > \omega_*$, where $\omega_*$ is given by \eqref{omegastar}, that is
\[
L(am) \leq a^\omega L(m),
\]
for all $0 < a < 1$ and $m \geq 0$.

    \begin{lemma}\label{nonempty}
        Suppose $\alpha \in (0,3)$, $\frac{6}{6-\alpha} < \gamma < \frac{3 + \alpha}{3}$, $L$ satisfies \eqref{Lnice2} and, for $\alpha \in (0,2]$, \eqref{Linstab} for $\omega > \omega^*$, then $\mathcal{K} \neq \emptyset$.
    \end{lemma}

    \begin{proof}
    Take $\varrho = \mathds{1}_{\{1 \leq r(\boldsymbol{x}) \leq 2\} \times [-1,1]} \in L^1_+(\mathbb{R}^3) \cap L^\gamma(\mathbb{R}^3)$, we have that 
    $$\int_{\mathbb{R}^3} \frac{\varrho L(m_{\varrho}(r))}{r^2} \dd \boldsymbol{x} < \infty.$$
    \begin{case}[$\alpha \in (2,3)$]
        By Lemma \ref{largealpha}, there exists $\lambda^* > 0$ such that $\varrho_{\lambda^*} \in \mathcal{K}$.
    \end{case}
    \begin{case}[$\alpha \in (0,2)$]
        By considering $\varrho_{\kappa(\varrho)}$, we may assume $\kappa(\varrho) = 1$. We employ the energy-critical scaling as defined in \eqref{energycritical}, $$\leftindex^\xi{\varrho}(\boldsymbol{x}) = \xi^{3-\alpha} \rho\big(\xi^{2-\gamma} \boldsymbol{x}\big),$$ 
    for $\xi > 0$. We consider how $Q(\leftindex^\xi{\varrho})$ scales with $\xi > 0$. Applying a change of variables, one obtains that  $\kappa(\leftindex^\xi\varrho) = \kappa(\varrho) = 1$ and further that
    \begin{align}\label{usingk}
            3 a_0 \int_{\mathbb{R}^3} \leftindex^\xi\varrho^\gamma \dd \boldsymbol{x} + \frac{\alpha}{2} \int_{\mathbb{R}^3} \leftindex^\xi\varrho \Phi_\alpha * \leftindex^\xi\varrho \dd \boldsymbol{x} & = \xi^{(6-\alpha)\gamma - 6} \bigg (3 a_0 \int_{\mathbb{R}^3} \varrho^\gamma \dd \boldsymbol{x} + \frac{\alpha}{2} \int_{\mathbb{R}^3} \varrho \Phi_\alpha * \varrho \dd \boldsymbol{x} \bigg ) \nonumber \\
            & = - 3a_0 \xi^{(6-\alpha)\gamma - 6} \frac{3 + \alpha - 3\gamma}{2-\alpha} \int_{\mathbb{R}^3} \varrho^\gamma \dd \boldsymbol{x}.
    \end{align}
    Employing \eqref{Linstab} and a change of variables, we see that for $\xi \geq 1$
    \begin{equation}\label{kinetbound}
    \begin{split}
        \int_{\mathbb{R}^3} \frac{\leftindex^\xi\varrho L(m_{\leftindex^\xi\varrho}(r))}{r^2} \dd \boldsymbol{x} & = \xi^{\gamma + 1 - \alpha}\int_{\mathbb{R}^3} \frac{\varrho L\big(\xi^{3 \gamma - 3 - \alpha} m_{\varrho}(r)\big)}{r^2} \dd \boldsymbol{x} \\
        & \leq \xi^{\gamma + 1 - \alpha - \omega(3+\alpha - 3\gamma)} \int_{\mathbb{R}^3} \frac{\varrho L(m_{\varrho}(r))}{r^2} \dd \boldsymbol{x}.
    \end{split}
    \end{equation}
    Thus, combining, we see that for $\xi \geq 1$
    \begin{equation}\label{Qineq}
        Q(\leftindex^\xi{\varrho}) \leq \xi^{\gamma + 1 - \alpha - \omega(3+\alpha - 3\gamma)} \int_{\mathbb{R}^3} \frac{\varrho L(m_{\varrho}(r))}{r^2} \dd \boldsymbol{x} - 3a_0 \xi^{(6-\alpha)\gamma - 6} \frac{3 + \alpha - 3\gamma}{2-\alpha} \int_{\mathbb{R}^3} \varrho^\gamma \dd \boldsymbol{x},
    \end{equation}
    with equality for $\xi = 1$. Then, since $\omega > \omega_*$, we have that
    \[
        \gamma + 1 - \alpha - \omega(3+\alpha - 3\gamma) < (6-\alpha)\gamma - 6.
    \]
    Hence, from \eqref{Qineq}, we see that
        \[
            Q(\leftindex^\xi{\varrho}) \longrightarrow - \infty, \qquad \mbox{as $\xi \to \infty$.}
        \]
        Thus, there exists $\xi > 0$ such that $Q(\leftindex^\xi{\varrho}) \leq 0$ and so by Lemma \ref{smallalpha}, there exists $\lambda^*_1 = \lambda^*_1(\leftindex^\xi\varrho) > 0$, such that $\leftindex^\xi{\varrho}_{\lambda^*_1} \in \mathcal{K}$.
    \end{case}
    \begin{case}[$\alpha = 2$]
        As in Case 2, we have that
        \begin{align*}
            \int_{\mathbb{R}^3} \frac{\leftindex^\xi\varrho L(m_{\leftindex^\xi\varrho}(r))}{r^2} \dd \boldsymbol{x} + \frac{\alpha}{2}\int_{\mathbb{R}^3} \leftindex^\xi\varrho \Phi_\alpha * \leftindex^\xi\varrho \dd \boldsymbol{x} & \leq \xi^{\gamma + 1 - \alpha - \omega(3+\alpha - 3\gamma)} \int_{\mathbb{R}^3} \frac{\varrho L(m_{\varrho}(r))}{r^2} \dd \boldsymbol{x} \\
            & \qquad + \xi^{(6-\alpha)\gamma - 6} \frac{\alpha}{2} \int_{\mathbb{R}^3} \varrho \Phi_\alpha * \varrho \dd \boldsymbol{x}.
        \end{align*}
        Hence, we see that
        \[
            \int_{\mathbb{R}^3} \frac{\leftindex^\xi\varrho L(m_{\leftindex^\xi\varrho}(r))}{r^2} \dd \boldsymbol{x} + \frac{\alpha}{2}\int_{\mathbb{R}^3} \leftindex^\xi\varrho \Phi_\alpha * \leftindex^\xi\varrho \dd \boldsymbol{x} \longrightarrow - \infty, \qquad \mbox{as $\xi \to \infty$.}
        \]
        Thus, there exists $\xi > 0$ such that
        \[
             \int_{\mathbb{R}^3} \frac{\leftindex^\xi\varrho L(m_{\leftindex^\xi\varrho}(r))}{r^2} \dd \boldsymbol{x} < - \frac{\alpha}{2}\int_{\mathbb{R}^3} \leftindex^\xi\varrho \Phi_\alpha * \leftindex^\xi\varrho \dd \boldsymbol{x}.
        \]
        Hence by Lemma \ref{largealpha}, there exists $\lambda^* = \lambda^*(\leftindex^\xi\varrho) > 0$, such that $\leftindex^\xi{\varrho}_{\lambda^*} \in \mathcal{K}$.
    \end{case}
    \end{proof}
    \setcounter{case}{0}

    Having established that the admissible set is non-empty for angular momentum functions $L$ satisfying \eqref{Lnice2} and \eqref{Linstab} with $\omega > \omega_*$, we now define
    \[
        \ell_\mu \coloneq \inf_{\varrho \in \mathcal{K}} S_\mu(\varrho),
    \]
    the minimised free-energy $S_\mu$ over the admissible set $\mathcal{K}$. We show that $\ell_\mu \geq 0$.

    \begin{lemma}
        Suppose $\alpha \in (0,3)$, $\frac{6}{6 - \alpha} < \gamma < \frac{3 + \alpha}{3}$, $L$ satisfies \eqref{Lnice2} and \eqref{Linstab} for $\omega > \omega_*$, then $\ell_\mu \geq 0$.
    \end{lemma}

    \begin{proof}
        \begin{case}[$\alpha \in [2,3)$]
            Let $\varrho \in \mathcal{K}$, then
            \begin{equation}\label{biggerthan01}
                S_\mu(\varrho) = \frac{Q(\varrho)}{\alpha} + \frac{\alpha - 2}{2 \alpha} \int_{\mathbb{R}^3} \frac{\varrho L(m_{\varrho}(r))}{r^2} \dd \boldsymbol{x} + a_0 \frac{3 + \alpha - 3\gamma}{\alpha (\gamma - 1)} \int_{\mathbb{R}^3} \varrho^\gamma \dd \boldsymbol{x} + \mu \int_{\mathbb{R}^3} \varrho \dd \boldsymbol{x}.
            \end{equation}
            Since $L \geq 0$, $\alpha \geq 2$ and $\gamma < \frac{3 + \alpha}{3}$, we have that $S_\mu(\varrho) > 0$. Since this is true for all $\varrho \in \mathcal{K}$, the result follows.
        \end{case}
        \begin{case}[$\alpha \in (0,2)$]
            Let $\varrho \in \mathcal{K}$, then by Lemma \ref{smallalpha} (d), $\kappa(\varrho) \geq \lambda^*_1(\varrho) = 1$. Thus, since $\gamma < \frac{3 + \alpha}{3}$, we see that
            \begin{equation}\label{biggerthan02}
            \begin{split}
                S_\mu(\varrho) & = \frac{Q(\varrho)}{2} + \frac{2 - \alpha}{4} \int_{\mathbb{R}^3} \varrho \Phi_\alpha * \varrho \dd \boldsymbol{x} + a_0 \frac{5 - 3\gamma}{2 (\gamma - 1)} \int_{\mathbb{R}^3} \varrho^\gamma \dd \boldsymbol{x} + \mu \int_{\mathbb{R}^3} \varrho \dd \boldsymbol{x} \\
                & = a_0 \frac{5 - 3\gamma}{2 \alpha(\gamma - 1)} (\alpha - 3(\gamma - 1) \kappa(\varrho)^{3(\gamma - 1) - \alpha} ) \int_{\mathbb{R}^3} \varrho^\gamma \dd \boldsymbol{x} + \mu \int_{\mathbb{R}^3} \varrho \dd \boldsymbol{x} \\
                & \geq a_0 \frac{5 - 3\gamma}{2 \alpha(\gamma - 1)} (3 + \alpha - 3\gamma ) \int_{\mathbb{R}^3} \varrho^\gamma \dd \boldsymbol{x} + \mu \int_{\mathbb{R}^3} \varrho \dd \boldsymbol{x} \\
                & > 0.
            \end{split}
            \end{equation}
            Thus, the result follows.
        \end{case}
    \end{proof}
    \setcounter{case}{0}

To establish strong convergence of the potential energy along minimising sequences in this section, we require the following lemma, which characterises the axisymmetric representation of the potential.

\begin{lemma}\label{potentially2}
Let $\alpha \in (0,3)$. Given an axisymmetric function $f \in L^\frac{6}{6-\alpha}_{\rm{axi}}(\mathbb{R}^3)$ then $(\Phi_\alpha*f)(\boldsymbol{x})$ is well-defined, axisymmetric and can be written as
    \begin{equation*}
        (\Phi_\alpha*f)(r,z) = \int_{\mathbb{R}}\int^{\infty}_0 K(r,\eta,z,w) f(\eta,w)\,\eta\dd \eta \dd w,
    \end{equation*}
    where
    \[
    K(r,\eta,z,w) \coloneq \int_{\mathbb{S}^{1}} \Phi_\alpha(r\boldsymbol{e_1} - \eta \boldsymbol{y},z-w) \dd S(\boldsymbol{y}),
    \]
    and $\boldsymbol{e}_1 = (1,0)$. Moreover, there exists $C = C(\alpha) > 0$ such that
\begin{equation}\label{Kbehave}
\begin{split}
\begin{cases}
    \displaystyle
    |K(r,\eta,z,w)| \leq C (r\eta)^{-\frac{\alpha}{2}}, & \mbox{ for $\alpha \in(0,1)$}, \\[1mm]
    \displaystyle
    |K(r,\eta,z,w)| \leq C(r\eta)^{-\frac{1}{2}} |(r,z)-(\eta,w)|^{1-\alpha}, & \mbox{ for $\alpha \in (1,3)$}.
\end{cases}
\end{split}
\end{equation}
\end{lemma}

\begin{proof}
    Since $f \in L^\frac{6}{6-\alpha}_{\rm{axi}}(\mathbb{R}^3)$, by the HLS inequality Lemma \ref{simplehls}, we have that $\Phi_\alpha * f$ is well-defined. Moreover, since $f$ is axisymmetric, we can apply Fubini's Theorem and the co-area formula to obtain
    \[
        \begin{split}
            (\Phi_\alpha * f)(\boldsymbol{v}) & = \int_{\mathbb{R}^3} \Phi_\alpha(\boldsymbol{v} - \boldsymbol{u}) f(\boldsymbol{u}) \dd \boldsymbol{u} \\
            & = \int_{\mathbb{R}} \int^\infty_0 \int_{\{|(y_1,y_2)| = \eta\}} \Phi_\alpha(\boldsymbol{v} - (y_1,y_2,w)) \dd S(y_1,y_2) f(\eta,w) \dd \eta \dd w \\
            & = \int_{\mathbb{R}} \int^\infty_0 \int_{\mathbb{S}^1} \Phi_\alpha(\boldsymbol{v} - (\eta \boldsymbol{y},w)) \dd S(\boldsymbol{y}) f(\eta,w) \eta \dd \eta \dd w.
        \end{split}
    \]
    Due to the rotational symmetry of $\Phi_\alpha$, we have that for $\boldsymbol{v} = (x_1,x_2,z)$
    \[
        \int_{\mathbb{S}^1} \Phi_\alpha(\boldsymbol{v} - (\eta \boldsymbol{y},w)) \dd S(\boldsymbol{y}) = \int_{\mathbb{S}^1} \Phi_\alpha(r(\boldsymbol{v})\boldsymbol{e_1} - \eta \boldsymbol{y},z-w) \dd S(\boldsymbol{y}).
    \]
    Thus, $\Phi_\alpha * f$ is axisymmetric and the representation formula follows. Using polar coordinates, we get
    \[
        \begin{split}
            |K(r,\eta,z,w)| & = \bigg |\int_{\mathbb{S}^1} \Phi_\alpha(r\boldsymbol{e_1} - \eta \boldsymbol{y},z-w) \dd S(\boldsymbol{y}) \bigg | \\
            & = \frac{1}{\alpha} \int^{2\pi}_0 \big ( (r-\eta \cos \theta)^2 + (\eta \sin \theta )^2 + (z-w)^2 \big)^{-\frac{\alpha}{2}} \dd \theta \\
            & = \frac{2}{\alpha} \int^{\pi}_0 \big ( (r - \eta)^2 + 2r\eta(1-\cos{\theta}) + (z-w)^2 \big)^{-\frac{\alpha}{2}} \dd \theta.
        \end{split}
    \]
    For $\theta \in [0,\pi]$, we have that there exists $c>0$ such that $2(1-\cos \theta) \geq c\theta^2$. Thus, we obtain that
    \[
        \begin{split}
            |K(r,\eta,z,w)| & = \bigg |\int_{\mathbb{S}^1} \Phi_\alpha(r\boldsymbol{e_1} - \eta \boldsymbol{y},z-w) \dd S(\boldsymbol{y}) \bigg | \\
            & \leq \frac{2}{\alpha} \int^{\pi}_0 \big ( (r - \eta)^2 + cr\eta \theta^2 + (z-w)^2 \big)^{-\frac{\alpha}{2}} \dd \theta.
        \end{split}
    \]
    \begin{case}[$\alpha \in (0,1)$]
        Then we obtain that
        \[
        |K(r,\eta,z,w)| \leq \frac{2}{\alpha} \int^{\pi}_0 \big (cr\eta \theta^2 \big)^{-\frac{\alpha}{2}} \dd \theta \leq C (r \eta)^{-\frac{\alpha}{2}}.
        \]
    \end{case}
    \begin{case}[$\alpha \in (1,3)$]
        Then, letting $\zeta > 0$, we obtain that
        \[
        \begin{split}
            |K(r,\eta,z,w)| & \leq \frac{2}{\alpha} \bigg ( \int^{\zeta}_0 \big ((r-\eta)^2 + (z-w)^2 \big)^{-\frac{\alpha}{2}} \dd \theta + \int^{\infty}_\zeta \big (cr\eta \theta^2 \big)^{-\frac{\alpha}{2}} \dd \theta \bigg ) \\
            & \leq C \big( \zeta |(r,z) - (\eta,w)|^{-\alpha} + \zeta^{1-\alpha}(r \eta)^{-\frac{\alpha}{2}}\big).
        \end{split}
        \]
        This upper bound is minimised when $\zeta \propto \frac{|r-\eta|}{\sqrt{r\eta}}$. Thus, we obtain that
        \[
            |K(r,\eta,z,w)| \leq C(r\eta)^{-\frac{1}{2}} |(r,z)-(\eta,w)|^{1-\alpha}.
        \]
    \end{case}
\end{proof}
\setcounter{case}{0}

We are now in a position to establish strong convergence of the potential energy term.

\begin{lemma}\label{weaktostrong2}
    Suppose that $\alpha \in (0,3)$, $\frac{6}{6 - \alpha} <\gamma < \frac{3}{3 - \alpha}$, $\varrho_k \in L^1_+(\mathbb{R}^3) \cap L^\gamma_{\rm{axi}} (\mathbb{R}^3)$ is a bounded sequence, $z \mapsto \varrho_k(r,z)$ is a radially decreasing function for {\it a.e.} $r$ and there exists $\tilde{\varrho} \in L^1_+(\mathbb{R}^3) \cap L^\gamma_{\rm{axi}} (\mathbb{R}^3)$ such that $\varrho_k \rightharpoonup \tilde{\varrho}$ as $k \to \infty$ in $L^\gamma(\mathbb{R}^3)$. Then, we have (up to a subsequence)
    \[
        \int_{\mathbb{R}^3} \varrho_k \Phi_\alpha * \varrho_k \dd \boldsymbol{x} \longrightarrow \int_{\mathbb{R}^3} \tilde{\varrho} \Phi_\alpha * \tilde{\varrho} \dd \boldsymbol{x}, \qquad \mbox{as $k \to \infty$.}
    \]
\end{lemma}

\begin{proof}
We split the proof into six steps.
\begin{step}
    We define the cylinder $U(R) \coloneq B_R \times [-R,R],$ where $B_R = \{ \boldsymbol{x} \in \mathbb{R}^2 \, | \, |\boldsymbol{x}| < R\}$ and $R \in \mathbb{N}$. As in Step 1 of the proof of \cite[Lemma 4.6.]{Carrillo_2026}, we can apply the theory of compact embeddings of Sobolev spaces and fractional Sobolev spaces to obtain for $R,R'\in \mathbb{N}$, that
    \begin{equation}\label{3.1}
        \int_{U(R')} \varrho_k \Phi_\alpha *(\varrho_k \mathds{1}_{U(R)}) \dd \boldsymbol{x} \longrightarrow \int_{U(R')} \tilde{\varrho} \Phi_\alpha * (\tilde{\varrho} \mathds{1}_{U(R)}) \dd \boldsymbol{x}, \qquad \mbox{as $k \to \infty$.}
    \end{equation}
\end{step}
\begin{step}
    Taking $R' > R \geq 1$, we have that since $\varrho_k$ are bounded in $L^1(\mathbb{R}^3)$, that
    \begin{equation}\label{3.2}
        \begin{split}
        \bigg |\int_{U(R')^c} \varrho_k \Phi_\alpha *(\varrho_k \mathds{1}_{U(R)}) \dd \boldsymbol{x} \bigg | & \leq \frac{C}{(R' - R)^{\alpha}} \int_{\mathbb{R}^3}\int_{\mathbb{R}^3} |\varrho_k(\boldsymbol{x})| |\varrho_k(\boldsymbol{y})| \dd \boldsymbol{x} \dd \boldsymbol{y} \\
        & \leq \frac{C}{(R' - R)^{\alpha}}.
        \end{split}
    \end{equation}
    Thus, taking $R'$ sufficiently larger than $R$, combining \eqref{3.1} and \eqref{3.2}, we obtain
    \begin{equation*}
        \int_{\mathbb{R}^3} \varrho_k \Phi_\alpha *(\varrho_k \mathds{1}_{U(R)}) \dd \boldsymbol{x} \longrightarrow \int_{\mathbb{R}^3} \tilde{\varrho} \Phi_\alpha * (\tilde{\varrho} \mathds{1}_{U(R)}) \dd \boldsymbol{x}, \qquad \mbox{as $k \to \infty$.}
    \end{equation*}
\end{step}
\begin{step}
    By symmetry, we can apply the previous steps so we have that
    \[
        \int_{U(R)} \varrho_k \Phi_\alpha *(\varrho_k \mathds{1}_{U(R)^c}) \dd \boldsymbol{x} \longrightarrow \int_{U(R)} \tilde{\varrho} \Phi_\alpha *(\tilde{\varrho} \mathds{1}_{U(R)^c}) \dd \boldsymbol{x}, \qquad \mbox{as $k \to \infty$.}
    \]
\end{step}
    \begin{step}
        Now, we need only show that for $R$ large enough,
        \[
            \int_{U(R)^c} \varrho_k \Phi_\alpha *(\varrho_k \mathds{1}_{U(R)^c}) \dd \boldsymbol{x},
        \]
        is sufficiently small for all $k$. Now, we have that
        \[
            \begin{split}
            \int_{U(R)^c} \varrho_k \Phi_\alpha *(\varrho_k \mathds{1}_{U(R)^c}) \dd \boldsymbol{x} &= \int_{B_{R}^c \times \mathbb{R}} \varrho_k \Phi_\alpha *(\varrho_k \mathds{1}_{B_{R}^c \times \mathbb{R}}) \dd \boldsymbol{x} + 2\int_{B_{R}^c \times \mathbb{R}} \varrho_k \Phi_\alpha *(\varrho_k \mathds{1}_{B_{R} \times [-R,R]^c}) \dd \boldsymbol{x} \\
            & \qquad + \int_{B_{R} \times [-R,R]^c} \varrho_k \Phi_\alpha *(\varrho_k \mathds{1}_{B_{R} \times [-R,R]^c}) \dd \boldsymbol{x} \\
            & =: I_1 + I_2 + I_3.
            \end{split}
        \]
        Considering Lemma \ref{potentially2}, we have that $I_1$ can be written as
        \begin{equation}\label{3.3}
            I_1 = 2\pi \int_{\mathbb{R}}\int^{\infty}_R \varrho_k (r,z)\,r \int_{\mathbb{R}}\int^{\infty}_R K(r,\eta,z,w) \varrho_k (\eta,w)\,\eta\dd \eta \dd w \dd r \dd z.
        \end{equation}
        We have the following cases.
        \begin{case}[$\alpha \in (0,1)$]
            From \eqref{Kbehave}, \eqref{3.3} and that $\varrho_k$ are bounded in $L^1(\mathbb{R}^3)$, we see that
            \begin{equation}\label{3.4}
                |I_1| \leq C R^{-\alpha} \int_{\mathbb{R}^3}\int_{\mathbb{R}^3} |\varrho_k(\boldsymbol{x})| |\varrho_k(\boldsymbol{y})| \dd \boldsymbol{x} \dd \boldsymbol{y} \leq C R^{-\alpha}.
            \end{equation}
        \end{case}
        \begin{case}[$\alpha \in (1,3)$]
            We have that $1< \frac{4}{5-\alpha} < \frac{6}{6-\alpha} < \gamma$. Since $\varrho_k$ are bounded in $L^1(\mathbb{R}^3) \cap L^\gamma(\mathbb{R}^3)$, we have that by interpolation, $\varrho_k$ are bounded in $L^\frac{4}{5-\alpha}(\mathbb{R}^3)$. From \eqref{Kbehave}, \eqref{3.3}, H\"older's inequality and the HLS inequality Lemma \ref{simplehls}, we see that
            \begin{equation}\label{3.5}
                \begin{split}
                    |I_1| & \leq C R^{- \frac{3-\alpha}{2}} \int_{\mathbb{R}}\int^{\infty}_R \int_{\mathbb{R}}\int^{\infty}_R |\varrho_k (r,z)|\, |(r,z)-(\eta,w)|^{1-\alpha} |\varrho_k (\eta,w)| r^\frac{5-\alpha}{4} \eta^\frac{5-\alpha}{4}\dd \eta \dd w \dd r \dd z \\
                    & \leq C R^{- \frac{3-\alpha}{2}} \bigg ( \int_{\mathbb{R}}\int^{\infty}_R |\varrho_k (r,z)|^\frac{4}{5-\alpha}\,r \dd r \dd z \bigg )^\frac{5-\alpha}{2} \\
                    & \leq C R^{- \frac{3-\alpha}{2}} \| \varrho_k \|^2_{L^\frac{4}{5-\alpha}(\mathbb{R}^3)} \\
                    & \leq C R^{- \frac{3-\alpha}{2}}.
                \end{split}
            \end{equation}
        \end{case}
        \begin{case}[$\alpha = 1$]
            By the Cauchy-Schwarz inequality, we see that for any $0<\delta<1$, that
            \[
                \Phi_1 \leq C(\Phi_{1+\delta} + \Phi_{1-\delta}).
            \]
            Thus, taking $\delta$ small enough so that the HLS exponent corresponding to $\Phi_{1+\delta}$ satisfies $\frac{4}{4-\delta} < \frac{6}{5} < \gamma$. Thus, we can apply Case 1 and Case 2 to obtain that
            \begin{equation}\label{3.6}
                |I_1| \leq C \big (R^{\delta - 1} + R^{\frac{\delta}{2} - 1} \big).
            \end{equation}
        \end{case}
        Combining \eqref{3.4}, \eqref{3.5} and \eqref{3.6}, we see that $I_1 \to 0$ as $R \to \infty$.
    \end{step}
    \setcounter{case}{0}
    \begin{step}
        We now consider $I_3$. Since $\varrho_k$ are bounded in $L^1(\mathbb{R}^3) \cap L^\gamma(\mathbb{R}^3)$, we have that by interpolation that there exists $\bar{C}>0$ such that $\|\varrho_k\|_{L^p(\mathbb{R}^3)}^p \leq 2\bar{C}$ for all $1 \leq p \leq \gamma$ and $k$. We define $V(R,i) \coloneq B_R\times[i,(i+1)]$ for $R,i \in \mathbb{N}$. Since $z \mapsto \varrho_k(r,z)$ is radially decreasing for {\it a.e.} $r$, we have that 
        $$\|\varrho_k\|^p_{L^p(V(R,i))} \leq \frac{\bar{C}}{i+1},$$
        for all $1 \leq p \leq \gamma$, $i \in \mathbb{N}$ and $k$. Thus, since $z \mapsto \varrho_k(r,z)$ is a radial function for {\it a.e.} $r$, we have that
        \begin{equation}\label{3.7}
            \begin{split}
                |I_3| & \leq 4 \int_{\bigcup^\infty_{i = R} V(R,i)} \int_{\bigcup^\infty_{j = R} V(R,j)} |\varrho_k(\boldsymbol{x})| |\Phi_\alpha(\boldsymbol{x} - \boldsymbol{y})| |\varrho_k(\boldsymbol{y})| \dd \boldsymbol{x} \dd \boldsymbol{y} \\
                & \leq 8 \sum^\infty_{i=R} \sum^\infty_{j=i} \int_{V(R,i)} \int_{V(R,j)} |\varrho_k(\boldsymbol{x})| |\Phi_\alpha(\boldsymbol{x} - \boldsymbol{y})| |\varrho_k(\boldsymbol{y})| \dd \boldsymbol{x} \dd \boldsymbol{y}.
            \end{split}
        \end{equation}
        For $j=i,i+1$, we have that by the HLS inequality Lemma \ref{simplehls},
        \begin{align}\label{3.8}
            \int_{V(R,i)} \int_{ V(R,j)} |\varrho_k(\boldsymbol{x})| |\Phi_\alpha(\boldsymbol{x} - \boldsymbol{y})| |\varrho_k(\boldsymbol{y})| \dd \boldsymbol{x} \dd \boldsymbol{y} & \leq C \|\varrho_k\|_{L^\frac{6}{6-\alpha}(V(R,i))} \|\varrho_k\|_{L^\frac{6}{6-\alpha}(V(R,j))} \nonumber \\
            & \leq \frac{C}{(i+1)^\frac{6-\alpha}{3}}.
        \end{align}
        For $j \geq i + 2$, using the form of $\Phi_\alpha$, we see that
        \begin{equation}\label{3.9}
        \begin{split}
            \int_{V(R,i)} \int_{ V(R,j)} |\varrho_k(\boldsymbol{x})| |\Phi_\alpha(\boldsymbol{x} - \boldsymbol{y})| |\varrho_k(\boldsymbol{y})| \dd \boldsymbol{x} \dd \boldsymbol{y} & \leq C (j - i -1)^{-\alpha} \|\varrho_k\|_{L^1(V(R,i))} \|\varrho_k\|_{L^1(V(R,j))} \\
            & \leq \frac{C(j - i -1)^{-\alpha}}{(i+1)(j+1)}.
        \end{split}
        \end{equation}
        Here, the last inequality will be used in order to obtain convergence of the sum in \eqref{3.7}.
        \begin{case}[$\alpha \in {(0,1]}$]
        We have that for $j \geq i+ 2$, that
        \[
            \frac{(j - i -1)^{-\alpha}}{(i+1)(j+1)} \leq \frac{(j - i -1)^{-\frac{2+ \alpha}{2}}}{(i+1)^{\frac{2+ \alpha}{2}}}.
        \]
            Combining this with \eqref{3.7}, \eqref{3.8}, \eqref{3.9}, we see that
            \begin{equation}\label{3.10}
                |I_3| \leq C \bigg ( 1 + \sum^\infty_{j=1} j^{-\frac{2+\alpha}{2}} \bigg )\sum^\infty_{i=R} (i+1)^{-\frac{2+\alpha}{2}}.
            \end{equation}
            Since $\frac{2+\alpha}{2} > 1$, the series converge.
        \end{case}
        \begin{case}[$\alpha \in {(1,3)}$]
            Combining \eqref{3.7}, \eqref{3.8}, \eqref{3.9}, we see that
            \begin{equation}\label{3.11}
                |I_3| \leq C \bigg ( 1 + \sum^\infty_{j=1} j^{-\alpha} \bigg )\sum^\infty_{i=R} (i+1)^{-\frac{6-\alpha}{3}}.
            \end{equation}
            Since $\alpha > 1$, the series converge.
        \end{case}
        Combining \eqref{3.10} and \eqref{3.11}, we obtain that $I_3 \to 0$ as $R \to \infty$.
        \end{step}
        \setcounter{case}{0}
        \begin{step}
            We split $I_2$ into two parts
            \[
                \begin{split}
                I_2 & = 2\int_{(B_{2R}\setminus B_R) \times \mathbb{R}} \varrho_k \Phi_\alpha *(\varrho_k \mathds{1}_{B_{R} \times [-R,R]^c}) \dd \boldsymbol{x} + 2\int_{B_{2R}^c \times \mathbb{R}} \varrho_k \Phi_\alpha *(\varrho_k \mathds{1}_{B_{R} \times [-R,R]^c}) \dd \boldsymbol{x} \\
                & =: I_{2,1} + I_{2,2}.
                \end{split}
            \]
            By the same approach as in Step 5, one can show that $I_{2,1} \to 0$ as $R \to \infty$. For the second integral, we can see that
            \[
                    {\rm dist} (B_{R} \times [-R,R]^c,B_{2R}^c \times \mathbb{R}) = R,
            \]
            Hence, we have from the form of $\Phi_\alpha$ and the fact that $\varrho_k$ are bounded in $L^1(\mathbb{R}^3) \cap L^\gamma(\mathbb{R}^3)$
            \[
                |I_{2,2}| \leq C R^{-\alpha} \int_{\mathbb{R}^3}\int_{\mathbb{R}^3} |\varrho_k(\boldsymbol{x})| |\varrho_k(\boldsymbol{y})| \dd \boldsymbol{x} \dd \boldsymbol{y} \leq C R^{-\alpha}.
            \]
            Hence $I_{2,2} \to 0$ as $R \to \infty$. This concludes the proof.
    \end{step}
\end{proof}
\setcounter{step}{0}

We now establish the existence of minimisers of the free-energy functional $S_\mu$ over the admissible set $\mathcal{K}$.

    \begin{lemma}
        Suppose $\alpha \in (0,3)$, $\frac{6}{6 - \alpha} < \gamma < \frac{3 + \alpha}{3}$, $L$ satisfies \eqref{Lnice2}, \eqref{Lstable2} and, for $\alpha \in (0,2]$, \eqref{Linstab} for $\omega > \omega_*$, then there exists a minimiser $\bar{\rho}$ of $S_\mu$ over $\mathcal{K}$ and $\ell_\mu > 0$. Moreover, for any minimiser $\bar{\rho}$, there exists a translation $\boldsymbol{y} \in \{\boldsymbol{0}\} \times \mathbb{R}$ such that $z \mapsto T^{\boldsymbol{y}}\bar{\rho}(r,z)$ is a radially decreasing function for {\it a.e} $r$.
    \end{lemma}

    \begin{proof}
    We split the proof into three steps.
        \begin{step}
        Suppose that $\varrho_k \in \mathcal{K}$ is a minimising sequence of $S_\mu$, i.e. that
    \[
    Q(\varrho_k) = 0,
    \]
    for all $0 < \lambda < 1$ we have that
    \[
    Q((\varrho_k)_\lambda) > 0,
    \]
    and
    \[
    \lim_{k \to \infty} S_{\mu}(\varrho_k) = \ell_\mu.
    \]
    We would like to show that we can take a minimising sequence that is radially decreasing in the $z$--coordinate. To do this, we take the radially decreasing rearrangement in Lemma \ref{rearrangement} in the $z$--coordinate. That is, define $\tilde{\varrho}_k$ by $\tilde{\varrho}_k(r,z) = (\varrho_k(r,\cdot))^{\#}(z)$. Note that since the radially decreasing rearrangement preserves $L^p$ norms, we have that
    \[
        \int_{\mathbb{R}^3} \tilde{\varrho}_k \dd \boldsymbol{x} = \int_{\mathbb{R}^2} \int_{\mathbb{R}} \tilde{\varrho}_k(\boldsymbol{y},z) \dd z \dd \boldsymbol{y} = \int_{\mathbb{R}^2} \int_{\mathbb{R}} \varrho_k(\boldsymbol{y},z) \dd z \dd \boldsymbol{y} = \int_{\mathbb{R}^3} \varrho_k \dd \boldsymbol{x}.
    \]
    Similarly, we see that $\| \tilde{\varrho}_k \|_{L^\gamma(\mathbb{R}^3)} = \| \varrho_k \|_{L^\gamma(\mathbb{R}^3)}$ and
    \[
        \int_{\mathbb{R}^3} \frac{\tilde{\varrho}_k L(m_{\tilde{\varrho}_k}(r))}{r^2} \dd \boldsymbol{x} = \int_{\mathbb{R}^3} \frac{\varrho_k L(m_{\varrho_k}(r))}{r^2} \dd \boldsymbol{x}.
    \]
    By Lemma \eqref{rearrangement}, we see that since $z \mapsto -\Phi_{\alpha}(\boldsymbol{x},z)$ is a radially decreasing function for $z \in \mathbb{R}$ and any $\boldsymbol{x} \in \mathbb{R}^2$, we obtain that
    \begin{equation}\label{partrearrange}
        \begin{split}
            \int_{\mathbb{R}^3} \tilde{\varrho}_k \Phi_\alpha * \tilde{\varrho}_k \dd \boldsymbol{x} & = \int_{\mathbb{R}^2} \int_{\mathbb{R}^2} \int_{\mathbb{R}} \int_{\mathbb{R}} \tilde{\varrho}_k(\boldsymbol{x},z) \Phi_\alpha(\boldsymbol{x}-\boldsymbol{y},z-w) \tilde{\varrho}_k(\boldsymbol{y},w) \dd z \dd w \dd \boldsymbol{x} \dd \boldsymbol{y} \\
            & \leq \int_{\mathbb{R}^2} \int_{\mathbb{R}^2} \int_{\mathbb{R}} \int_{\mathbb{R}} \varrho_k(\boldsymbol{x},z) \Phi_\alpha(\boldsymbol{x}-\boldsymbol{y},z-w) \varrho_k(\boldsymbol{y},w) \dd z \dd w \dd \boldsymbol{x} \dd \boldsymbol{y} \\
            & =\int_{\mathbb{R}^3} \varrho_k \Phi_\alpha * \varrho_k \dd \boldsymbol{x}.
        \end{split}
    \end{equation}
    Hence, we have that $\tilde{\varrho}_k \in L^1_+(\mathbb{R}^3) \cap L^\gamma(\mathbb{R}^3)$ and satisfies $Q(\tilde{\varrho}_k) \leq 0$. Thus, by Lemma \ref{largealpha} or \ref{smallalpha}, there exists $0 <\lambda^*(\tilde{\varrho}_k) \leq 1$ such that $(\tilde{\varrho}_k)_{\lambda^*(\tilde{\varrho}_k)} \in \mathcal{K}$. Moreover, it is easy to show that $(\tilde{\varrho}_k)_{\lambda^*(\tilde{\varrho}_k)}(r,z) = ((\varrho_k)_{\lambda^*(\tilde{\varrho}_k)}(r,\cdot))^\#(z)$. Hence, by Lemma \ref{largealpha} (c) or Lemma \ref{smallalpha} (c) we see that
    \[
        S_\mu((\tilde{\varrho}_k)_{\lambda^*(\tilde{\varrho}_k)}) \leq S_\mu((\varrho_k)_{\lambda^*(\tilde{\varrho}_k)}) \leq S_\mu(\varrho_k).
    \]
    This implies that $(\tilde{\varrho}_k)_{\lambda^*(\tilde{\varrho}_k)}$ is a minimising sequence. Hence, we can assume that $\varrho_k$ are radially decreasing in the $z$--coordinate. Using \eqref{biggerthan01} and \eqref{biggerthan02}, we see that
    \begin{equation}\label{seqbound2}
        \| \varrho_k \|_{L^\gamma} + \| \varrho_k \|_{L^1} \leq C.
    \end{equation}
        We have that since $\varrho_k \in \mathcal{K}$, that
    \begin{equation}\label{sequenceHLS2}
        \| \varrho_k \|_{L^\gamma}^\gamma \leq - \frac{\alpha}{6 a_0} \int_{\mathbb{R}^3} \varrho_k \Phi_\alpha * \varrho_k \dd \boldsymbol{x} \leq C \| \varrho_k \|_{L^1}^\frac{\gamma(6 - \alpha) - 6}{3(\gamma - 1)} \| \varrho_k \|_{L^\gamma}^\frac{\gamma \alpha}{3(\gamma - 1)}.
    \end{equation}
    Thus, combining this with \eqref{seqbound2}, since $\frac{6}{6 - \alpha} < \gamma < \frac{3 + \alpha}{3}$, we get
    \[
        \| \varrho_k \|_{L^\gamma}^\frac{\gamma(3\gamma - 3 - \alpha)}{3(\gamma - 1)} \leq C \| \varrho_k \|_{L^1}^\frac{\gamma(6 - \alpha) - 6}{3(\gamma - 1)} \leq C,
    \]
    and thus
    \begin{equation}\label{lowerseq2}
        \| \varrho_k \|_{L^\gamma} \geq C.
    \end{equation}
    Thus, using \eqref{biggerthan01} and \eqref{biggerthan02}, we see that $\ell_\mu > 0$. By Lemma \ref{weaklowersemi}, we have that \eqref{seqbound2} implies there exists $\tilde{\varrho} \in L^\gamma(\mathbb{R}^3)$ such that (up to a subsequence) $\varrho_k \rightharpoonup \tilde{\varrho}$ as $k \to \infty$ in $L^\gamma(\mathbb{R}^3)$. Since $\frac{3+\alpha}{3} < \frac{3}{3-\alpha}$, Lemma \ref{weaktostrong2} tells us that (up to a subsequence)
    \begin{equation}\label{potentialstrong2}
        \int_{\mathbb{R}^3} \varrho_k \Phi_\alpha * \varrho_k \dd \boldsymbol{x} \longrightarrow \int_{\mathbb{R}^3} \tilde{\rho} \Phi_\alpha * \tilde{\rho} \dd \boldsymbol{x}, \qquad \mbox{as $k \to \infty$.}
    \end{equation}
    Combining \eqref{sequenceHLS2} and \eqref{lowerseq2}, we have that
    \[
    - \int_{\mathbb{R}^3} \varrho_k \Phi_\alpha * \varrho_k \dd \boldsymbol{x} \geq C.
    \]
    Combining with \eqref{potentialstrong2}, we have that
    \[
    - \int_{\mathbb{R}^3} \tilde{\rho} \Phi_\alpha * \tilde{\rho} \dd \boldsymbol{x} \geq C.
    \]
    Hence, we see that $\tilde{\rho} \not \equiv 0$.
    \end{step}
    \begin{step}
    However, by weak lower semi-continuity of the $L^\gamma$ norm, Lemma \ref{weaktostrong2} and verbatim the same approach as that of the proof of Lemma \ref{lowerkinetic}, we obtain that $Q(\tilde{\varrho}) \leq 0$. As such, there exists $0<\lambda^*(\tilde{\varrho}) \leq 1$ such that $\bar{\rho} := \tilde{\varrho}_{\lambda^*(\tilde{\varrho})} \in \mathcal{K}$. Then, we have that (up to a subsequence) $\varrho_k^* := (\varrho_k)_{\lambda^*(\tilde{\varrho})} \rightharpoonup \bar{\rho}$ as $k \to \infty$ in $L^\gamma(\mathbb{R}^3)$. From Lemma \ref{weaklowersemi}, we have that $\bar{\varrho} \geq 0$ {\it a.e.} and
    \[
    \| \bar{\rho} \|_{L^\gamma} + \| \bar{\rho} \|_{L^1} \leq C.
    \]
    Now, by Lemma \ref{weaklowersemi} and weak lower semi-continuity of $\int_{\mathbb{R}^3} \varrho^\gamma \dd \boldsymbol{x}$, we get that
    \[
    \int_{\mathbb{R}^3} \bar{\rho} \dd \boldsymbol{x} \leq \liminf_{k \to \infty} \int_{\mathbb{R}^3} \varrho_k^* \dd \boldsymbol{x}, \qquad \text{ and } \qquad \int_{\mathbb{R}^3} (\bar{\rho})^\gamma \dd \boldsymbol{x} \leq \liminf_{k \to \infty} \int_{\mathbb{R}^3} (\varrho_k^*)^\gamma \dd \boldsymbol{x}.
    \]
    By verbatim the same approach as that of the proof of Lemma \ref{lowerkinetic}, since $L$ satisfies \eqref{Lnice2}, \eqref{Lstable2}, we have that
    \[
        \int_{\mathbb{R}^3} \frac{\bar{\rho} L(m_{\bar{\rho}}(r))}{r^2} \dd \boldsymbol{x} \leq \liminf_{k \to \infty} \int_{\mathbb{R}^3} \frac{\varrho^*_k L(m_{\varrho^*_k}(r))}{r^2} \dd \boldsymbol{x}.
    \]
    Hence, using Lemma \ref{largealpha} or \ref{smallalpha}, we have that
    $$
    S_\mu(\bar{\rho}) \leq \liminf_{k \to \infty} S_\mu(\varrho_k^*) \leq \liminf_{k \to \infty} S_\mu(\varrho_k) = \ell_\mu.
    $$
    Thus $\bar{\rho}$ is a minimiser.
    \end{step}
    \begin{step}
        Employing Lemma \ref{rearrangement}, for the radially decreasing rearrangement in the $z$--coordinate of $\bar{\rho}$, given by $\rho^*(r,z) = (\bar{\rho}(r,\cdot))^{\#}(z)$, as in \eqref{partrearrange}, we have equality in
        \begin{equation}\label{partrearrange2}
        \begin{split}
            \int_{\mathbb{R}^3} \rho^* \Phi_\alpha * \rho^* \dd \boldsymbol{x} & = \int_{\mathbb{R}^2} \int_{\mathbb{R}^2} \int_{\mathbb{R}} \int_{\mathbb{R}} \rho^*(\boldsymbol{x},z) \Phi_\alpha(\boldsymbol{x}-\boldsymbol{y},z-w) \rho^*(\boldsymbol{y},w) \dd z \dd w \dd \boldsymbol{x} \dd \boldsymbol{y} \\
            & \leq \int_{\mathbb{R}^2} \int_{\mathbb{R}^2} \int_{\mathbb{R}} \int_{\mathbb{R}} \bar{\rho}(\boldsymbol{x},z) \Phi_\alpha(\boldsymbol{x}-\boldsymbol{y},z-w) \bar{\rho}(\boldsymbol{y},w) \dd z \dd w \dd \boldsymbol{x} \dd \boldsymbol{y} \\
            & =\int_{\mathbb{R}^3} \bar{\rho} \Phi_\alpha * \bar{\rho} \dd \boldsymbol{x},
        \end{split}
    \end{equation}
        if and only if there is a translation $\boldsymbol{y} \in \{\boldsymbol{0}\} \times \mathbb{R}$ such that $z \mapsto T^{\boldsymbol{y}} \bar{\rho}(r,z)$ is a radially decreasing function for {\it a.e.} $r$. If there is no equality in \eqref{partrearrange2}, then as in Step 1, we can obtain a function $\bar{\rho}^* \in \mathcal{K}$ such that $S_\mu(\bar{\rho}^*) < S_\mu(\bar{\rho})$, a contradiction, thus we conclude the lemma.
    \end{step}
    \end{proof}
    \setcounter{step}{0}

    For $\alpha \in (0,2)$, in order to prove that minimisers of $S_\mu$ over $\mathcal{K}$ are indeed rotating Riesz star solutions, and to establish the instability of these steady states, we require slightly stronger superhomogeneity assumptions on $L$. We therefore define
    \[
        \bar{\omega} := \frac{(2-\alpha)((6-\alpha)\gamma - 6)(5-3\gamma)}{3\alpha (\gamma - 1)(3+\alpha -3\gamma)} > 0.
    \]
    We will require that $\omega > \omega_* + \bar{\omega}$. 

\begin{lemma}\label{nonempty2}
        Suppose $\alpha \in (0,2)$, $\frac{6}{6-\alpha} < \gamma < \frac{3 + \alpha}{3}$, $L$ satisfies \eqref{Lnice2} and \eqref{Linstab} for $\omega > \omega_* + \bar{\omega}$. Then, there exists $\bar{Q} \in (0,\infty]$ such that for all $\varrho \in L^1_+(\mathbb{R}^3) \cap L^\gamma(\mathbb{R}^3)$ satisfying
        $$0<Q(\varrho) < \bar{Q}\int_{\mathbb{R}^3} \varrho^\gamma \dd \boldsymbol{x},$$ 
        and $\kappa(\varrho) = 1$, we have that $S_\mu(\varrho) > \ell_\mu$. Moreover, if \eqref{Linstab} is satisfied for 
        \begin{equation}\label{secondcond}
            \omega \geq \frac{\gamma + 1 - \alpha}{3 + \alpha - 3\gamma},
        \end{equation}
        then $\bar{Q} = \infty$.
    \end{lemma}

    \begin{remark}
        Note that the condition \eqref{secondcond} is required so that
        \[
            \xi \mapsto \int_{\mathbb{R}^3} \frac{\leftindex^\xi\varrho L(m_{\leftindex^\xi\varrho}(r))}{r^2} \dd \boldsymbol{x},
        \]
        is non-increasing in $\xi$. This is only a stronger assumption than $\omega > \omega_* + \bar{\omega}$ in the case that $\alpha \in (1,2)$ and $\frac{5-\alpha}{3} < \gamma < \frac{3+\alpha}{3}$, as this is the only range under which
        \[
            \frac{\gamma + 1 - \alpha}{3 + \alpha - 3\gamma} > \omega_* + \bar{\omega}.
        \]
    \end{remark}

    \begin{proof}
    Take $\varrho \in L^1_+(\mathbb{R}^3) \cap L^\gamma(\mathbb{R}^3)$ such that $Q(\varrho) > 0$ and $\kappa(\varrho) = 1$. Employing \eqref{Qineq}, as in Lemma \ref{nonempty}, we obtain that $\xi \mapsto Q(\leftindex^\xi \varrho)$ is strictly decreasing and thus there exists $\bar{\xi} > 1$ such that $Q(\leftindex^{\bar{\xi}} \varrho) = 0$. Since $\kappa(\leftindex^\xi \varrho) = \kappa(\varrho) = 1$, by Lemma \ref{smallalpha} (c), we see that $\lambda^*_1(\leftindex^{\bar{\xi}} \varrho) = 1$. Hence, by the definition of $\ell_\mu$, we obtain $S_\mu(\leftindex^{\bar{\xi}} \varrho) \geq \ell_\mu$. We aim to show that $\xi \mapsto S_\mu(\leftindex^\xi \varrho)$ is strictly decreasing. Employing estimates similar to \eqref{usingk} and \eqref{kinetbound}, we obtain that
    \begin{align}\label{not so nice bound}
        S_\mu(\leftindex^\xi \varrho) & \leq \xi^{\gamma + 1 - \alpha - \omega(3+\alpha - 3\gamma)} \frac{1}{2} \int_{\mathbb{R}^3} \frac{\varrho L(m_{\varrho}(r))}{r^2} \dd \boldsymbol{x} \nonumber \\
        & \qquad + \xi^{(6-\alpha)\gamma - 6}\frac{a_0(5-\alpha - 3\gamma)(3 + \alpha -3\gamma)}{\alpha(2-\alpha)(\gamma - 1)} \int_{\mathbb{R}^3} \varrho^\gamma \dd \boldsymbol{x} + \xi^{3\gamma - 3 - \alpha} \mu \int_{\mathbb{R}^3} \varrho \dd \boldsymbol{x} \nonumber \\
        & =: T_\mu(\varrho,\xi) + \xi^{3\gamma - 3 - \alpha} \mu \int_{\mathbb{R}^3} \varrho \dd \boldsymbol{x},
    \end{align}
    with equality for $\xi = 1$. Since $\gamma < \frac{3+\alpha}{3}$, we have that 
    \[
        \xi \mapsto \xi^{3\gamma - 3 - \alpha} \mu \int_{\mathbb{R}^3} \varrho \dd \boldsymbol{x},
    \]
    is strictly decreasing. Thus, from \eqref{not so nice bound}, $\xi \mapsto S_\mu(\leftindex^\xi \varrho)$ is strictly decreasing if $\xi \mapsto T_\mu(\varrho,\xi)$ is decreasing. For $\omega > \omega_*$, we have that for $\xi \geq 1$,
    \begin{align*}
        \frac{\d T_\mu(\varrho,\xi)}{\d \xi} & = \xi^{3\gamma - 3 - \alpha} \bigg ( \xi^{7 - \alpha - (5-\alpha)\gamma - \omega(3+\alpha - 3\gamma)} \frac{\gamma + 1 - \alpha - \omega(3+\alpha - 3\gamma)}{2} \int_{\mathbb{R}^3} \frac{\varrho L(m_{\varrho}(r))}{r^2} \dd \boldsymbol{x} \\
        & \qquad + \frac{a_0(5-\alpha - 3\gamma)(3 + \alpha -3\gamma)((6-\alpha)\gamma - 6)}{\alpha(2-\alpha)(\gamma - 1)} \int_{\mathbb{R}^3} \varrho^\gamma \dd \boldsymbol{x} \bigg ) \\
        & \leq \xi^{3\gamma - 3 - \alpha} \frac{\d T_\mu(\varrho,\xi)}{\d \xi} \bigg |_{\xi = 1}.
    \end{align*}
    Thus, we just need to show that $\frac{\d T_\mu(\varrho,\xi)}{\d \xi} \big |_{\xi = 1} < 0$. Employing \eqref{Qineq} for $\xi = 1$ and some tedious algebra, we obtain that
    \begin{equation}\label{alsonotnice}
        \frac{\d T_\mu(\varrho,\xi)}{\d \xi} \bigg |_{\xi = 1} = \frac{\gamma + 1 - \alpha - \omega(3+\alpha - 3\gamma)}{2} Q(\varrho) + (\omega_* + \bar\omega - \omega) \frac{3 a_0(3 + \alpha - 3\gamma)^2}{2(2-\alpha)} \int_{\mathbb{R}^3} \varrho^\gamma \dd \boldsymbol{x}.
    \end{equation}
    Thus, if $\omega > \omega_* + \bar\omega$, then there exists such a $\bar{Q}$ such that $\frac{\d T_\mu(\varrho,\xi)}{\d \xi} \big |_{\xi = 1} < 0$ if $Q(\varrho) < \bar{Q} \int_{\mathbb{R}^3} \varrho^\gamma \dd \boldsymbol{x}$. Moreover, if $\omega$ also satisfies $\omega \geq \frac{\gamma + 1 - \alpha}{3 + \alpha - 3\gamma}$, then the first term on the right-hand side of \eqref{alsonotnice} is negative, hence $\bar{Q} = \infty$.
    \end{proof}

    \begin{remark}\label{nottechnical}
    In Lemma \ref{nonempty2}, we employed the energy-critical scaling. However, considering instead a more general scaling of the form $\xi \mapsto \xi^a \varrho(\xi^b \,\cdot\,)$, we obtain
\[
\begin{aligned}
S_\mu(\xi^a \varrho(\xi^b \,\cdot\,))
&= \xi^{a-b} \frac{1}{2} \int_{\mathbb{R}^3} \frac{\varrho L(\xi^{a-3b}m_{\varrho}(r))}{r^2}\,\dd \boldsymbol{x}
+ \xi^{a\gamma-3b} 3a_0 \int_{\mathbb{R}^3} \varrho^\gamma \,\dd \boldsymbol{x} \\
&\qquad
+ \xi^{2a-(6-\alpha)b} \frac{1}{2} \int_{\mathbb{R}^3} \varrho \,\Phi_\alpha * \varrho \,\dd \boldsymbol{x}
+ \xi^{a-3b} \mu \int_{\mathbb{R}^3} \varrho \,\dd \boldsymbol{x}.
\end{aligned}
\]

To apply \eqref{Linstab} in order to derive an upper bound for the kinetic term, we require $a < 3b$. Under this condition, it follows that
\[
\begin{aligned}
S_\mu(\xi^a \varrho(\xi^b \,\cdot\,))
&\leq
\xi^{a-b+\omega(a-3b)}
\frac{1}{2}
\int_{\mathbb{R}^3}
\frac{\varrho L(m_{\varrho}(r))}{r^2}\,\dd \boldsymbol{x}
+ \xi^{a\gamma-3b}
3a_0
\int_{\mathbb{R}^3}
\varrho^\gamma \,\dd \boldsymbol{x} \\
&\qquad
+ \xi^{2a-(6-\alpha)b}
\frac{1}{2}
\int_{\mathbb{R}^3}
\varrho \,\Phi_\alpha * \varrho \,\dd \boldsymbol{x}
+ \xi^{a-3b}
\mu
\int_{\mathbb{R}^3}
\varrho \,\dd \boldsymbol{x} \\
&=: T_\mu(\varrho,\xi)
+ \xi^{a-3b}\mu \int_{\mathbb{R}^3} \varrho \,\dd \boldsymbol{x}.
\end{aligned}
\]

Since $a < 3b$, the mass term is strictly decreasing in $\xi$. Proceeding as in the proof of the previous lemma, for $\varrho$ satisfying $Q(\varrho)=0$ and $\kappa(\varrho)=1$, we find
\[
\frac{\d}{\d\xi}T_\mu(\varrho,\xi)\bigg|_{\xi=1}
=
(\omega_*+\bar\omega-\omega)
\frac{3a_0(3b-a)(3+\alpha-3\gamma)}{2(2-\alpha)}
\int_{\mathbb{R}^3}\varrho^\gamma\,\dd \boldsymbol{x}.
\]

Therefore, in order for $T_\mu(\varrho,\xi)$ to be decreasing in $\xi$ near $\xi=1$ under the constraint $Q(\varrho)=0$, it is necessary that
\[
\omega > \omega_*+\bar\omega.
\]
The fact that this condition remains invariant under different choices of scaling suggests that it reflects a fundamental structural requirement rather than merely a technical assumption.
    \end{remark}
    
    We are now in a position to prove that minimisers are indeed rotating Riesz star solutions. We recall the definition of $\mathbb{R}^3_\varepsilon$ given in the proof of Theorem \ref{lmulaxi}.

\begin{proof}[Proof of Theorem \ref{Existofrot}]
    We will show that the Euler-Lagrange equations of the minimisation problem correspond to \eqref{stationaxi}. We split the proof into four steps.
    \begin{step}
        As such, for $\varepsilon > 0$, we consider the set
    \[
    \Gamma_\varepsilon := \Big \{ \boldsymbol{x} \in \mathbb{R}^3_\varepsilon \, \Big | \, \varepsilon < \bar{\rho}(\boldsymbol{x}) < \frac{1}{\varepsilon} \Big \}.
    \]
    Let $w \in L^\infty(\mathbb{R}^3)$ be such that $w$ has compact support in $\mathbb{R}^3_\varepsilon$ and is non-negative on $\mathbb{R}^3 \setminus \Gamma_\varepsilon$. Then, defining
    \[
        \bar{\rho}_\tau \coloneq \bar{\rho} + \tau w,
    \]
    for $\tau \geq 0$ small enough, we have $\bar{\rho}_\tau \in L^1_+(\mathbb{R}^3) \cap L^\gamma(\mathbb{R}^3)$. We show that for such $\tau$, $\tau \mapsto Q(\rho_\tau)$ and $\tau \mapsto S_\mu(\rho_\tau)$ are both continuous. This follows from the fact that $L$ is continuous, $w \in L^\infty$ and has compact support in $\mathbb{R}^3_\varepsilon$, thus for all $0 \leq \tau \leq \delta$, 
    $$L(m_{ \bar{\rho}_\tau}(r)) \leq L(\|\bar{\rho}\|_{L^1} + \delta \|w\|_{L^1}) =: L_\delta,$$
    where the inequality follows from the fact that $L'(m) \geq 0$ for all $m \geq 0$. Thus, $\bar{\rho}_\tau L(m_{ \bar{\rho}_\tau}(r))r^{-2}$ can be dominated by 
    $$\frac{ \bar{\rho} L(m_{ \bar{\rho}}(r))}{r^2} \mathds{1}_{\{r < \varepsilon\}} + \frac{ (\bar{\rho} + \delta |w|) L_{\delta}}{\varepsilon^2} \mathds{1}_{\{r \geq \varepsilon\}},$$
    for $0 \leq \tau \leq \delta$, which is integrable. Hence, by the dominated convergence Theorem, we have continuity of 
    $$\tau \mapsto \int_{\mathbb{R}^3} \frac{ \bar{\rho}_\tau L(m_{ \bar{\rho}_\tau}(r))}{r^2} \dd \boldsymbol{x}.$$
    The continuity of the internal energy, potential energy and mass are trivial.
    \end{step}
    \begin{step}
        Suppose $\alpha \in (0,2)$, there are two cases.
        \begin{case}
                There exists $\delta > 0$ such that $Q((\bar{\rho}_\tau)_{\kappa(\bar{\rho}_\tau)}) < 0$ for $0<\tau<\delta$. Thus, by Lemma \ref{smallalpha}, we have that $\lambda^*_1(\bar{\rho}_\tau) < \kappa(\bar{\rho}_\tau)$ for $0 < \tau < \delta$. For simplicity, we write $\lambda(\tau) \coloneq \lambda^*_1(\bar{\rho}_\tau)$. We will show that $\lambda(\tau)$ is differentiable for $0 < \tau < \delta$. To do this, we define the function $f:(0,\delta) \times \mathbb{R}_+ \to \mathbb{R}$ by
        \[
            f(\tau,\lambda) \coloneq Q((\bar{\rho}_\tau)_\lambda).
        \]
        We will apply the Implicit Function Theorem to $f$ to prove differentiability of $\lambda(\tau)$. As in the proof of Lemma \ref{smallalpha}, we have that $\frac{\partial f}{\partial \lambda}(\tau, \lambda)$ exists and as in Step 1, it is continuous in $(\tau,\lambda)$. By \eqref{differentbehave}, we have that since $\lambda(\tau) < \kappa(\bar{\rho}_\tau)$ for $0<\tau<\delta$, that
        \[
            \frac{\partial f}{\partial \lambda}(\tau, \lambda(\tau)) < 0.
        \]
        We now show that $f$ is differentiable with respect to $\tau$. We need only show this for the kinetic energy, as the other terms are trivially differentiable. For $0< \tau < \delta$ and $h \neq 0$ such that $0< \tau + h < \delta$, since $L'$ is continuous, we have by integration by parts that
        \[
            \bigg |\frac{ L(m_{ \bar{\rho}_{\tau + h}}(r)) - L(m_{ \bar{\rho}_\tau}(r))}{h} \bigg | = \bigg | \int^1_0 L'(m_{ \bar{\rho}_{\tau + th}}(r)) m_{w}(r) \dd t \bigg | \leq L'_\delta |m_{w}(r)|,
        \]
        where
        $$L'(m_{ \bar{\rho}_\tau}(r)) \leq \max_{m \in [0,\|\bar{\rho}\|_{L^1} + \delta \|w\|_{L^1}]} L'(m) =: L'_{\delta}.$$
        Since $\supp (w) \subset \mathbb{R}^3_\varepsilon$, we note that
        \[
            \frac{ \bar{\rho}_{\tau + h} L(m_{ \bar{\rho}_{\tau + h}}(r)) - \bar{\rho}_\tau L(m_{ \bar{\rho}_\tau}(r))}{h r^2} = \frac{ w L(m_{\bar{\rho}_{\tau + h}}(r))}{r^2} + \frac{ \bar{\rho}_\tau \big ( L(m_{ \bar{\rho}_{\tau + h}}(r)) - L(m_{ \bar{\rho}_\tau}(r)) \big )}{h r^2},
        \]
        is dominated by
        \[
             \bigg ( \frac{ |w| L_\delta + (\bar{\rho} + \delta |w|) L'_\delta |m_{w}(r)|}{\varepsilon^2} \bigg ) \mathds{1}_{\{r \geq \varepsilon\}}.
        \]
        Thus, we can apply the dominated convergence Theorem to obtain that $f$ is differentiable with respect to $\tau$, with
        \[
        \begin{split}
            \frac{\partial f}{\partial \tau}(\tau, \lambda) & = \lambda \int_{\mathbb{R}^3} \frac{w L(m_{\bar{\rho}_\tau}(r)) + \bar{\rho}_\tau L'(m_{\bar{\rho}_\tau}(r)) m_w(r)}{r^2} \dd \boldsymbol{x} \\
            & \qquad + 3 a_0 \gamma \lambda^{\frac{3(\gamma -1)}{2}} \int_{\mathbb{R}^3} w \bar{\rho}_\tau^{\gamma - 1} \dd \boldsymbol{x} + \alpha \lambda^\frac{\alpha}{2} \int_{\mathbb{R}^3} w \Phi_\alpha * \bar{\rho}_\tau \dd \boldsymbol{x}.
        \end{split}
        \]
        By a similar reasoning to before, this is continuous in $(\tau, \lambda)$. Since $f(\tau,\lambda(\tau)) = 0$ for each $0<\tau<\delta$, by the Implicit Function Theorem, there exists a neighbourhood interval $I \subset (0,\delta)$ of $\tau$ and a continuously differentiable function $g:I \to \mathbb{R}_+$ such that
        \[
            f(s,g(s)) = 0,
        \]
        for all $s \in I$ and $g(\tau) = \lambda(\tau)$. By Lemma \ref{smallalpha}, since $Q((\bar{\rho}_s)_{\kappa(\bar{\rho}_s)}) < 0$ for $s \in I$, we must have that $g(s) = \lambda(s) < \kappa(\bar{\rho}_s)$. Thus, this implies that $\tau \mapsto \lambda(\tau)$ is continuously differentiable for $0 < \tau < \delta$. Since $\lambda(\tau) < \kappa(\bar{\rho}_\tau)$ and $\tau \mapsto \kappa(\bar{\rho}_\tau)$ is continuous, we have that $\lambda(\tau)$ is bounded for $0 \leq \tau < \delta$. Suppose that $\lambda(\tau) \not \to 1 = \lambda^*_1(\bar{\rho})$ as $\tau \to 0$. Then, there exists $\tau_k \to 0$, as $k \to \infty$, such that $\lambda(\tau_k) \not \to 1$, but since this sequence is bounded, there exists $\iota \neq \lambda_1^*(\bar{\rho})$ and a subsequence (which we relabel) such that
        \[
            \lambda(\tau_k) \to \iota \leq \kappa(\bar{\rho}), \qquad \mbox{as $k \to \infty$.}
        \]
        But then, by continuity, we have that
        \[
            0 = Q((\bar{\rho}_{\tau_k})_{\lambda(\tau_k)}) \longrightarrow Q(\bar{\rho}_\iota), \qquad \mbox{as $k \to \infty$.}
        \]
        But, this contradicts the uniqueness of $\lambda^*_1(\bar{\rho})$. Thus, we must have that $\lambda(\tau) \to 1 = \lambda^*_1(\bar{\rho})$ as $\tau \to 0$. Thus, $\tau \mapsto \lambda(\tau)$ is continuous for $0 \leq \tau < \delta$.         
        Hence, similar to the previous part, we have that $\tau \mapsto S_\mu((\bar{\rho}_\tau)_{\lambda(\tau)})$ is continuous for $0 \leq \tau < \delta$ and differentiable for $0 < \tau < \delta$. Since $\bar{\rho}$ is the minimiser of $S_\mu$ over $\mathcal{K}$ we can apply the mean value Theorem, to obtain for each $0<\tau<\delta$, there exists $\zeta_\tau \in (0,\tau)$, such that
        \[
            \begin{split}
                0 & \leq \frac{S_\mu((\bar{\rho}_{\tau})_{\lambda^*_1(\bar{\rho}_{\tau})}) - S_\mu(\bar{\rho}_{\lambda^*_1(\bar{\rho})})}{\tau} \\
                & = \frac{\d \lambda(\zeta_\tau)}{\d \tau}\frac{Q((\bar{\rho}_{\zeta_\tau})_{\lambda(\zeta_\tau)})}{2\lambda(\zeta_\tau)} + \frac{\lambda(\zeta_\tau)}{2} \int_{\mathbb{R}^3} \frac{w L(m_{\bar{\rho}_{\zeta_\tau}}(r)) + \bar{\rho}_\tau L'(m_{\bar{\rho}_{\zeta_\tau}}(r)) m_w(r)}{r^2} \dd \boldsymbol{x} \\
                & \qquad + \frac{a_0 \gamma \lambda(\zeta_\tau)^{\frac{3(\gamma -1)}{2}}}{\gamma - 1} \int_{\mathbb{R}^3} w \bar{\rho}_{\zeta_\tau}^{\gamma - 1} \dd \boldsymbol{x} + \lambda(\zeta_\tau)^\frac{\alpha}{2} \int_{\mathbb{R}^3} w \Phi_\alpha * \bar{\rho}_{\zeta_\tau} \dd \boldsymbol{x} + \mu \int_{\mathbb{R}^3} w \dd \boldsymbol{x}.
            \end{split}
        \]
        The first term on the right-hand side is 0, as $Q((\bar{\rho}_{\zeta_\tau})_{\lambda(\zeta_\tau)}) = 0$ by definition. Using a similar approach to the previous step, letting $\tau \to 0$, we obtain that
        \begin{equation}\label{3.21}
            0 \leq \frac{1}{2} \int_{\mathbb{R}^3} \frac{w L(m_{\bar{\rho}}(r)) + \bar{\rho} L'(m_{\bar{\rho}}(r)) m_w(r)}{r^2} \dd \boldsymbol{x} + \int_{\mathbb{R}^3} w \bigg ( \frac{a_0 \gamma}{\gamma - 1} \bar{\rho}^{\gamma - 1} + \Phi_\alpha * \bar{\rho} + \mu \bigg ) \dd \boldsymbol{x}.
        \end{equation}
        \end{case}
            \begin{case}
                There exists a sequence $\tau_k \to 0$, as $k \to \infty$, such that $Q((\bar{\rho}_{\tau_k})_{\kappa(\bar{\rho}_{\tau_k})}) \geq 0$. By continuity, we have that $Q((\bar{\rho}_{\tau_k})_{\kappa(\bar{\rho}_{\tau_k})}) \to 0$ as $k \to \infty$. Then, by Lemma \ref{nonempty2} and the fact that $\bar{\rho}$ is a minimiser, we have that for $k$ large enough
                $$S_{\mu}((\bar{\rho}_{\tau_k})_{\kappa(\bar{\rho}_{\tau_k})}) - S_\mu(\bar{\rho}) \geq 0.$$
                Thus, since $\tau \mapsto \kappa(\bar{\rho}_{\tau})$ is continuous and differentiable for all $\tau \geq 0$, using a similar approach to the previous case, we can apply L'H\^opital's rule, to obtain
                \[
                    \begin{split}
                        0 & \leq \lim_{k \to \infty} \frac{S_{\mu}((\bar{\rho}_{\tau_k})_{\kappa(\bar{\rho}_{\tau_k})}) - S_\mu(\bar{\rho})}{\tau_k} \\
                        & = \frac{\d \kappa(\bar{\rho}_{\tau})}{\d \tau} \bigg|_{\tau=0} \frac{Q(\bar{\rho})}{2} + \frac{1}{2} \int_{\mathbb{R}^3} \frac{w L(m_{\bar{\rho}}(r)) + \bar{\rho} L'(m_{\bar{\rho}}(r)) m_w(r)}{r^2} \dd \boldsymbol{x} \\
                        & \qquad + \frac{a_0 \gamma}{\gamma - 1} \int_{\mathbb{R}^3} w \bar{\rho}^{\gamma - 1} \dd \boldsymbol{x} + \int_{\mathbb{R}^3} w \Phi_\alpha * \bar{\rho} \dd \boldsymbol{x} + \mu \int_{\mathbb{R}^3} w \dd \boldsymbol{x}.
                    \end{split}
                \]
                Since $Q(\bar{\rho}) = 0$, this reduces to \eqref{3.21}.
            \end{case}
            \end{step}
            \setcounter{case}{0}
            \begin{step} Suppose $\alpha \in [2,3)$. Then, by the continuity of $\tau \mapsto Q(\rho_\tau)$, there exists $\delta > 0$ such that $\lambda^*(\bar{\rho}_\tau)$ exists for all $0 \leq \tau < \delta$. Thus, by a similar argument to Case 1 of the previous step, \eqref{3.21} is satisfied.
            \end{step}
        \begin{step}
        Since \eqref{3.21} is satisfied in all cases, by similar calculations to \eqref{3.22new} and \eqref{3.23new}, we obtain
        \[
            0 \leq \int_{\mathbb{R}^3} w \bigg ( \frac{a_0 \gamma}{\gamma - 1}  \bar{\rho}^{\gamma - 1} + \int^\infty_r\frac{L(m_{\bar{\rho}}(s))}{s^3} \dd s + \Phi_\alpha * \bar{\rho} + \mu \bigg ) \dd \boldsymbol{x}.
        \]
        Since this is true for all such $w$, we have that
        \begin{align*}
            & \frac{a_0\gamma}{\gamma - 1} \bar{\rho}^{\gamma - 1} + \int^\infty_r\frac{L(m_{\bar{\rho}}(s))}{s^3} \dd s + \Phi_\alpha * \bar{\rho} + \mu = 0, \qquad  \text{ {\it a.e.} on } \Gamma_\varepsilon,\\
            & \frac{a_0\gamma}{\gamma - 1} \bar{\rho}^{\gamma - 1} + \int^\infty_r\frac{L(m_{\bar{\rho}}(s))}{s^3} \dd s + \Phi_\alpha * \bar{\rho} + \mu \geq 0, \qquad \text{ {\it a.e.} on } \mathbb{R}^3_\varepsilon \setminus \Gamma_\varepsilon.
        \end{align*}
        Thus, letting $\varepsilon \to 0$, we obtain
    \[
    \begin{split}
        & \frac{a_0\gamma}{\gamma - 1} \bar{\rho}^{\gamma - 1} + \int^\infty_r\frac{L(m_{\bar{\rho}}(s))}{s^3} \dd s + \Phi_\alpha * \bar{\rho} + \mu = 0, \qquad  \text{ {\it a.e.} on } \Gamma,\\
        & \int^\infty_r\frac{L(m_{\bar{\rho}}(s))}{s^3} \dd s + \Phi_\alpha * \bar{\rho} + \mu \geq 0, \qquad \qquad \qquad \quad \,\,\,\, \text{ {\it a.e.} on } \mathbb{R}^3 \setminus \Gamma.
    \end{split}
    \]
    Similar to the proof of Theorem \ref{lmulaxi}, we have that $\bar{\rho} \in C(\mathbb{R}^3) \cap L^\infty(\mathbb{R}^3)$ and has compact support. Thus $\bar{\rho}$ is a rotating Riesz star solution as given by Definition \ref{rotatingriesz} with compact support.
    \end{step}
    \end{proof}
\setcounter{step}{0}

\section{Instability in the mass-supercritical regime}\label{growing}

We now show that, under suitable conditions on \(L\), there exist solutions of the CEREs \eqref{0.0} that start arbitrarily close to a rotating Riesz star solution and exhibit growth of their support. Since Theorem \ref{Existofrot} established that rotating Riesz star solutions are compactly supported, such behaviour implies instability of these steady states.

\smallskip

For $\alpha \in (0,2)$, we define the set
    \[
        \mathcal{I}_\mu := \{ (\varrho,\boldsymbol{\mathcal{M}}) \, | \, Q(\varrho) > 0, \, \kappa(\varrho) > 1, \, I_\mu(\varrho,\boldsymbol{\mathcal{M}}) < \ell_\mu \},
    \]
    while for $\alpha \in [2,3)$, we define
    \[
        \mathcal{I}_\mu := \{ (\varrho,\boldsymbol{\mathcal{M}}) \, | \, Q(\varrho) > 0, \, I_\mu(\varrho,\boldsymbol{\mathcal{M}}) < \ell_\mu \},
    \]
    where
    $$
        I_\mu(\varrho,\boldsymbol{\mathcal{M}}):= E(\varrho,\boldsymbol{\mathcal{M}}) + \mu \int_{\mathbb{R}^3} \varrho \dd \boldsymbol{x}.
    $$
    We will show that $\mathcal{I}_\mu$ is invariant under the semi-flow of the CEREs \eqref{0.0}.

    \begin{remark}
        In the regime $\alpha \in (0,2)$, the set would no longer remain invariant without the additional assumption that $\kappa(\varrho)>1$. Indeed, without this condition, it would be possible to have $Q(\rho(t)) = 0$ with $\lambda^*_2(\rho(t))=1$.
    \end{remark}

    \begin{lemma}[Invariance of $\mathcal{I}_\mu$ under the solution semi-flow]\label{invarianta}
Suppose $\alpha\in(0,3)$, $\frac{6}{6-\alpha}<\gamma<\frac{3+\alpha}{3}$, $L$ satisfies \eqref{Lnice2}, \eqref{Lstable2} and, for $\alpha \in (0,2)$, \eqref{Linstab} for $\omega > \omega_* + \bar{\omega}$ and $\omega \geq \frac{\gamma + 1 - \alpha}{3+\alpha-3\gamma}$. Then, if $(\rho,\boldsymbol{\mathcal{M}})$ is a classical solution of the attractive polytropic CEREs satisfying {\rm(I}$_1)$--{\rm(I}$_3)$ and {\rm(A}$_1)$--{\rm(A}$_4)$ with the square of the angular momentum given by $L(m_\rho(r(\boldsymbol{x})))$ and initial data $(\rho_0,\boldsymbol{\mathcal{M}}_0)\in \mathcal{I}_\mu$. Then
\[
(\rho(t),\boldsymbol{\mathcal{M}}(t))\in \mathcal{I}_\mu, \qquad \text{for all } t>0.
\]
\end{lemma}

\begin{proof}
For $\alpha \in [2,3)$, the proof is similar to \cite[Lemma 4.4.]{Carrillo_2026}. So we now assume that $\alpha \in (0,2)$. By conservation of mass and energy,
\[
S_\mu(\rho(t)) \leq I_\mu(\rho(t),\boldsymbol{\mathcal{M}}(t)) = I_\mu(\rho_0,\boldsymbol{\mathcal{M}}_0) < \ell_\mu,
\quad \text{for all } t >0.
\]
Suppose for a contradiction the statement is false, then by continuity, there exists a first time $t_0 > 0$ such that $(\rho(t_0),\boldsymbol{\mathcal{M}}(t_0))\not\in \mathcal{I}_\mu$.
\begin{case}
Suppose $Q(\rho(t_0)) = 0$ and $\kappa(\rho(t_0))\geq 1$. Then, by Lemma \ref{smallalpha}, $\rho(t_0) \in \mathcal{K}$. By definition of $\ell_\mu$, we have
\[
S_\mu(\rho(t_0))\ge \ell_\mu,
\]
which is a contradiction. 
\end{case}
\begin{case}
    Suppose $Q(\rho(t_0)) > 0$ and $\kappa(\rho(t_0)) = 1$. By Lemma \ref{nonempty2}, we have
    \[
        S_\mu(\rho(t_0)) > \ell_\mu,
    \]
    which is a contradiction.
\end{case}
\end{proof}
\setcounter{case}{0}

We next analyse the concavity properties of the mapping $\lambda \mapsto S_{\mu,\lambda}(\varrho)$, which play a fundamental role in establishing instability of rotating Riesz star solutions. Although the quantity $\kappa(\varrho)$ was originally introduced for $\alpha \in (0,2)$, we note that its defining formula remains well defined for $\alpha \in (2,3)$ provided that $\gamma > \frac{5}{3}$.

\begin{lemma}\label{concave}
    Suppose $\alpha \in (0,3)$, $\frac{6}{6-\alpha} < \gamma < \frac{3+\alpha}{3}$, $L$ satisfies \eqref{Lnice2} and $\varrho \in L^1_+(\mathbb{R}^3) \cap L^\gamma(\mathbb{R}^3)$ such that $\varrho \not\equiv 0$ and 
    $$\int_{\mathbb{R}^3} \frac{\varrho L(m_{\varrho}(r))}{r^2} \dd \boldsymbol{x} < \infty,$$
    then $\lambda \mapsto S_{\mu,\lambda}(\varrho)$ is concave for
    \[
    \begin{cases}
        0<\lambda \leq \kappa(\varrho), & \text{ if } \alpha \in (0,2), \\
        \lambda > 0, & \text{ if } \alpha \in [2,3) \text{ and } \gamma \leq \frac{5}{3}, \\
        \kappa(\varrho) \leq \lambda < \infty, & \text{ if } \alpha \in (2,3) \text{ and } \gamma > \frac{5}{3}.
    \end{cases}
    \]
\end{lemma}

    \begin{proof}
    We see that from \eqref{secondderi}
    \[
    \frac{\d^2 S_{\mu,\lambda}(\varrho)}{\d \lambda^2} = \frac{3 a_0 (3\gamma - 5) \lambda^{\frac{3(\gamma - 1) - 4}{2}}}{4} \int_{\mathbb{R}^3} \varrho^\gamma \dd \boldsymbol{x} + \frac{\alpha (\alpha - 2) \lambda^{\frac{\alpha - 4}{2}}}{8} \int_{\mathbb{R}^3} \varrho \Phi_\alpha * \varrho \dd \boldsymbol{x}.
    \]
    \begin{case}[$\alpha \in (0,2)$]
        Since $\gamma < \frac{3+\alpha}{3}$, we see that $\frac{\d^2 S_{\mu,\lambda}(\varrho)}{\d \lambda^2} \leq 0$ if and only if $0<\lambda \leq \kappa(\varrho)$ with equality for $\lambda = \kappa(\varrho)$.
    \end{case}
    \begin{case}[$\alpha \in [2,3)$ and $\gamma \leq \frac{5}{3}$]
        Then $\frac{\d^2 S_{\mu,\lambda}(\varrho)}{\d \lambda^2} < 0$ for all $\lambda > 0$.
    \end{case}
    \begin{case}[$\alpha \in [2,3)$ and $\gamma > \frac{5}{3}$]
        Then the signs of the coefficients are flipped as compared to Case 1. Rearranging, we see that $\frac{\d^2 S_{\mu,\lambda}(\varrho)}{\d \lambda^2} \leq 0$ if and only if $\kappa(\varrho) \leq \lambda < \infty$ with equality for $\lambda = \kappa(\varrho)$.
    \end{case}
\end{proof}
\setcounter{case}{0}

\begin{theorem}\label{smallgamma2}
    Suppose $\alpha\in(0,3)$, $\frac{6}{6-\alpha}<\gamma<\frac{3+\alpha}{3}$, $L$ satisfies \eqref{Lnice2}, \eqref{Lstable2} and, for $\alpha \in (0,2)$, \eqref{Linstab} for $\omega > \omega_* + \bar{\omega}$ and $\omega \geq \frac{\gamma + 1 - \alpha}{3+\alpha-3\gamma}$. If $(\rho,\boldsymbol{\mathcal{M}})$ is a classical solution of the attractive polytropic CEREs satisfying {\rm(I}$_1)$--{\rm(I}$_3)$ and {\rm(A}$_1)$--{\rm(A}$_4)$ with the square of the angular momentum given by $L(m_\rho(r(\boldsymbol{x})))$ with initial conditions $(\rho_0,\boldsymbol{\mathcal{M}}_0) \in \mathcal{I}_\mu$ for some $\mu > 0$. Then, for $R(t)=\max_{\boldsymbol{x}\in\Omega(t)}\{|\boldsymbol{x}|\}$, there exists a sequence $t_m \to \infty$ as $m \to \infty$ such that
$$
R(t_m) \to \infty, \qquad \mbox{as $m \to \infty$}.
$$
\end{theorem}

\begin{remark}
        Let $\bar{\rho}$ be any minimiser of $S_\mu$ over $\mathcal{K}$. Then, for every $0<\lambda<1$, we have $(\bar{\rho}_\lambda,\boldsymbol{0}) \in \mathcal{I}_\mu$. Since Theorem \ref{Existofrot} established that $\bar{\rho}$ is a compactly supported rotating Riesz star solution, Theorem \ref{smallgamma2} therefore implies instability of rotating Riesz star solutions.
    \end{remark}

\begin{proof} 
We split the proof into five steps.
\begin{step}
    Suppose there does not exist $\Lambda > 0$ such that $Q(\rho(t)) \geq \Lambda$ for all $t>0$. Thus, there exists $t_0 \in (0,\infty]$ and a sequence $t_m \to t_0$ as $m \to \infty$ such that 
$$
Q(\rho(t))> 0 \,\,\,\,\mbox{ for all }t \in [0,t_0), \qquad\,\,\, 
\lim_{m \to \infty} Q(\rho(t_m)) = 0.
$$
 By Lemma \ref{invarianta}, we have that $Q(\rho(t))>0$, $S_\mu(\rho(t)) < \ell_\mu$ and, for $\alpha \in (0,2)$, $\kappa(\rho(t)) > 1$ for all $t>0$. Thus, since $t \mapsto Q(\rho(t))$ is continuous, $t_0 =  \infty$.
\end{step}
\begin{step}
    We first consider when $\alpha \in [2,3)$. 
    \begin{case}
        Suppose $\alpha \in [2,3)$ with $\gamma \leq \frac{5}{3}$. By Lemma \ref{concave}, and a similar proof to that of \cite[Lemma 4.12.]{Carrillo_2026}, there exists a subsequence of $(t_m)_m$ (which we relabel as $(t_m)_m$) such that $R(t_m) \to \infty$ as $m \to \infty$.
    \end{case} 
    \begin{case}
        Suppose $\alpha \in (2,3)$ with $\gamma > \frac{5}{3}$. If
        \begin{equation*}
        \liminf_{m \to \infty} \frac{Q(\rho(t_m))}{-\int_{\mathbb R^3}\rho(t_m)\Phi_\alpha*\rho(t_m)\dd \boldsymbol{x}} > 0,
        \end{equation*}
        owing to the positivity of $L$, as in Case 1 of the proof of \cite[Lemma 4.12.]{Carrillo_2026}, we have that there exists a subsequence of $(t_m)_m$ (which we relabel as $(t_m)_m$) such that $R(t_m) \to \infty$ as $m \to \infty$. If 
        \begin{equation*}
        \liminf_{m \to \infty} \frac{Q(\rho(t_m))}{-\int_{\mathbb R^3}\rho(t_m)\Phi_\alpha*\rho(t_m)\dd \boldsymbol{x}} = 0,
        \end{equation*}
        then there exists a subsequence of $(t_m)_m$ (which we relabel as $(t_m)_m$), such that
    \begin{equation*}
        \lim_{m \to \infty} \frac{Q(\rho(t_m))}{-\int_{\mathbb R^3}\rho(t_m)\Phi_\alpha*\rho(t_m)\dd \boldsymbol{x}} = 0.
    \end{equation*}
        Using the definition of $Q(\rho)$ and the positivity of $L$, we can rewrite $\kappa(\rho(t))$ as
            \begin{align*}
                \kappa(\rho) & = \bigg( \frac{2 (5 - 3\gamma )\big(Q(\rho) - \int_{\mathbb{R}^3} \rho L(m_{\rho}(r))r^{-2} \dd \boldsymbol{x} - \frac{\alpha}{2} \int_{\mathbb{R}^3} \rho \Phi_\alpha * \rho \dd \boldsymbol{x}\big)}{\alpha (\alpha-2) \int_{\mathbb{R}^3} \rho \Phi_\alpha * \rho \dd \boldsymbol{x}} \bigg)^\frac{2}{3 +\alpha - 3\gamma} \\
                & \leq \bigg( \frac{3\gamma - 5}{\alpha - 2} + \frac{2 (5 - 3\gamma )Q(\rho)}{\alpha (\alpha-2) \int_{\mathbb{R}^3} \rho \Phi_\alpha * \rho \dd \boldsymbol{x}} \bigg)^\frac{2}{3 +\alpha - 3\gamma}.
            \end{align*}
        Thus, combining this with the fact that $\frac{3\gamma - 5}{\alpha - 2} < 1$ since $\gamma < \frac{3+\alpha}{3}$, we obtain that $\lim_{m \to \infty}\kappa(\rho(t_m)) < 1$. Thus, for large enough $m$, employing Lemmas \ref{largealpha} (b) and Lemma \ref{concave}, we have that $\lambda \mapsto S_\mu(\rho(t_m))$ is concave on $[1,\lambda^*(\rho(t_m))]$. By a similar proof to that of Case 2 of the proof of \cite[Lemma 4.12.]{Carrillo_2026}, we obtain a contradiction. Thus, there exists a sequence $t_m \to \infty$ as $m \to \infty$ such that $R(t_m) \to \infty$ as $m \to \infty$.
    \end{case}
\end{step}    
\setcounter{case}{0}
    \begin{step}
        Suppose that $\alpha \in (0,2)$. We show that $\kappa(\rho(t_m)) \to \infty$ as $m \to \infty$. By Lemma \ref{concave}, we have that $\lambda \mapsto S_{\mu,\lambda}(\rho(t))$ is concave for all $0 < \lambda \leq \kappa(\rho(t))$.
        \begin{case}
            Suppose that $Q(\rho(t)_{\lambda}) > 0$ for all $\lambda > 0$. Then, since $Q(\rho(t)_{\kappa(\rho(t))}) > 0$ and $\kappa(\rho(t)_{\kappa(\rho(t))}) = 1$, by Lemma \ref{nonempty2} and Taylor's Theorem we have that
            \begin{equation}\label{>0}
                \ell_\mu \leq S_\mu(\rho(t)_{\kappa(\rho(t))}) \leq S_\mu(\rho(t)) + \frac{Q(\rho(t))}{2}(\kappa(\rho(t)) - 1).
            \end{equation}
        \end{case}
        \begin{case}
            Suppose there exists $\lambda > 0$ such that $Q(\rho(t)_{\lambda}) \leq 0$, then by continuity and Lemma \ref{smallalpha} (d), we have that $\lambda^*_1(\rho(t)) \leq \kappa(\rho(t))$. Thus, since $\rho(t)_{\lambda^*_1(\rho(t))} \in \mathcal{K}$, by Taylor's Theorem, we have that
            \begin{equation}\label{=0}
            \begin{split}
                \ell_\mu \leq S_\mu(\rho(t)_{\lambda^*_1(\rho(t))}) & \leq S_\mu(\rho(t)) + \frac{Q(\rho(t))}{2}(\lambda^*_1(\rho(t)) - 1) \\
                & \leq S_\mu(\rho(t)) + \frac{Q(\rho(t))}{2}(\kappa(\rho(t)) - 1).
            \end{split}
            \end{equation}
        \end{case}
        Thus, combining and rearranging \eqref{>0} and \eqref{=0}, we obtain that
        \begin{equation*}
            \kappa(\rho(t)) \geq \frac{2(\ell_\mu - S_\mu(\rho(t)))}{Q(\rho(t))} \geq \frac{2(\ell_\mu - I_\mu(\rho_0,\boldsymbol{\mathcal{M}}_0))}{Q(\rho(t))},
        \end{equation*}
        for all $t > 0$. Thus, since $Q(\rho(t_m)) \to 0$ as $m \to \infty$, we have that $\kappa(\rho(t_m)) \to \infty$ as $m \to \infty$.
    \end{step}
    \setcounter{case}{0}
    \begin{step}
        We now show that $R(t_m) \to \infty$ as $m \to \infty$.
        \begin{case}
            Suppose there exists $C>0$ such that $\int_{\mathbb{R}^3} \rho(t_m)^\gamma \dd \boldsymbol{x} \leq C$ for all $m$. Then, since 
            $$\kappa(\rho(t_m)) = \bigg( \frac{6 a_0 (2 - 3(\gamma - 1))\int_{\mathbb{R}^3}\rho^\gamma(t_m) \dd \boldsymbol{x}}{\alpha (\alpha-2) \int_{\mathbb{R}^3} (\rho \Phi_\alpha * \rho)(t_m) \dd \boldsymbol{x}} \bigg)^\frac{2}{3 + \alpha - 3\gamma} \longrightarrow \infty, \qquad \mbox{as $m \to \infty$,}$$
            we have that $\int_{\mathbb{R}^3} (\rho \Phi_\alpha * \rho)(t_m) \dd \boldsymbol{x} \to 0$ as $m \to \infty$. Thus, since the kinetic energy and internal energy are positive and $Q(\rho(t_m)) \to 0$ as $m \to \infty$, we must have that $\int_{\mathbb{R}^3}\rho^\gamma(t_m) \dd \boldsymbol{x} \to 0$ as $m \to \infty$. If there exists a subsequence of $(t_m)_m$ (which we relabel as $(t_m)_m$) such that $R(t_m)$ is bounded, then there exists $D>0$ such that $R(t_m) \leq D$ for all $m$. Thus, by H\"older's inequality, we have that
            \[
                M = \| \rho(t_m)\|_{L^1(B_D)} \leq \Big (\frac{2\pi D^3}{3}\Big)^\frac{\gamma - 1} {\gamma} \|\rho(t_m)\|_{L^\gamma(\mathbb{R}^3)} \longrightarrow 0, \qquad \mbox{as $m \to \infty$.}
            \]
            This contradicts conservation of mass, so we obtain that $R(t_m) \to \infty$ as $m \to \infty$.
        \end{case}
        \begin{case}
            Suppose there exists a subsequence of $(t_m)_m$ (which we relabel as $(t_m)_m$) such that $\int_{\mathbb{R}^3} \rho(t_m)^\gamma \dd \boldsymbol{x} \to \infty$ as $m \to \infty$. Then, since 
            $$\kappa(\rho(t_m))  \longrightarrow \infty, \qquad \mbox{as $m \to \infty$,}$$
            we have that $Q(\rho(t_m)) \to \infty$ as $m \to \infty$. This is a contradiction, so we must be in Case 1.
        \end{case}
    \end{step}
    \begin{step}
        We now suppose that there exists $\Lambda > 0$ such that $Q(\rho(t)) \geq \Lambda$ for all $t>0$. Consider the second moment of the density,
        \[
            H(t) = \int_{\Omega(t)} \rho(t) |\boldsymbol{x}|^2 \dd \boldsymbol{x}.
        \]
        Then, using $\eqref{0.0}_1$, integration by parts and $\rho |_{\partial \Omega(t)} \equiv 0$, we see that
 $$H'(t)= \int_{\partial \Omega(t)} |\boldsymbol{x}|^2 \boldsymbol{\mathcal{M}} \cdot \nu \dd S(\boldsymbol{x}) + \int_{\Omega(t)}\rho_t|\boldsymbol{x}|^2\dd \boldsymbol{x} = 2\int_{\Omega(t)}\boldsymbol{\mathcal{M}} \cdot \boldsymbol{x}\dd \boldsymbol{x},$$
where $\nu(\boldsymbol{x})$ is the normal to the surface $\partial \Omega(t)$ at $\boldsymbol{x} \in \partial \Omega(t)$. Note, that similarly, using $\eqref{0.0}_2$, that
 \begin{equation}\label{2.2}
 \begin{split}
 H''(t)& =2\int_{\partial \Omega(t)} \bigg( \frac{\boldsymbol{\mathcal{M}} \otimes \boldsymbol{\mathcal{M}}}{\rho} : \boldsymbol{x}\otimes \nu \bigg) \dd S(\boldsymbol{x}) + 2\int_{\Omega(t)}\boldsymbol{\mathcal{M}}_t \cdot \boldsymbol{x}\dd \boldsymbol{x} \\
 & = 2\Big(\int_{\Omega(t)}\bigg | \frac{\boldsymbol{\mathcal{M}}}{\sqrt{\rho}} \bigg |^2\dd \boldsymbol{x}+3 a_0\int_{\Omega(t)} \rho^\gamma \dd \boldsymbol{x}-\int_{\Omega(t)}\rho\nabla\Phi_\alpha\ast\rho\cdot \boldsymbol{x}\dd \boldsymbol{x}\Big).
 \end{split}
 \end{equation}
A simple calculation shows that
 \begin{align*}\nonumber
 I&:=\int_{\Omega(t)}\rho\nabla\Phi_\alpha\ast\rho\cdot \boldsymbol{x}\dd \boldsymbol{x}\nonumber\\
 &=\int_{\Omega(t)}\rho(\boldsymbol{x})\int_{\Omega(t)}\frac{\rho(\boldsymbol{y})(\boldsymbol{x}-\boldsymbol{y})\cdot \boldsymbol{y}}{|\boldsymbol{x}-\boldsymbol{y}|^{\alpha+2}}\dd \boldsymbol{y}\dd \boldsymbol{x}\nonumber\\
 &= - \int_{\Omega(t)}\rho(\boldsymbol{x})\int_{\Omega(t)}\frac{\rho(\boldsymbol{y})|\boldsymbol{x}-\boldsymbol{y}|^2}{|\boldsymbol{x}-\boldsymbol{y}|^{\alpha+2}}\dd \boldsymbol{y}\dd \boldsymbol{x}+\int_{\Omega(t)}\rho(\boldsymbol{x})\int_{\Omega(t)}\frac{\rho(\boldsymbol{y})(\boldsymbol{x}-\boldsymbol{y})\cdot \boldsymbol{x}}{|\boldsymbol{x}-\boldsymbol{y}|^{\alpha+2}}\dd \boldsymbol{y}\dd \boldsymbol{x}\nonumber\\
 &= - \int_{\Omega(t)}\rho(\boldsymbol{x})\int_{\Omega(t)}\frac{\rho(\boldsymbol{y})|\boldsymbol{x}-\boldsymbol{y}|^2}{|\boldsymbol{x}-\boldsymbol{y}|^{\alpha+2}}\dd \boldsymbol{y}\dd \boldsymbol{x}-\int_{\Omega(t)}\rho(\boldsymbol{y})\int_{\Omega(t)}\frac{\rho(\boldsymbol{x})(\boldsymbol{y}-\boldsymbol{x})\cdot \boldsymbol{x}}{|\boldsymbol{x}-\boldsymbol{y}|^{\alpha+2}}\dd \boldsymbol{x}\dd \boldsymbol{y}\nonumber\\
 &=-\alpha\int_{\Omega(t)}\rho\Phi_\alpha\ast\rho\dd \boldsymbol{x}-I.
 \end{align*}
 Then, we have 
 \begin{equation}\label{Iidentity}
     I=-\frac{\alpha}{2}\int_{\Omega(t)}\rho\Phi_\alpha\ast\rho\dd \boldsymbol{x}.
 \end{equation}
        Combining \eqref{2.2} and \eqref{Iidentity}, we obtain
        \[
            \begin{split}
                H''(t) & = 2\Big(\int_{\Omega(t)}\bigg | \frac{\boldsymbol{\mathcal{M}}}{\sqrt{\rho}} \bigg |^2\dd \boldsymbol{x}+3 a_0\int_{\Omega(t)}\rho^\gamma\dd \boldsymbol{x} + \frac{\alpha}{2}\int_{\Omega(t)}\rho\Phi_\alpha\ast\rho\dd \boldsymbol{x}\Big)(t) \\
                & = 2Q(\rho(t)) + 2 \int_{\mathbb{R}^3} \bigg ( \bigg|\frac{\mathcal{M}^r}{\sqrt{\rho}}\bigg|^2 + \bigg|\frac{\mathcal{M}^3}{\sqrt{\rho}}\bigg|^2 \bigg)(t)\,\dd\boldsymbol{x} \\
                & \geq 2\Lambda.
            \end{split}
        \]
        Thus, by Taylor's Theorem, we see that
        \begin{equation}\label{HTaylor}
            H(t) \geq H(0) + H'(0)t + \Lambda t^2.
        \end{equation}
        Equally, we see that by H\"older's inequality, that
        \begin{equation}\label{HHolder}
            H(t) = \int_{\Omega(t)} \rho(t) |\boldsymbol{x}|^2 \dd \boldsymbol{x} \leq R(t)^2 M.
        \end{equation}
        Thus, combining \eqref{HTaylor} and \eqref{HHolder} and rearranging, we obtain that
        \[
            \frac{R(t)}{t} \geq \Big ( \frac{\Lambda}{M} \Big )^\frac{1}{2} + O(t^{-\frac{1}{2}}).
        \]
        Thus, we see that $R(t) \to \infty$ as $t \to \infty$.
    \end{step}
\end{proof}
\setcounter{step}{0}
\setcounter{case}{0}

\appendix

\section{Convolution inequalities}
We now present several convolution inequalities that have been used in the main text of the thesis.

\begin{lemma}[Hardy-Littlewood-Sobolev (HLS) Inequality]\label{simplehls}
	Let $\alpha \in (0,n)$ and $1 < p,r < \infty$ such that
	\[
	\frac{1}{p} + \frac{\alpha}{n} = 1 + \frac{1}{r}.
	\]
	Then, there exists a constant $C > 0$ such that 
    for all $f \in L^p(\mathbb{R}^n)$,
	\[
	\big\| |\cdot |^{-\alpha} * f \big\|_{L^r(\mathbb{R}^n)} \leq C\|f\|_{L^p(\mathbb{R}^n)}.
	\]
    When $r = \frac{2n}{\alpha}$ and $p=\frac{2n}{2n-\alpha}$, the constant is given by $C_{n,\alpha}$, where
    \[
    C_{n,\alpha} := \pi^\frac{\alpha}{2} \frac{\Gamma(\frac{n - \alpha}{2})}{\Gamma( n - \frac{\alpha}{2})} 
    \bigg( \frac{\Gamma(\frac{n}{2})}{\Gamma(n)} \bigg)^{\frac{\alpha - n}{n}}.
    \]
    Moreover, if $f \in L^1(\mathbb{R}^n) \cap L^p(\mathbb{R}^n)$ for $p > \frac{n}{n-\alpha}$, then $|\cdot |^{-\alpha} * f \in C(\mathbb{R}^n) \cap L^\infty(\mathbb{R}^n)$.
\end{lemma}

\begin{lemma}[Variation of HLS Inequality \text{\cite[Theorem 3.1]{Calvez_2017}}]\label{HLSineq}
Suppose that $\alpha \in (0,n)$ and $p \geq \frac{2n}{2n - \alpha}$.
Then there exists $C_* = C_*(\alpha,p,n) > 0$ such that, for any $f \in L^1(\mathbb{R}^n) \cap L^p(\mathbb{R}^n)$,
\begin{equation}\label{HLS}
\bigg| \int_{\mathbb{R}^n} \int_{\mathbb{R}^n} f(\boldsymbol{x}) |\boldsymbol{x} - \boldsymbol{y}|^{-\alpha} f(\boldsymbol{y}) \dd \boldsymbol{x} \dd \boldsymbol{y} \bigg| \leq C_* \| f \|_{L^1}^\frac{p (2n - \alpha) - 2n}{n(p - 1)} \| f \|_{L^p}^\frac{\alpha p}{n(p - 1)}.
\end{equation}
In particular, when $p = \frac{n + \alpha}{n} > \frac{2n}{2n - \alpha}$, for any
$f \in L^1(\mathbb{R}^n) \cap L^{\frac{n + \alpha}{n}}(\mathbb{R}^n)$,
\[
\bigg| \int_{\mathbb{R}^n} \int_{\mathbb{R}^n} f(\boldsymbol{x}) |\boldsymbol{x} - \boldsymbol{y}|^{-\alpha} f(\boldsymbol{y}) \dd \boldsymbol{x} \dd \boldsymbol{y} \bigg|
\leq C_* \| f \|_{L^1}^\frac{n - \alpha}{n} \| f \|_{L^{\frac{n + \alpha}{n}}}^\frac{n + \alpha}{n}.
\]
Furthermore, $C_*$ is the sharp constant that satisfies \eqref{HLS} and is strictly bounded above by the standard HLS constant $C_{n,\alpha}$, i.e., $C_* < C_{n,\alpha}$.
\end{lemma}

\begin{lemma}[Weak Young's Convolution Inequality]\label{precised}
	Let $1 < p < \infty$ and $q \geq \frac{2p}{2p-1}$. Then, for $h \in L^p_{\rm w}(\mathbb{R}^n)$, there exists $C_*^h> 0$ such that for any $f,g \in L^1(\mathbb{R}^n) \cap L^q(\mathbb{R}^n)$,
	\[
	\bigg| \int_{\mathbb{R}^n} \int_{\mathbb{R}^n} f(\boldsymbol{x}) h(\boldsymbol{x} - \boldsymbol{y}) g(\boldsymbol{y}) \dd \boldsymbol{x} \dd \boldsymbol{y} \bigg| \leq C_*^h (\|f\|_{L^1}\|g\|_{L^1})^\frac{q(2p-1)-2p}{2p(q-1)} (\|f\|_{L^q}\|g\|_{L^q})^\frac{q}{2p(q-1)}.
	\]
\end{lemma}

\begin{lemma}[Riesz's Rearrangement Inequality \cite{Lieb_2001}]\label{rearrangement}
    Let $f, g, h : \mathbb{R}^n \to \mathbb{R}_+$ be measurable functions. Define
    \[
    I(f,g,h) \coloneq \int_{\mathbb{R}^n} \int_{\mathbb{R}^n} f(\boldsymbol{x}) g(\boldsymbol{x} - \boldsymbol{y}) h(\boldsymbol{y}) \dd \boldsymbol{x} \dd \boldsymbol{y}.
    \]
    For a given measurable function $\phi : \mathbb{R}^n \to \mathbb{R}_+$, define the radially decreasing rearrangement $\phi^{\#}$ of $\phi$ as
    \begin{equation*}
    \phi^{\#} (\boldsymbol{x}) \coloneq \int^\infty_0 \mathds{1}_{\{\boldsymbol{y} \, | \, \phi(\boldsymbol{y}) > t\}^{\#}} (\boldsymbol{x}) \dd t,
    \end{equation*}
    with
    \[
        A^{\#} = \Big \{ \boldsymbol{x} \in \mathbb{R}^n \, \Big | \, \frac{\omega_n}{n} |\boldsymbol{x}|^n \leq |A|\Big\},
    \]
    for any Lebesgue measurable set $A \subset \mathbb{R}^n$. Then
    \begin{equation}\label{rearrangeineq}
    I(f,g,h) \leq I(f^{\#},g^{\#},h^{\#}).
    \end{equation}
    If $g$ is a strictly radially decreasing function, then the equality holds in \eqref{rearrangeineq}
    if and only if there exists some $\boldsymbol{z} \in \mathbb{R}^n$ such that $f(\boldsymbol{x}) = f^{\#}(\boldsymbol{x} - \boldsymbol{z})$ and $h(\boldsymbol{x}) = h^{\#}(\boldsymbol{x} - \boldsymbol{z})$.
\end{lemma}

\bigskip
\bigskip
\noindent{\bf Acknowledgments.}
This work was carried out while the author was a doctoral student at the Mathematical Institute, University of Oxford, and was supported by the Mathematical Institute, University of Oxford. For the purpose of open access, the author has applied a CC BY public copyright license to any Author Accepted Manuscript (AAM) version
arising from this submission.

\bigskip
\medskip
\noindent{\bf Conflict of Interest:} The author declares that they have no conflict of interest.
The author also declares that this manuscript has not been previously published,
and will not be submitted elsewhere before your decision.

\bigskip
\noindent{\bf Data availability:} Data sharing is not applicable to this article as no datasets were generated or analysed during the current study.

\printbibliography

\end{document}